\documentclass[12pt, usenames,dvipsnames]{amsart}
\usepackage{amssymb,amsmath,amsfonts,amscd,euscript,mathrsfs,bm,bbm}
\usepackage{lineno}

\usepackage{amsthm}
\usepackage{amsopn} 
\usepackage{verbatim} 
\usepackage{mathtools}
\usepackage{enumitem}

\usepackage{fullpage}
\usepackage[backref=page]{hyperref}

\DeclareMathAlphabet{\mathpzc}{OT1}{pzc}{m}{it}
\usepackage[sans]{dsfont}

\usepackage{tikz}
\usepackage{tikz-cd}
\usetikzlibrary{babel}
\usetikzlibrary{matrix}
\usetikzlibrary{shapes}
\usetikzlibrary{arrows}
\usetikzlibrary{calc,3d}
\usetikzlibrary{decorations,decorations.pathmorphing}
\usetikzlibrary{through}
\tikzset{ext/.style={circle, draw,inner sep=1pt},int/.style={circle,draw,fill,inner sep=1pt},nil/.style={inner sep=1pt}}
\tikzset{exte/.style={circle, draw,inner sep=3pt},inte/.style={circle,draw,fill,inner sep=3pt}}
\tikzset{diagram/.style={matrix of math nodes, row sep=3em, column sep=2.5em, text height=1.5ex, text depth=0.25ex}}
\tikzset{diagram2/.style={matrix of math nodes, row sep=0.5em, column sep=0.5em, text height=1.5ex, text depth=0.25ex}}
\tikzset{every picture/.append style={baseline=-.65ex}}

\usepackage{todonotes}

\usepackage{hyperref}

\usepackage{color}

\let\le\leqslant
\let\ge\geqslant
\let\leq\leqslant
\let\geq\geqslant

\usetikzlibrary{calc}

\newcommand{\composition}[2]{%
    \begin{tikzpicture}[x=0.8cm, y=0.8cm, scale=0.45]

        \xdef\Hmax{0}
        \foreach \v in {#2}{%
            \pgfmathparse{max(\Hmax,\v)}%
            \xdef\Hmax{\pgfmathresult}%
        }
        \pgfmathtruncatemacro{\HmaxInt}{\Hmax}

        \xdef\col{0}
        \foreach \v in {#2}{%
            \pgfmathtruncatemacro{\vi}{\v}%

            \ifnum\vi>0
                \foreach \row in {1,...,\vi}{%
                    \fill[gray!10] (\col, \row-1) rectangle (\col+1, \row);
                    \draw[gray!50, thin] (\col, \row-1) rectangle (\col+1, \row);
                }
            \else
                \draw[gray!50, thin] (\col, 0) -- (\col+1, 0);
            \fi

            \node[below, font=\small] at (\col+0.5, 0) {\vi};

            \pgfmathparse{\col+1}%
            \xdef\col{\pgfmathresult}%
        }

    \end{tikzpicture}%
}

\newcommand{\glueComposition}[3]{%
    \begin{tikzpicture}[x=0.8cm, y=0.8cm,scale=0.45]

        \xdef\tmpidx{0}
        \foreach \va in {#2}{%
            \expandafter\xdef\csname Bot\tmpidx\endcsname{\va}%
            \pgfmathparse{int(\tmpidx+1)}\xdef\tmpidx{\pgfmathresult}%
        }
        \xdef\tmpidx{0}
        \foreach \vb in {#3}{%
            \expandafter\xdef\csname Top\tmpidx\endcsname{\vb}%
            \pgfmathparse{int(\tmpidx+1)}\xdef\tmpidx{\pgfmathresult}%
        }

        \xdef\Hmax{0}
        \pgfmathtruncatemacro{\Nmax}{#1-1}
        \foreach \i in {0,...,\Nmax}{%
            \expandafter\xdef\expandafter\va\expandafter{\csname Bot\i\endcsname}%
            \expandafter\xdef\expandafter\vb\expandafter{\csname Top\i\endcsname}%
            \pgfmathparse{max(\Hmax, \va+\vb)}%
            \xdef\Hmax{\pgfmathresult}%
        }
        \pgfmathtruncatemacro{\HmaxInt}{\Hmax}

        \foreach \i in {0,...,\Nmax}{%
            \expandafter\xdef\expandafter\va\expandafter{\csname Bot\i\endcsname}%
            \expandafter\xdef\expandafter\vb\expandafter{\csname Top\i\endcsname}%
            \pgfmathtruncatemacro{\vai}{\va}%
            \pgfmathtruncatemacro{\vbi}{\vb}%

            \ifnum\vai>0
                \foreach \row in {1,...,\vai}{%
                    \fill[gray!10] (\i, \row-1) rectangle (\i+1, \row);
                    \draw[gray!50, thin] (\i, \row-1) rectangle (\i+1, \row);
                }%
            \fi

            \ifnum\vbi>0
                \foreach \row in {1,...,\vbi}{%
                    \pgfmathtruncatemacro{\rowshift}{\row+\vai}%
                    \fill[black!80] (\i, \rowshift-1) rectangle (\i+1, \rowshift);
                    \draw[gray!60, thin] (\i, \rowshift-1) rectangle (\i+1, \rowshift);
                }%
            \fi

            \ifnum\vai=0
                \ifnum\vbi=0
                    \draw[gray!50, thin] (\i, 0) -- (\i+1, 0);
                \fi
            \fi

            \pgfmathtruncatemacro{\vtotal}{\vai+\vbi}%
            \node[below, font=\small] at (\i+0.5, 0) {$\vtotal$};
        }

    \end{tikzpicture}%
}

\newcommand{\nc}{\newcommand}
\newtheorem{theoremA}{Theorem}

\newtheorem{corA}[theoremA]{Corollary}

\numberwithin{equation}{section}

\newtheorem*{conj*}{Conjecture}
\newtheorem{thm}[equation]{Theorem}
\newtheorem{prop}[equation]{Proposition}
\newtheorem{lem}[equation]{Lemma}
\newtheorem{cor}[equation]{Corollary}
\newtheorem*{cor*}{Corollary}
\newtheorem{rem}[equation]{Remark}
\newtheorem{prop::def}[equation]{Proposition-Definition}

\newtheorem{dfn}[equation]{Definition}
\newtheorem{definition}[equation]{{\bf Definition}}
\newtheorem{notation}[equation]{{\bf Notation}}
\newtheorem{example}[equation]{Example}

\newtheorem{conj}[equation]{{\bf Conjecture}}

\newcommand{\calC}{\mathcal{C}}

\nc{\fB}{\mathfrak{B}}
\nc{\gl}{\mathfrak{gl}}
\nc{\GL}{\mathsf{GL}}
\nc{\g}{\mathfrak{g}}
\nc{\gh}{\widehat\g}
\nc{\h}{\mathfrak{h}}
\nc{\wfh}{\widehat{\mathfrak{h}}}
\nc{\la}{\lambda}
\nc{\al}{\alpha }
\nc{\be}{\beta }
\nc{\ve}{\varepsilon }
\nc{\om}{\omega }
\nc{\lr}{\text{-}}
\nc{\ta}{\theta}
\nc{\ch}{{\mathop {\rm ch}}}
\nc{\Tr}{{\mathop {\rm Tr}\,}}
\nc{\tr}{{\mathrm tr}}
\nc{\Id}{{\mathop {\rm Id}}}
\nc{\ad}{{\mathop {\rm ad}}}
\nc{\End}{{\mathop {\rm End}}}

\nc{\bra}{\langle}
\nc{\ket}{\rangle}
\nc{\bi}{{\bf i}}
\nc{\pa}{\partial}
\nc{\ld}{\ldots}
\nc{\cd}{\cdots}
\nc{\hk}{\hookrightarrow}
\nc{\T}{\otimes}
\nc{\gr}{\mathrm{gr}}
\nc{\ov}{\overline}

\nc{\msl}{\mathfrak{sl}}
\nc{\mgl}{\mathfrak{gl}}
\nc{\U}{\mathrm U}
\nc{\V}{\EuScript V}
\nc{\cO}{\mathcal{O}}
\nc{\cL}{\mathcal{L}}
\nc{\Res}{{\mathbf{Res}}}
\nc{\Ind}{{\mathbf{Ind}}}
\nc{\Coind}{{\mathbf{CoInd}}}
\nc{\LInd}{{\mathbf{LInd}}}
\nc{\RCoind}{{\mathbf{RCoInd}}}

\nc{\sR}{\mathbf{R}}
\nc{\sZ}{\mathbf{Z}}

\newcommand{\bZ}{{\mathbb Z}}

\newcommand{\Ser}{{\EuScript{S}\mathrm{er}}}

\newcommand{\fp}{{\mathfrak p}}
\newcommand{\fh}{{\mathfrak h}}

\newcommand{\fg}{{\mathfrak g}}

\newcommand{\fb}{{\mathfrak b}}

\newcommand{\cF}{{\mathcal{F}}}
\newcommand{\fn}{{\mathfrak n}}

\newcommand{\Hom}{\mathrm{Hom}}

\newcommand{\wt}{\mathrm{wt}}

\nc{\bfI}{\mathbf I}

\nc{\Q}{\mathfrak Q}
\nc{\fr}{\mathfrak r}
\nc{\W}{\mathbb W}
\nc{\bU}{\mathbb U}

\nc{\Gm}{\mathbb{G}_{m}}
\nc{\bA}{\mathbb A}

\newcommand{\calF}{\mathcal{F}}

\newcommand{\calD}{\mathcal{D}}
\newcommand{\bbY}{\mathbf{Y}}

\newcommand{\de}{\text{-}}

\DeclareMathOperator{\Ext}{Ext}

\newcommand{\Span}{\mathsf{Span}}

\newcommand{\Mat}{\mathsf{Mat}}

\newcommand{\St}{\mathbf{S}\kern -.25em{\mathbf{c}}}

\newcommand{\DL}{\mathsf{DL}}

\newcommand{\vrt}{\mathsf{vrt}}
\newcommand{\hor}{\mathsf{hor}}
\newcommand{\hb}{\mathsf{hbs}} 

\providecommand{\eprint}[2][]{\href{http://arxiv.org/abs/#2}{arXiv:#2}}

\newcommand{\one}{\mathbbm{1}}

\nc{\sD}{{{{D}\hspace*{-.9em}\text{\bf{--}}\hspace*{0.3em}}}}
\nc{\sL}{{{{L}\hspace*{-.7em}\text{\bf{--}}\hspace*{0.3em}}}}
\nc{\sLD}{{{{LD}\hspace*{-.9em}\text{\bf{---}}\hspace*{0.3em}}}}
\nc{\sDL}{{{{DL}\hspace*{-.9em}\text{\bf{---}}\hspace*{0.3em}}}}
\nc{\Ar}{{\mathsf{Ar}}}

\nc{\disorder}{\mathsf{Disorder}}
\nc{\rk}{\mathsf{rk}}
\nc{\precsucc}{\overset{\prec}{\succ}}


\begin{document}

\title[Cauchy identities for staircase matrices]
{Cauchy identities for staircase matrices}

\author[Feigin]{Evgeny Feigin}
\address{Evgeny Feigin:\newline
School of Mathematical Sciences, Tel Aviv University, Tel Aviv
69978, Israel
}
\email{evgfeig@gmail.com}

\author[Khoroshkin]{Anton Khoroshkin}
\address{Anton Khoroshkin: \newline
Department of Mathematics, University of Haifa, Mount Carmel, 3498838, Haifa, Israel
}

\email{khoroshkin@gmail.com}

\author[Makedonskyi]{Ievgen Makedonskyi}
\address{Ievgen Makedonskyi:\newline
Yanqi Lake Beijing Institute of Mathematical Sciences And Applications (BIMSA), 
No. 544, Hefangkou Village, Huaibei Town, Huairou District, Beijing 101408.}
\email{makedonskii\_e@mail.ru}

\begin{abstract}
The well known Cauchy identity expresses the product of terms $(1 - x_i y_j)^{-1}$ for $(i,j)$ indexing entries of a rectangular $m\times n$-matrix as a sum over partitions $\lambda$ of products of Schur polynomials: $s_{\lambda}(x)s_{\lambda}(y)$. 
Algebraically, this identity comes from the decomposition of the symmetric algebra of the space of rectangular matrices, considered as a $\mathfrak{gl}_m$-$\mathfrak{gl}_n$-bimodule.
We generalize the Cauchy decomposition by replacing rectangular matrices with arbitrary staircase-shaped matrices equipped with the left and right actions of the Borel upper-triangular subalgebras.
For any given staircase shape $\mathsf{Y}$
we describe left and right ``standard" filtrations on the  
symmetric algebra of the space of shape $\mathsf{Y}$ matrices. We show that the subquotients of these filtrations are tensor products of Demazure and opposite van der Kallen modules over the Borel subalgebras. 
On the level of characters, 
we derive two distinct expansions for the product $(1 - x_i y_j)^{-1}$ for $(i,j) \in \mathsf{Y}$ written as 
sums of products of key polynomials $\kappa_\lambda(x)$ and (opposite) Demazure atoms $a^{\mu}(y)$. 
\end{abstract}

\maketitle
{\small
\tableofcontents
}

\setcounter{section}{-1}
\section{Introduction}
Let us fix positive integers $n$ and $m$ and variables $x=(x_1,\dots,x_n)$, $y=(y_1,\dots,y_m)$.
The celebrated Cauchy identity (see e.g. \cite{Fu}) states that 
\[
\prod_{i=1}^{n}\prod_{j=1}^m \frac{1}{1-x_iy_j} = \sum_{\la} s_\la(x)s_\la(y),
\]
where the sum ranges over partitions $\la$ with at most $\min(m,n)$ parts and $s_\la$ is the corresponding Schur polynomial. This identity has the following representation-theoretic interpretation. Let $\Mat_{m\times n}$ be the space 
of rectangular matrices and let $S(\Mat_{n\times m})$ be the symmetric algebra of this space. Then $S(\Mat_{n\times m})$ as a $\gl_n\times \gl_m$-bimodule decomposes as a direct sum of tensor products
of left and right irreducible modules $V_\la\T V_\la^{op}$ corresponding to the partitions $\la$ (recall that 
Schur polynomials compute the characters of the irreducible $\gl_n$-modules).
The importance of the Cauchy identity is explained by a huge  number of 
applications 
and generalizations popping up in various fields of mathematics, such as combinatorics, representation theory, probability theory, mathematical physics
(see e.g. \cite{BC,BP,BW,Ok,OR,St}). 

The classical Cauchy identity admits various generalizations. 
In particular, one can consider the space of square upper-triangular matrices. 
This leads to the non-symmetric Cauchy identity \cite{CK,FL,Las}:
\[
\prod_{1\le i\le j\le m} \frac{1}{1-x_iy_j} = \sum_{\mu\in\bZ^m_{\geq0}} \kappa_\mu(x)\, a^\mu(y),
\]
where $\kappa_\mu(x)$ are the key polynomials, $a^\mu(y)$ are the opposite Demazure atoms  (see e.g. \cite{Al, AGL, Mas} for combinatorial definitions of these polynomials) and the sum in the right hand side ranges over all compositions $\mu\in\bZ_{\ge 0}^m$ (the opposite and the classical Demazure atoms are related by a simple twist \eqref{eq:left-right}). 
The non-symmetric Cauchy
identity also has an important representation-theoretic incarnation. 
Let $\fb_m\subset\gl_m$ be the Borel subalgebra of upper triangular matrices.
Then the key polynomial $\kappa_\mu(x)$ is the character of the Demazure module $D_\mu$  of 
$\fb_m$ 
(see \cite{Dem1,Dem2}) and the Demazure atom $a^\mu(y)$ is the character of the opposite (right) van der Kallen 
module $K^{op}_\mu$ \cite{vdK} (these are the quotients of the Demazure modules by all the Demazure submodules
strictly contained in a given one). Then one shows that there exists a filtration on the 
symmetric algebra of the space of upper-triangular matrices such that the graded pieces
(the subquotients) are labeled by the compositions $\mu$ and the $\mu$-th 
subquotient is isomorphic (as $\fb_m$-bimodule) to the tensor product $D_\mu\T K_\mu^{op}$. We note that the direct sum decomposition is no longer available (the same 
feature shows up in the $q$-versions of both symmetric and non-symmetric identities, see \cite{FKhM1, FKhM2, FKhMO, FMO}).

Since symmetric algebras of the spaces of matrices of  different shapes lead to interesting
structures, it is natural to ask for the generalization to the case of the general staircase
matrices.  Let $\overline{n}=(n_1\le \dots\le n_m)$ be a collection
of nonnegative integers. We visualize $\overline{n}$ as a Young diagram $\bbY_{\ov{n}}$ with $n_j$ being the length of the $j$-th column; here is an
example for $\ov{n}=(2,4,4,4,5,5)$:
\[\bbY_{(2,4,4,4,5,5)} = \begin{tikzpicture}[scale=0.4]
\draw[step=1cm] (0,2) grid (6,0);
	\draw[step=1cm] (0,1) grid (6,0);
	\draw[step=1cm] (1,0) grid (6,-1);
	\draw[step=1cm] (1,-1) grid (6,-2);
	\draw[step=1cm] (4,-2) grid (6,-3);
\end{tikzpicture}\]
Let $\Mat_{\ov{n}}$ be the space of matrices of shape  $\bbY_{\ov{n}}$. A simple, but
important observation is that the Borel subalgebra $\fb_{n_m}$ acts on $\Mat_{\ov{n}}$
from the left (simply by multiplication) and $\fb_m$ acts from the right.
Our goal is to describe
the symmetric space of $\Mat_{\ov{n}}$ as $\fb_{n_m}\de \fb_m$ bimodule and to obtain the Cauchy-type character identities.

In \cite{Las}, a formula for the expansion of the product 
$\prod_{(i,j)\in \bbY_{\ov{n}}} (1-x_iy_j)^{-1}$ (the character of the symmetric space) is given as a sum over all compositions
of the products $D^{(1)}(\kappa_\mu(x))\, D^{(2)}(a_\mu(y))$, where $D^{(1)}$ and $D^{(2)}$ are certain products of Demazure operators. However, as in the classical and non-symmetric cases, it is very desirable to obtain a formula with each term being a character of 
a certain bimodule over the Borel subalgebras. This goal was achieved in \cite{AGL,AE}
for tuples $\ov{n}$ such that $n_{i+1}-n_i=1$ for all $i$. 

In this paper, we derive different Cauchy formulas for all possible partitions $\ov{n}$. 
These formulas are consequences of the descriptions 
of the left and the right standard filtrations on the symmetric space $S^{\bullet}(\Mat_{\ov{n}})$.
Recall that a standard filtration $\cF^{\ov{c}}$ on a module $M$ 
(where ${\ov{c}}$ belongs to the certain labeling set)
is defined as follows: $\cF^{\ov{c}}$ is the smallest submodule of $M$ containing all $\ov{d}$-weight spaces of $M$ such that $\ov{d}$ is not less than or equal to $\ov{c}$ (in particular, $\cF^{\ov{c}}\supset \cF^{\ov{d}}$
whenever ${\ov{c}}\preceq {\ov{d}}$). 
We deal with the right and left  standard filtrations: in both cases 
$M=S(\Mat_{\ov{n}})$, but in the first case $M$ is considered as a right $\fb_m$-module,
$\ov{c}\in\bZ_{\ge 0}^m$
and in the second case, we consider the left $\fb_{n_m}$-action, $\ov{c}\in \bZ_{\ge 0}^{n_m}$
(in both cases, we use the Cherednik order on the set of weights, Definition \ref{def::Cherednik::order}). Our goal is to answer two questions:
\begin{itemize}
\item  for which $\ov{d}$ the subquotient $\cF^{\ov{d}}\left/ \sum_{\overline{c} \succ \overline{d}}\mathcal{F}^{\overline{c}} \right.$  is non-trivial? 
\item how to describe the non-trivial subquotients?   
\end{itemize}
The answer to the first question gives us the range of summation in the Cauchy identities,
and the answer to the second produces the exact form of terms on the right-hand side of the 
identities.

We start with the right standard filtration. Let $\overline{d}=(d_1,\dots,d_m)$, $d_i\ge 0$. We say that $\ov{d}$ is $\ov{n}$-admissible, if 
for all $l \leq m$: $n_l \geq \#\{i=1,\dots,l:\ d_i >0\}$. 
Our first theorem is as follows:

\begin{theoremA}\label{thm:A} (Theorem \ref{thm:RSF})
The $\ov{d}$-th subquotient of the right standard filtration is non-trivial if and only if 
$\ov{d}$ is $\ov{n}$-admissible. If non-trivial, the subquotient is isomorphic to  
$D_{\hb(\overline d)} \otimes K_{\overline d}^{op}$.
 \end{theoremA}

An important outcome is that each subquotient decomposes as $\fb_{n_m}\de\fb_m$ bimodule 
into a tensor product; the modules that show up in the theorem above are 
the right van der Kallen module  $K_{\overline d}^{op}$ and the left Demazure module 
$D_{\hb(\overline d)}$. The operation $\hb$ (the half-bubble-sort algorithm) is one of the central combinatorial ingredients of the paper, see Section \ref{subsection:Serpentines} for more details. As an immediate corollary of the description of the right standard filtration,
we obtain

\begin{corA}\label{cor:A} (Corollary \ref{cor:firstCauchy})
The following generalized ``right" Cauchy identity holds true
    \[\prod_{(i,j)\in \bbY_{\overline{n}}}\frac{1}{1-x_iy_j}=\sum_{\overline d \text{ is }\overline n \text{-admissible}}\kappa_{\hb(\overline d)}(x)\, a^{\overline d}(y).\]
\end{corA}

Our next task is to describe the left standard filtration. It turned out that the corresponding
subquotients also admit nice and clean descriptions as tensor products of certain left and 
right modules. To state the answer, we need the following combinatorial construction.
To the Young diagram $\bbY_{\ov{n}}$, we attach a subset $\St_{\ov{n}}$ of its cells called the subset of staircase corners and a 
poset structure on this subset. In particular, the set $\St_{\ov{n}}$ contains at most one element in each row and each column.
The precise definition is given in Section~\ref{sec::ST::corners}, here we provide an example
for $\ov{n}=(2,4,4,4,5,5)$ -- the corresponding poset of staircase corners and its Hasse diagram (the cells in the poset are 
marked by blue dots and edges going left and down define the poset structure):

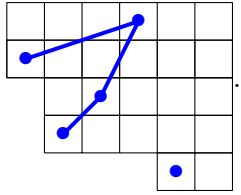
\[
\begin{tikzpicture}[scale=0.5]
\draw[step=1cm] (0,2) grid (6,0);
\draw[step=1cm] (0,1) grid (6,0);
\draw[step=1cm] (1,0) grid (6,-1);
\draw[step=1cm] (1,-1) grid (6,-2);
\draw[step=1cm] (4,-2) grid (6,-3);
 \node (v0) at (0.5,0.5) {{{\color{blue}$\bullet$}}};
 \node (v1) at (1.5,-1.5) { {{\color{blue}$\bullet$}}};
 \node (v2) at (4.5,-2.5) { {{\color{blue}$\bullet$}}};
 \node (v3) at (3.5,1.5) { {{\color{blue}$\bullet$}}};
 \node (v4) at (2.5,-0.5) { {{\color{blue}$\bullet$}}};
\draw[blue,line width=1.5pt] (0.5,0.5) edge (3.5,1.5);
\draw[blue,line width=1.5pt] (1.5,-1.5) edge (2.5,-0.5);
\draw[blue,line width=1.5pt] (2.5,-0.5) edge (3.5,1.5);
\end{tikzpicture}. 
\]

Let $\DL_{\overline n}$ (called DL-dense arrays) be the set of order-preserving maps from $\St_{\ov{n}}$ to $\bZ_{\ge 0}$. In other words, a DL-dense array attaches a non-negative integer to each element of the poset in a way compatible with the poset structure. For $A\in \DL_{\overline n}$ its horizontal weight
$\hor(A)\in \bZ_{\ge 0}^{n_m}$ is the vector of components of $A$ column-wise (note that $A$ contains at most one non-zero element
in each row and in each column). Similarly, one defines the vertical weight $\vrt(A)\in\bZ_{\ge 0}^m$.
Here is our second theorem.

\begin{theoremA}\label{thm:B} (Corollary \ref{cor:nvanderKallenQuotientDemazure} and Corollary \ref{cor:LSFtriv}) 
The subquotient of the left standard filtration 
corresponding to the weight $\mu\in\bZ_{\ge 0}^{n_m}$ is non-trivial if and only if $\mu=\hor(A)$ for some $A \in \DL_{\overline n}$.
In this case, the subquotient is isomorphic
to the tensor product $D_{\hor(A)} \otimes K_{\overline n,\vrt(A)}^{op}$
for certain (right) generalized van der Kallen module $K_{\overline n,\vrt(A)}^{op}$.
\end{theoremA}

We derive several properties of the $\ov{n}$-van der Kallen modules. In particular,  the following theorem explains the name of the generalized van der Kallen modules. Let us prepare certain notation. 
For a partition $\la\in\bZ_{\ge 0}^{n_m}$ let $\DL_{\ov{n}}(\la)\subset\DL_{\ov{n}}$ be the set of DL-dense arrays whose multi-set of entries coincides with that of $\la$.  We endow 
the set $\DL_{\ov{n}}(\la)$ with the following poset structure: $B\preceq A$ if $D^{op}_{\vrt(B)}\subset D^{op}_{\vrt(A)}$.
\begin{theoremA} (Proposition \ref{prop:KD/D})
For all $A\in\DL_{\ov{n}}$
one has the isomorphism of right $\fb_m$-modules: 
\(
K^{op}_{\overline n,\vrt(A)} \simeq D^{op}_{\vrt(A)}\left/\sum_{B\prec A} D^{op}_{\vrt(B)} \right. .
\)
\end{theoremA}
The representation-theoretic properties of generalized van der Kallen modules allow us to derive the following character identities coming from the left standard filtration:
\begin{corA} (Corollary \ref{cor:LCF})
The following generalized ``left" Cauchy identity holds true 
\begin{gather*}
\prod_{(i,j)\in \bbY_{\overline{n}}}\frac{1}{1-x_iy_j}=
\sum_{\substack{A \in \DL_{\overline n}\\ \hb(\ov{d})=\hor(A)}} \kappa_{\hor(A)}(x)\, a^{\ov{d}}(y).
\end{gather*}
\end{corA}

Let us give several remarks on the logic of our arguments. We start with a combinatorial study 
of the decomposition of tensor products of Demazure modules based on character identities found in~\cite{As1, As2, AQ1, HLMW} and derive certain applications for the symmetric space $S(\Mat_{\ov{n}})$. We also introduce the poset $\St_{\ov{n}}$ of staircase matrices, the half-bubble-sort operation
$\hb$ and the set $\DL_{\ov{n}}$ of DL-dense arrays, which play a crucial role in our construction.
We then proceed to the study of the right standard filtration using the theory 
of highest weight categories and representation theory of the Borel subalgebras.
Together with the combinatorial results, this gives a proof of Theorem~\ref{thm:A}. We note that
each subquotient for the right standard filtration is isomorphic to a tensor product
of a left Demazure module and a right van der Kallen module. It turned out that 
the description of the left standard filtration (see Theorem~\ref{thm:B}) comes in parallel with 
the answer to the following question: which van der Kallen modules show up in pairs
with a given Demazure module in Theorem~\ref{thm:A}? We answer this question using
combinatorics of the DL-dense arrays. As a result, we introduce generalized van
der Kallen modules, which are also described as quotients of Demazure modules
by certain embedded Demazure submodules.

Our paper is organized as follows. In Section~\ref{sec::combinatorics}, we introduce our primary combinatorial structures: serpentines, staircase corners, and DL-dense arrays. Section~\ref{sec::Demazure::all} focuses on  Demazure modules and their tensor products. 
In Section~\ref{sec::right::Cauchy} we introduce the left and right standard filtrations on the symmetric algebra of the matrix spaces. We describe the subquotients of the right standard filtrations and derive the first positive ``right" Cauchy formula. 
Section~\ref{sec::left::Cauchy} is devoted to the study of the left standard filtration: we describe the subquotients and derive Cauchy-type formulas.
 Examples illustrating our identities are provided in Appendix~\ref{sec::examples}.

\section*{Acknowledgments}
We are grateful to  Olga Azenhas, Michel Brion, Thomas Gobet, and C\'edric Lecouvey for 
useful correspondence and to Daniel Orr and Alek Vainstein for useful discussions.

The work of Ie. Makedonskyi was partially supported by BJNSF grant IS24002.

Last but not least, we are very grateful to the anonymous referee for the careful reading of the manuscript and for 
a number of valuable comments and suggestions, which significantly improved the quality of the paper.
 
\section{Combinatorics}
\label{sec::combinatorics}
In this section, we fix the notation and introduce the main combinatorial objects of our paper. The results of this section
are used later on to derive Cauchy-type formulas and to describe the representation theory behind them. 
The main definitions of this section are Definitions \ref{def: serp} and \ref{def::DL::indices}, and the main
results are Lemmas \ref{lem::reordering} and \ref{lem:inequalityStaircase}.

We fix a positive integer  $m$ and a collection of  integers $\ov{n}:=(n_1,\ldots,n_m)$ such that
$$0<n_1\leq n_2 \leq \ldots \leq n_m.$$
The corresponding Young diagram  
\[
\bbY_{\overline{n}} = \{(i,j): 1\le j\le m,\ 1\le i\le n_j\}\subset \bZ_{>0}^2.
\]
is visualized as a collection of $m$ columns of lengths $n_1,\dots,n_m$ drawn from left to right. Here 
is an example of the Young diagram $\bbY_{\ov{n}}$ 
associated with $\overline n =(1,1,3,3,3,4,4)$: 
\begin{equation}
\label{ex::Young}
\bbY_{(1,1,3,3,3,4,4)} = \begin{tikzpicture}[scale=0.5]
	\draw[step=1cm] (0,1) grid (7,0);
	\draw[step=1cm] (2,0) grid (7,-1);
	\draw[step=1cm] (2,-1) grid (7,-2);
	\draw[step=1cm] (5,-2) grid (7,-3);
\end{tikzpicture}
\end{equation}
In particular, the cells in the upper row are of the form $(1,j)$ for 
$1\le j\le m$. 

The main object of our study is the symmetric algebra of staircase matrices of shape $\bbY_{\ov{n}}$. In other words, we are interested in the subspace of rectangular 
$n_m\times m$-matrices consisting of matrices with zero entries in the cells that do not belong 
to the Young diagram $\bbY_{\ov{n}}$. For example, if $\ov{n}=(1,2,\dots,m)$, then $\bbY_{\ov{n}}$ corresponds to the space of upper-triangular matrices.

\subsection{Combinatorics of serpentines} \label{subsection:Serpentines}\label{sec:Serpentines}
The key polynomials are the characters of the  Demazure modules \cite{Dem2}. They have many beautiful combinatorial properties, in particular, an explicit  
formula using the key tableaux is available (see \cite{LS}).
In \cite{As1, AQ1, HLMW}, the authors gave a formula for the decomposition of a product
of an arbitrary key polynomial with a single-part key polynomial as a linear
combination of key polynomials (thus generalizing the classical Pieri formula). 
The formula is multiplicity-free, but in general, it does contain
negative signs. However, there is a special case when the formula is totally
positive. More precisely, let $\la=(\la_1,\dots,\la_n)$ be a composition
and let $\kappa_\la(x_1,\dots,x_n)$ be the corresponding key polynomial. 
We consider the product of $\kappa_\la(x)$ with the key polynomial  $\kappa_{(0,\dots,0,d)}(x)$ with $n-1$ zeros 
(so $\kappa_{(0,\dots,0,d)}(x)$ is equal to the Schur polynomial $s_{(d,0,\dots,0)}(x_1,\dots,x_n)$). 
For a composition $\la$ and $1\le j\le n$ let $\la +\one_j$ be the composition obtained from $\la$ by 
increasing by one the $j$-th component. 
One has  (\cite{HLMW}, Theorem~6.1 and \cite{As1}, Theorem~ 3.10; see also a preprint \cite{AQ2}, Corollary~6.3.3 for the 
exact form we use):
\begin{equation}\label{AQ}
\kappa_\la(x) \kappa_{(0,\dots,0,d)}(x) = \sum_{\substack{c_1<\dots <c_d\\  
c_a-1\in\{\la_1,\dots,\la_n,c_{a-1}\}}} \kappa_{\la+\one_{j_1}+\dots+ \one_{j_d}}(x), 
\end{equation}
where $j_a$ is the largest index such that for $\mu^{(a-1)}=\la+\one_{j_1}+\dots +\one_{j_{a-1}}$ 
one has $\mu^{(a-1)}_{j_a}=c_a-1$.
\begin{rem}
In the formula above, we assume that $\mu^{(0)}=\la$.    
\end{rem}

\begin{example} Here is an example of the product of key polynomials as above:
\begin{multline*}
    \kappa_{(3,0,3,1)}(x)\, \kappa_{(0,0,0,2)}(x) = 
\kappa_{(3,1,3,2)}(x)    + \kappa_{(3,1,4,1)}(x) + 
    \kappa_{(3,0,3,3)}(x) + \kappa_{(3,0,4,2)}(x)
    \kappa_{(3,0,5,1)}(x).
\end{multline*}
Let us visualize $\lambda$ as a collection of columns with the number of cells in the $i$'th column equal to $\lambda_i$. Then the previous identity has the following pictorial presentation: 
$$\composition{4}{3,0,3,1} \cdot \glueComposition{4}{0,0,0,0}{0,0,0,2} = 
\glueComposition{4}{3,0,3,1}{0,1,0,1} + 
\glueComposition{4}{3,0,3,1}{0,1,1,0} + 
\glueComposition{4}{3,0,3,1}{0,0,0,2} + 
\glueComposition{4}{3,0,3,1}{0,0,1,1} +
\glueComposition{4}{3,0,3,1}{0,0,2,0}  
$$

In fact, according to formula \eqref{AQ}, possible values of $(c_1,c_2)$ are $(1,2)$, $(1,4)$,
$(2,3)$, $(2,4)$, and $(4,5)$. Below, we compute the summands of the right-hand side 
of \eqref{AQ} corresponding to each pair $(c_1,c_2)$:
\begin{itemize}
\item $(c_1,c_2)=(1,2)$, $j_1=2$, $\mu^{(1)}=(3,1,3,1)$, $j_2=4$, 
summand  $\kappa_{(3,1,3,2)}(x)$;  
\item $(c_1,c_2)=(1,4)$, $j_1=2$, $\mu^{(1)}=(3,1,3,1)$, $j_2=3$, 
summand  $\kappa_{(3,1,4,1)}(x)$;
\item $(c_1,c_2)=(2,3)$, $j_1=4$, $\mu^{(1)}=(3,0,3,2)$, $j_2=4$, 
summand  $\kappa_{(3,0,3,3)}(x)$;
\item $(c_1,c_2)=(2,4)$, $j_1=4$, $\mu^{(1)}=(3,0,3,2)$, $j_2=3$, 
summand  $\kappa_{(3,0,4,2)}(x)$;
\item $(c_1,c_2)=(4,5)$, $j_1=3$, $\mu^{(1)}=(3,0,4,1)$, $j_2=3$, 
summand  $\kappa_{(3,0,5,1)}(x)$.   
\end{itemize}
\end{example}

Our goal is to describe explicitly the compositions $\mu$ such that $\kappa_\mu$ shows up on the right-hand side of \eqref{AQ}. To this end, we put forward the following definition.

\begin{definition}\label{def: serp}
The set $\Ser_{d,\lambda}^n$ of $d$-serpentines  associated with a composition 
$\lambda=(\lambda_1,\ldots,\lambda_n)$ consists of compositions $\mu=(\mu_1,\ldots,\mu_n)$ such that 
\begin{itemize}
\item $\mu_i\geq\lambda_i$;
\item $\sum_{i=1}^{n}(\mu_i-\lambda_i)= d$;
\item if $i< j$ and $\lambda_i \leq \lambda_j$, then  $\mu_i\leq\lambda_j$;
\item if $i< j$ and $\lambda_j < \lambda_i\leq \mu_j$, then $\mu_i=\lambda_i$.
\end{itemize}
\end{definition}

\begin{prop}\label{prop:serp}
One has
\begin{equation}\label{eq:serp}
\kappa_\la(x) \kappa_{(0,\dots,0,d)}(x) = \sum_{\mu\in \Ser_{d,\lambda}^n} \kappa_\mu(x). 
\end{equation}
\end{prop}
\begin{proof}
We first show that if $\kappa_\mu(x)$ shows up in the right hand side of \eqref{AQ}, then $\mu\in \Ser_{d,\lambda}^n$. The first two conditions from Definition \ref{def: serp}
obviously hold true for any summand $\kappa_\mu(x)$ from the right hand side of \eqref{AQ}.

Now assume $i<j$, $\la_i\le \la_j$ and $\mu_i>\la_j$. Let us take $a$ such that  
\[
\mu^{(a-1)}_i=\la_j \text{ and } 
\mu^{(a)}_i> \la_j
\]
(note that such an $a$ does exist because $\mu^{(0)}_i\le \la_j$, $\mu^{(d)}_i> \la_j$ and $c_\bullet<c_{\bullet+1}$). Then 
$c_a=\la_j+1$, $j_a=i$. But this contradicts the condition that $j_a$ is the largest number such that $\mu^{(a-1)}_{j_a}=\la_j$ ($j_a=j>i$ works as well).

Our next task is to show that the following is impossible:  $i<j$, $\la_j<\la_i\le \mu_j$ and $\mu_i>\la_i$. In fact, if
$\mu_i>\la_i$, then there exists an $a$ such that $c_a=\la_i+1$, $j_a=i$. Since $i<j$ this implies that $\mu^{(a-1)}_j < \la_i$
and since $c_a=\la_i+1$ we obtain $\mu_j <\la_i$ (recall $c_a<c_{a+1}<\dots$).
So we get a contradiction with the condition $\la_i\le \mu_j$.

Now let us work out the opposite direction. We need to prove that if $\mu\in \Ser_{d,\lambda}^n$, then $\kappa_\mu$ shows up as a
summand in the right-hand side of \eqref{AQ}. There is a unique way to define the numbers $c_1,\dots,c_d$ and $j_1,\dots,j_d$ in a  way compatible with formula \eqref{AQ}.
Namely, $c_d$ is the largest number among $\mu_j$ such that $\mu_j>\la_j$ and $j_d$ is equal to the largest $j$ satisfying $\mu_j=c_d$, $\mu_j>\la_j$. We also define 
$\mu^{(d-1)}=\mu - \one_{j_d}$. We then define $c_{d-1}$, $j_{d-1}$ and $\mu^{(d-2)}$ by replacing $\mu=\mu^{(a)}$ with $\mu^{(a-1)}$. Using decreasing induction, we define all $c_a$, $j_a$, and $\mu^{(a)}$.
Our goal is to prove that $c_1<\dots <c_d$ and that $j_a$ is the largest index such that $\mu^{(a-1)}_{j_a}=c_a-1$.

Let us show that $c_{d-1}<c_d$. If not, then $c_{d-1}=c_d$ and for the numbers $j_d> j_{d-1}$ one has: 
\[
\mu_{j_d}=c_d>\la_{j_d},\ \mu_{j_{d-1}}=c_{d-1}>\la_{j_{d-1}}, \ \mu_{j_d}= \mu_{j_{d-1}}.
\]
However, this contradicts the definition of the $d$-serpentine. Indeed, if $\la_{j_{d-1}}\le \la_{j_d}$, then we do not have $\mu_{j_{d-1}}\le \la_{j_{d-1}}$; on the contrary,
if $\la_{j_{d-1}}>\la_{j_d}$, then we do know that $\mu_{j_{d}}>\la_{j_{d-1}}$, but we do not have $\mu_{j_{d-1}}=\la_{j_{d-1}}$.
Hence $c_d>c_{d-1}$ and by construction 
$j_d$ is the largest index such that $\mu^{(d-1)}_{j_d}=c_d-1$. 

To complete the proof, it suffices to show that $\mu-\one_{j_d}\in \Ser_{d-1,\lambda}^n$.
The first three conditions obviously hold true. (In particular, the third one is still valid since we do not change $\la$, but make  $\mu_{j_d}$ smaller).
Assume that $i<j$, $\la_j<\la_i\le\mu^{(d-1)}_j$. If $\mu^{(d-1)}_j=\mu_j$, then $\mu_i=\la_i$ and hence $\mu^{(d-1)}_i=\la_i$ as well.
But if $\mu^{(d-1)}_j=\mu_j-1$ (i.e. $j=j_d$), then we still have 
$\la_j<\la_i\le\mu_j$, implying $\la_i=\mu_i$. Since $\mu^{(d-1)}_i=\mu_i$ (just because $i\ne j$) we conclude 
$\la_i=\mu_i$. Now using the decreasing induction on $d$ one proves that $c_1<\dots <c_d$ (and by construction $j_d$ is the largest index such that $\mu^{(d-1)}_{j_d}=c_d-1$). 
Hence $\kappa_\mu(x)$ pops up on the right-hand side of \eqref{AQ}, and we are done. 
\end{proof}

\begin{definition}
The dominant part $\lambda_+$ of the composition $\lambda=(\la_1,\dots,\la_l)$ is the partition $(\lambda_+)_1\ge \dots \ge (\lambda_+)_l\in {S}_l \lambda$. In other words, $\lambda_+$ is the non-increasing reordering of $\lambda$.
\end{definition}

In the rest of the paper, we identify a composition $\lambda$ with the one obtained from $\la$ by adding several zeroes at the end.
Clearly, this identification respects the definition of the dominant part:
$$ (\lambda_1,\ldots,\lambda_n) \sim  (\lambda_1,\ldots,\lambda_n,0) \
 {\Rightarrow}\ 
(\lambda_1,\ldots,\lambda_n)_{+} \sim  (\lambda_1,\ldots,\lambda_n,0)_{+}.
$$
For a composition $\lambda=(\lambda_1,\dots,\lambda_n)$ with $\lambda_n \neq 0$, we denote by $(\lambda,d)$ the composition $\lambda=(\lambda_1,\dots,\lambda_n,d)$. 
If $\la_n=0$, then we first pass to the equivalent composition with the nonzero last entry and then add $d$ at the very end.  

Let us recall the standard partial order called dominance order on the set of dominant weights (partitions):
\begin{equation*}
\mu \geq \lambda  \ \stackrel{\mathsf{def}}{\Leftrightarrow} \ 
|\mu| = |\lambda| \ \& \ \forall s=1,\ldots, n \ \text{ we have }
\mu_1+\ldots+\mu_s \geq \lambda_1+\ldots+\lambda_s.
\end{equation*}
\begin{lem}\label{dominant WeightInequality}
  Assume that $\mu\in \Ser_{d,\lambda}^n$. Then  $\mu_+ \geq (\lambda,d)_+$.
\end{lem}
\begin{proof}
Let $r_1,\dots,r_s$ be the set of numbers such that $(\mu_+)_{i} \neq (\lambda_+)_{i}$. Then, by the third property from the definition of serpentine, we have
\[(\lambda_+)_{r_1}\geq (\mu_+)_{r_2}>(\lambda_+)_{r_2}\geq \dots \geq (\mu_+)_{r_s}>(\lambda_+)_{r_s}.\]

Therefore, we have the following property of the dominant parts of $\lambda$
and $\mu$:
\[(\mu_+)_{i}\geq (\lambda_+)_{i}\geq (\mu_+)_{i+1}.\]
In other words, $\mu_+$ forms a horizontal strip associated with $\lambda_+$.
Let $t$ be any number such that $(\lambda_+)_{t} \geq d \geq (\lambda_+)_{t+1}$.
For $p\leq t$ we clearly have $\sum_{q=1}^p(\mu_+)_q \geq \sum_{q=1}^p(\lambda_+)_q$. For $p \geq t$, using the equality $\sum_{i=1}^{n+1}((\mu_+)_{i}-(\lambda_+)_{i})=d$, we have:
\begin{equation*}\sum_{q=1}^{p+1}(\mu_+)_q-
\sum_{q=1}^t(\lambda_+)_q -d-\sum_{q=t+1}^p(\lambda_+)_q
=\sum_{q=p+1}^{n+1}(\lambda_+)_q-\sum_{q=p+2}^{n+1}(\mu_+)_q\geq 0.
\end{equation*}
Therefore, $\mu_+\geq (\lambda,d)_+$ in the dominance order and the proof is completed.
\end{proof}

\begin{cor}    \label{cor:weightinequality}
Let $\ov{d}=(d_1,\dots,d_m)\in\bZ_{\ge 0}^m$.
If $\mu^{(0)}=(0)$, $\mu^{(j)} \in \Ser_{d_j,\mu^{(j-1)}}^{n_j}$, $j=1,\dots,m$, then $\mu_+^{(m)} \geq (d_1,\dots,d_m)_+$.
\end{cor}

\begin{cor}\label{cor:SubsequenceInequality}
If $\mu^{(0)}=(0)$, $\mu^{(j)} \in \Ser_{d_j,\mu^{(j-1)}}^{n_j}$, $j=1,\dots,m$, $\mu_+^{(m)} = (d_1,\dots,d_m)_+$, then $\mu_+^{(i)} = (d_1,\dots,d_i)_+$ for any $1 \leq i \leq m$.
\end{cor}
\begin{proof}
    Assume that $\mu^{(j)}_+\neq (d_1,\dots,d_j)_+$. Then by Lemma \ref{dominant WeightInequality} we have $\mu^{(j)}_+ > (d_1,\dots,d_j)_+$. Then by Corollary \ref{cor:weightinequality} we get $\mu^{(m)}_+ > (d_1,\dots,d_m)_+$ which gives us a contradiction.
\end{proof}

Let us now come back to the set of staircase matrices, namely, 
recall that we have fixed a collection $\ov{n}=(n_1\le \dots \le n_m)$ of sizes of columns of the Young diagram $\bbY_{\ov{n}}$. The following lemma will be extensively used in the rest of the paper.

\begin{lem}
\label{lem::reordering}
For a composition $\lambda\in\bZ_{\ge 0}^{n}$ and a positive integer $d$, there exists a unique serpentine $\mu\in \Ser_{d,\lambda}^n$,  
$\mu_+=(\lambda,d)_+$ if and only if $\lambda_i=0$ for some $i\le n$. If this condition holds true, then 
$\mu$ admits the following recursive description.
A collection of increasing indices $i_0,i_1,\ldots$ 
and values $\mu_{i_s}>\lambda_{i_s}$ are defined by 
\begin{equation}
\label{eq::bubble::serpentine}
\left\{
\begin{array}{c}
i_0:=\max\{i \colon \lambda_i=0\} \text{ and } \mu_{i_0}:=\min\{d, \lambda_{i_0+1},\lambda_{i_0+2},\ldots,\lambda_{n}\}; \\
\text{ if }\mu_{i_0}\neq d, \text{ then }
i_1:=\max\{i \colon \lambda_i = \mu_{i_0}\} \text{ and } \mu_{i_1}:=\min\{d, \lambda_{i_1+1},\lambda_{i_1+2},\ldots,\lambda_{n}\};\\
\text{ if }\mu_{i_1}\neq d, \text{ then }
i_2:=\max\{i \colon \lambda_i = \mu_{i_1}\} \text{ and } \mu_{i_2}:=\min\{d, \lambda_{i_2+1},\lambda_{i_2+2},\ldots,\lambda_{n}\};\\
\ldots \\
\text{ we stop for $s$ such that }\mu_{i_s}=d.
\end{array}
\right\}
\end{equation} 
All other coordinates of the serpentine $\mu$ coincide with the corresponding coordinates of $\lambda$.
\end{lem}
Before providing a proof, let us present a pictorial example of the serpentine $\mu\in \Ser_{d,\lambda}^n$ such that $\mu_+=(\lambda,d)_+$ with $n=6$, $\lambda=(0,7,0,3,3,5)$, and $d=6$:
$$
n=6, \quad 
\lambda =\composition{6}{0,7,0,3,3,4}\ , \quad d=\glueComposition{6}{0,0,0,0,0,0}{0,0,0,0,0,6}\ , \quad \text{ and } \quad \mu=\glueComposition{6}{0,7,0,3,3,4}{0,0,3,0,1,2}
$$
\begin{proof}
First, let us notice that the number of nonzero elements in $(\lambda,d)_+$ is larger by one than the number of nonzero elements of $\lambda_{+}$. Therefore, the assumption of the existence of $i$ such that $\lambda_i=0$ is necessary.
Second, one can see from the recursive description~\eqref{eq::bubble::serpentine} that the corresponding $\mu$ is a serpentine such that $\mu_+=(\lambda,d)_{+}$.

Let us show that the recursive description~\eqref{eq::bubble::serpentine} holds true for any given $\mu\in\Ser_{d,\lambda}^{n}$ under the assumption $\mu_+=(\lambda,d)_{+}$.
Indeed, let us fix the set of indices $i_0<i_1<\ldots$ mentioned in~\eqref{eq::bubble::serpentine} which depends only on $\lambda$.
Let $\tilde{\lambda}:=\lambda+d\cdot\one_{i_0}$ be the composition that differs from $\lambda$ in exactly one place $i_0$, where we increase $\lambda_{i_0}$ from $0$ to $d$. The assumption $\tilde{\lambda}_{+}=(\lambda,d)_+=\mu_+$ implies the existence of the permutation $\sigma$ of the minimal length such that  $\mu_{i}=\tilde{\lambda}_{\sigma(i)}$ for all $i$.
Note that if $\sigma(i)\neq i$, then $\tilde\lambda_i\neq \tilde\lambda_{\sigma(i)}$, because $\sigma$ is of the minimal possible length.
Let $(l_0 l_1 \ldots l_k)$ be one of the cycles in the cyclic decomposition of $\sigma$. Note that if $l_s\neq i_0$, then $\tilde\lambda_{l_s}= \lambda_{l_s}<\mu_{l_s}=\tilde\lambda_{l_{s+1}}$. Hence, if $i_0$ does not belong to $\{l_0,\ldots,l_k\}$ we obtain the contradiction
$$\lambda_{l_0} <\lambda_{l_1} <\ldots <\lambda_{l_k} <\lambda_{l_0}
$$
(see the last condition in Definition~\ref{def: serp}). 
Therefore, the cyclic decomposition of $\sigma$ consists of a unique cycle $(i_0 l_1\ldots l_k)$.
In particular, we have:
$$
0< \mu_{i_0} = \lambda_{l_1} < \mu_{l_1} = \lambda_{l_2} < \ldots < \mu_{l_k} = \tilde\lambda_{i_0}=d.
$$

The last condition of Definition~\ref{def: serp} implies that $i_0<l_1<\ldots<l_k$ and the third condition implies that $\mu_{i_0}$ is not greater than $\lambda_{s}$ for any $s>i_0$. Hence, $l_1=i_1$ whenever $\mu_{i_0}<d$.
Similarly, we see that $\mu_{i_1}$ is not greater than $\lambda_s$ for all $s>i_1$ and $l_2=i_2$ whenever $\mu_{i_1}<d$ and so on.
We conclude that $\mu$ is defined by the recursive procedure~\eqref{eq::bubble::serpentine}.
\end{proof}

\begin{definition}\label{def:AdmissibleDefinition}
We say that a composition  $\ov{d}:=(d_1,\ldots,d_m)$ is {$\ov{n}$-admissible} if for all $s=1,\ldots,m$ the number of nonzero elements $d_j$ 
with $j\leq s$ does not exceed $n_s$:
\begin{equation}
\label{eq::admissible::def}
\#\{j\leq s\ \colon\ d_j\neq 0\} \ \leq  n_s.    
\end{equation}
\end{definition}

\begin{cor}\label{cor:iteratedserp}
Let $\ov{d}$ be $\ov{n}$--admissible.
Then there exists a unique chain $\mu^{(0)}=(0)$, $\mu^{(j)} \in \Ser_{d_j,\mu^{j-1}}^{n_j}$, $j=1,\dots,m$, such that $\mu_+^{(m)} = (d_1,\dots,d_m)_+$. If 
$\ov{d}$ is not $\ov{n}$--admissible, then there is no such chain.
\end{cor}
The following notation will be frequently used in the following sections ($\mu^{(m)}$ is as in the above Corollary):
\begin{notation}
The composition $\mu^{(m)}\in\bZ_{\geq0}^{n_m}$ assigned to an $\ov{n}$-admissible composition $\ov{d}\in\bZ_{\geq 0}^{m}$ as in Corollary~\ref{cor:iteratedserp} is called \emph{half-bubble-sort} of $\ov{d}$ and is denoted
\begin{equation}\label{eq:dT}
\hb(\overline d)=\mu^{(m)}\in\bZ_{\ge 0}^{n_m}.
\end{equation}
\end{notation}

The operation $\hb_{\ov n}$ has various nice properties, see for example Lemma \ref{cor:LCF}.

Note that $\hb(\ov{d})$ depends on a tuple $\ov{n}$, so the clearer notation will be $\hb_{\ov{n}}(\ov{d})$. However, if the Young diagram $\bbY_{\ov{n}}$ is fixed, then we omit $\ov{n}$ and write simply $\hb(\ov{d})$. See also Remark~\ref{rem::hbs} explaining the relationship to the well-known bubble-sort algorithm. 

\begin{example}
Let  $\ov{n}=(3,3,3,3)$. Then $\ov{d}=(d_1,d_2,d_3,d_4)$ is $\ov{n}$-admissible if and only if there exists $j$ such that $d_j=0$.
For an $\ov{n}$-admissible $\ov{d}$ the tuple $\hb(\ov{d})\in\bZ_{\ge 0}^3$ is obtained as follows: remove one zero entry from $\ov{d}$ and
reorder the remaining entries in the non-decreasing order.
\end{example}

\begin{example}
Let $m=3$, $\ov{n}=(3,4,4)$. Then 
\[
\hb(d_1,d_2,d_3)=\begin{cases}
(0,d_3,d_1,d_2), \ d_3\le d_1,d_2;\\
(0,d_1,d_3,d_2), \ d_3\le d_2, d_3>d_1;\\
(0,d_2,d_1,d_3), \ d_1\ge d_2, d_3>d_2;\\
(0,d_1,d_2,d_3), \ d_3 > d_2 > d_1.
\end{cases}
\]
\end{example}

\begin{example}
Let $\ov{n}=(n,\dots,n)\in \bZ_{\ge 0}^m$. Then $\hb(\ov{d})$ is the non-decreasing sorting of $\ov{d}$,
see section \ref{subsec:rectangular} for the precise statement.  
\end{example}

\begin{example}
Let $\ov{n}=(1,\dots,n)\in \bZ_{\ge 0}^n$. Then $\hb(\ov{d})=\ov{d}$, 
see section \ref{subsec:upper-triang} for more details.  
\end{example}

\subsection{The poset of staircase-corners}
\label{sec::ST::corners}
In this subsection, we assign to each partition $\ov{n}$ a subset of indices $\St_{\ov{n}}\subset \bbY_{\ov{n}}$ called the subset of staircase corners. We define the partial order on this set and describe some basic properties of this poset.

For a cell $(i,j)\in \bbY_{\overline{n}}$ we define the partition $\overline{n^{ij}}=(n^{ij}_1\le\ldots\le n^{ij}_{m-1})$ obtained from $\overline{n}$ by erasing the $i$-th row and the $j$-th column:
    \begin{equation}
    \label{eq::erased::diagram}
    n^{ij}_k:=\begin{cases}
        n_k, \text{ if } k< j \ \& \ n_k < i, \\
        n_k-1, \text{ if } k < j \ \& \ n_k \geq i, \\
        n_{k+1}-1, \text{ if } k\geq j. 
    \end{cases}.
    \end{equation}

There is a natural bijection between the cells of the complement of $i$'th row and $j$'th column of $\bbY_{\ov{n}}$ and the cells of $\bbY_{\ov{n^{ij}}}$ which we denote by $\pi_{i,j}$:
\begin{equation}
\label{eq::erased::pi}    
\pi_{ij}(s,t):=\begin{cases}
    (s,t) \text{ if } s<i\ \& \ j<t,\\
    (s-1,t) \text{ if } s>i\ \& \ j<t,\\
    (s,t-1) \text{ if } s<i\ \& \ j>t,\\
    (s-1,t-1) \text{ if } s>i\ \& \ j>t.
\end{cases}
    \end{equation}

\begin{definition}
\label{def::DL::indices}
The subset $\St_{\ov{n}}\subset \bbY_{\overline{n}}$ of {staircase corners} is defined inductively by the following properties:
\begin{itemize}
    \item in each row and in each column, there exists at most one cell ({staircase corner}) that belongs to $\St_{\ov{n}}$;
    \item if $n_j>n_{j-1}$, then the corner cell $(n_j,j)$ of the Young diagram $\bbY_{\ov{n}}$ belongs to $\St_{\ov{n}}$;
\item 
for any $(i,j)\in\St_{\ov{n}}$, we have
$(s,t)\in \St_{\ov{n}}\setminus\{(i,j)\}\ \Leftrightarrow \ \pi_{ij}(s,t)\in \St_{\ov{n^{ij}}}$.
In other words, for any staircase corner $(i,j)\in\St_{\ov{n}}$ any other staircase corner will be a staircase corner of the diagram with $i$'th row and $j$'th column removed.
\end{itemize} 
\end{definition}

We call a subset $S$ of the Young diagram $\bbY_{\overline{n}}$ a rook placement if,
for any cell $(i,j)\in \bbY_{\overline{n}}$, at least one of the two statements holds true: a). there exists a unique element of $S$ in the $i$-th row, 
b).  there exists a unique element of $S$ in the $j$-th column.

\begin{lem}
\label{lem::staircorners}
For any partition $\ov{n}$, the set of staircase corners $\St_{\ov{n}}$ exists and is unique. 
The set $\St_{\ov{n}}$ is a rook placement of the Young diagram $\bbY_{\overline{n}}$.
\end{lem}
\begin{proof}
By definition, all corners (cells $(n_j,j)$ such that $n_j>n_{j-1}$) of $\bbY_{\ov{n}}$ belong to the set $\St_{\ov{n}}$. Let us take one corner $(n_j,j)$
and remove the thin hook containing this corner. In the resulting Young diagram we again choose an arbitrary corner and remove 
the corresponding thin hook and proceed until the diagram becomes empty. This procedure produces a subset $\St$ of the set of cells of $\bbY_{\ov{n}}$. 

Let us show that $\St$ does 
not depend on the  choices of the corners. We use the following simple observation: for two arbitrary corners  of the diagram $\bbY_{\ov{n}}$
the operations of removing thin hooks corresponding to these corners commute. 
This observation implies the following statement: for any sequence of corners used for the construction of $\St$ the same set $\St$ is obtained if one starts with all the corners of the initial diagram $\bbY_{\ov{n}}$. 
In fact, let $(n_j,j)$ be the first element in the sequence of corners used to construct $\St$ such that
\begin{itemize}
\item $(n_j,j)$ is a corner of the initial diagram $\bbY_{\ov{n}}$;
\item there exists a corner in our sequence, chosen before  $(n_j,j)$, which is not a corner of the initial diagram $\bbY_{\ov{n}}$;
\item $(n_j,j)$ is the first element in our sequence of corners with the above properties.
\end{itemize}  
Then, by the observation above, the operation of removing the thin hook corresponding to $(n_j,j)$ commutes with all the removals coming before it and hence can be used at the very beginning.
Summarizing, $\St$ can be obtained by
removing all thin hooks for all corners of $\bbY_{\ov{n}}$, then removing all thin hooks for all corners of the resulting Young diagram and proceeding until the diagram becomes empty.

We have shown that the set $\St$ does not depend on the choices entering the procedure.
The last property in Definition \ref{def::DL::indices} implies that  $\St_{\ov{n}}$ if exists coincides with $\St$ (just take $(i,j)$ to be a corner of $\bbY_{\ov{n}}$). 
It remains to show that $\St$ satisfies all the properties from Definition  \ref{def::DL::indices}. 
This is easy to see for the first two properties. Let us prove that $\St$ satisfies the third property.

To emphasize that the initial Young diagram is of shape $\ov{n}$, we
write $\St(\ov{n})$ for $\St$. Let us fix an element $(i,j)\in \St(\ov{n})$.  
We need to show that $(s,t)\in \St({\ov{n}})$ if and only if  $\pi_{ij}(s,t)\in \St(\ov{n^{ij}})$
(recall that $\overline{n^{ij}}$ and $\pi_{ij}$ are defined in  ~\eqref{eq::erased::diagram} and \eqref{eq::erased::pi}).
We prove the ``only if" part, the ``if" part is very similar.
So for  $(s,t)\in \St({\ov{n}})$ we need to show that $\pi_{ij}(s,t)\in \St(\ov{n^{ij}})$.
Since $(i,j),(s,t)\in \St({\ov{n}})$, they are members of a sequence of corners used to construct $\St({\ov{n}})$.
Let us consider two options. If $(s,t)$ shows up in this sequence earlier than $(i,j)$ does, then $\pi_{ij}(s,t)\in \St(\ov{n^{ij}})$, since one
can construct $\St(\ov{n^{ij}})$ starting with the subsequence of corners
showing up before $(i,j)$ (which includes $(s,t)$). Now assume that  $(s,t)$ shows up later than $(i,j)$ does. 
Then $\St(\ov{n^{ij}})$ is obtained by the same sequence of corners as
$\St(\ov{n})$ with $(i,j)$ removed (which includes $(s,t)$).
\end{proof}

Let us provide some examples of the set of
staircase corners of a Young diagram.
We draw the elements of $\St_{\ov{n}}$ by blue dots and draw by different colors the thin hooks (row + columns).

\begin{example} 
Let $\ov{n}=(3,3,3,3,3,3,3)$ and  $\ov{n}=(1,2,3,4,4,4,4)$. Then one has
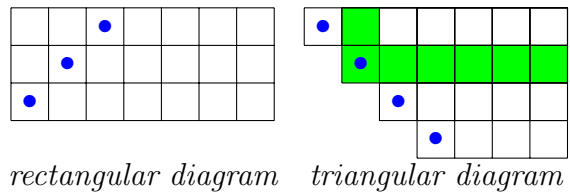
\[
\begin{array}{cc}
{\begin{tikzpicture}[scale=0.5]
	\draw[step=1cm] (0,-3) grid (7,0);
\node (v00) at (.5,-2.5) { {{\color{blue}$\bullet$}}};
\node (v10) at  (.5+1,-1.5) {{ \color{blue}$\bullet$ }};
\node (v20) at (.5+2,-0.5) {{ \color{blue}$\bullet$ }};
\end{tikzpicture}
}
&
{
\begin{tikzpicture}[scale=0.5]
  \fill[green] (1,-2) -- (1,0) -- (2,0)--(2,-2);
  \fill[green] (1,-2) -- (1,-1) -- (7,-1)--(7,-2);
	\draw[step=1cm] (0,-1) grid (7,0);
	\draw[step=1cm] (1,-2) grid (7,0);
	\draw[step=1cm] (2,-3) grid (7,0);
	\draw[step=1cm] (3,-4) grid (7,0);
\node (v00) at (.5,-0.5) { {{\color{blue}$\bullet$}}};
\node (v10) at  (.5+1,-1.5) {{ \color{blue}$\bullet$ }};
\node (v20) at (.5+2,-2.5) {{ \color{blue}$\bullet$ }};
\node (v20) at (.5+3,-3.5) {{ \color{blue}$\bullet$ }};
\end{tikzpicture} }
\\
{\text{rectangular diagram}} & {\text{triangular diagram}}
\end{array}
\]
\begin{example}\label{ex:large}
Here is a larger example for\\ 
$\ov{n}=(3,3,3,3,3,5,9,9,13,13,13,13,13,13,13,13, 13,13,16,16)$:
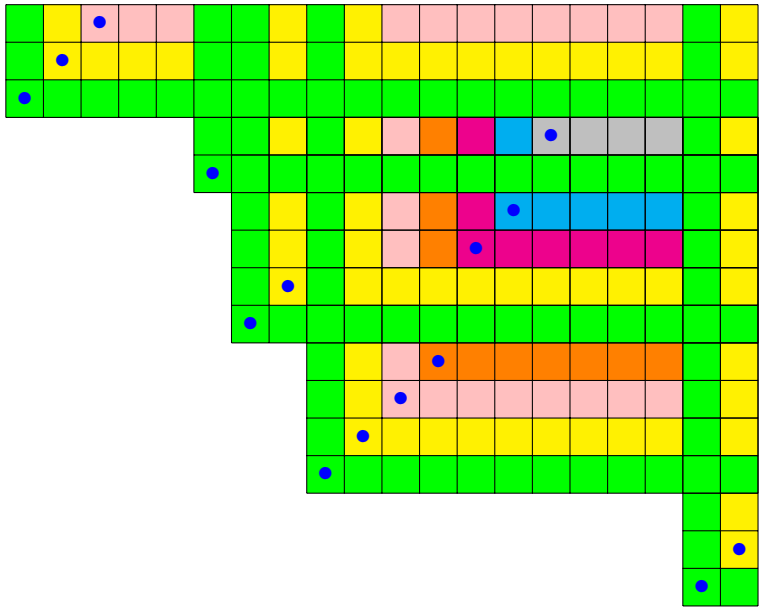
\[
\begin{tikzpicture}[scale=0.5]
  \fill[lightgray] (14,-2) -- (14,-1) -- (20,-1)--(20,-2);
  \fill[cyan] (13,-4) -- (13,2) -- (14,2)--(14,-4);
  \fill[cyan] (13,-4) -- (13,-3) -- (20,-3)--(20,-4);
  \fill[magenta] (12,-5) -- (12,2) -- (13,2)--(13,-5);
  \fill[magenta] (12,-5) -- (12,-4) -- (20,-4)--(20,-5);
  \fill[orange] (11,-8) -- (11,2) -- (12,2)--(12,-8);
  \fill[orange] (11,-8) -- (11,-7) -- (20,-7)--(20,-8);
  \fill[pink] (2,1) -- (2,2) -- (3,2)--(3,1);
  \fill[pink] (2,1) -- (2,2) -- (20,2)--(20,1);
  \fill[pink] (10,-9) -- (10,2) -- (11,2)--(11,-9);
  \fill[pink] (10,-9) -- (10,-8) -- (20,-8)--(20,-9);
   \fill[yellow] (19,-13) -- (19,2) -- (20,2)--(20,-13);
   \fill[yellow] (19,-13) -- (19,-12) -- (20,-12)--(20,-13);
   \fill[yellow] (9,-10) -- (9,2) -- (10,2)--(10,-10);
   \fill[yellow] (9,-10) -- (9,-9) -- (20,-9)--(20,-10);
   \fill[yellow] (7,-6) -- (7,2) -- (8,2)--(8,-6);
   \fill[yellow] (7,-6) -- (7,-5) -- (20,-5)--(20,-6);
   \fill[yellow] (1,0) -- (1,2) -- (2,2)--(2,0);
   \fill[yellow] (1,0) -- (1,1) -- (20,1)--(20,0);
   \fill[green] (6,-7) -- (6,2) -- (7,2)--(7,-7);
   \fill[green] (7,-7) -- (7,-6) -- (20,-6)--(20,-7);
   \fill[green] (5,-3) -- (5,2) -- (6,2)--(6,-3);
   \fill[green] (7,-3) -- (7,-2) -- (20,-2)--(20,-3);
  \fill[green] (0,-1) -- (0,2) -- (1,2)--(1,-1);
   \fill[green] (1,-1) -- (1,-0) -- (20,-0)--(20,-1);
   \fill[green] (8,-11) -- (8,2) -- (9,2)--(9,-11);
   \fill[green] (9,-11) -- (9,-10) -- (20,-10)--(20,-11);
  \fill[green] (18,-14) -- (18,2) -- (19,2)--(19,-14);
   \fill[green] (19,-14) -- (19,-13) -- (20,-13)--(20,-14);
	\draw[step=1cm] (0,-1) grid (20,2);
	\draw[step=1cm] (5,-3) grid (20,-1);
	\draw[step=1cm] (6,-7) grid (20,-1);
	\draw[step=1cm] (8,-11) grid (20,-1);
	\draw[step=1cm] (18,-14) grid (20,-1);
\node (v00) at (.5,-0.5) { {{\color{blue}$\bullet$}}};
\node (v10) at  (.5+1,0.5) {{ \color{blue}$\bullet$ }};
\node (v20) at (.5+2,1.5) {{ \color{blue}$\bullet$ }};
\node (v00) at (5.5,-2.5) { {{\color{blue}$\bullet$}}};
\node (v10) at  (5.5+9,-1.5) {{ \color{blue}$\bullet$ }};
\node (v00) at (6.5,-6.5) { {{\color{blue}$\bullet$}}};
\node (v10) at  (6.5+1,-5.5) {{ \color{blue}$\bullet$ }};
\node (v00) at (8.5,-10.5) { {{\color{blue}$\bullet$}}};
\node (v10) at  (8.5+1,-9.5) {{ \color{blue}$\bullet$ }};
\node (v00) at (8.5+2,-8.5) { {{\color{blue}$\bullet$}}};
\node (v10) at  (8.5+3,-7.5) {{ \color{blue}$\bullet$ }};
\node (v00) at (12.5,-4.5) { {{\color{blue}$\bullet$}}};
\node (v10) at  (13.5,-3.5) {{ \color{blue}$\bullet$ }};
\node (v10) at  (18.5,-13.5) {{ \color{blue}$\bullet$ }};
\node (v10) at  (19.5,-12.5) {{ \color{blue}$\bullet$ }};
\end{tikzpicture}
\]
\end{example}

We visualize by colors the algorithm suggested in the proof of Lemma~\ref{lem::staircorners}. Namely, first, we paint by green color the cells that belong to the row/columns of the corner cells of the Young diagram. Next, one looks on what remains and paint in yellow the hooks assigned to the corners of the Young diagram obtained after erasing the green cells. We continue to use a new color on each step and mark with blue points the centers of the staircase corner cells we find.
\end{example}

Let $\overline{n^{\tau}}:=(n^{\tau}_1,\ldots,n^{\tau}_{n_m})$ be the transposed Young diagram:
\begin{equation}\label{eq:transposed}
\begin{array}{c}
n^{\tau}_{1}:=\#\{i\colon n_i = n_m\}, \\
 n^{\tau}_2:=\#\{i\colon n_i\geq n_{m-1}\}, \\
 \ \ldots \,  \\
n^{\tau}_{n_m}:=\#\{i\colon n_i\geq 1\}=m. 
\end{array}
\end{equation}

\begin{lem}
The transposition of the Young diagram maps a staircase corner to a staircase corner:
$$
\left[\St_{\ov{n}}\right]^{\tau} = \St_{\ov{n^{\tau}}} \subset \bbY_{\ov{n^{\tau}}}.
$$
\end{lem}
\begin{proof}
The transposition of the Young diagram maps rows to columns and vice versa, and thus preserves corners. 
Hence, by the uniqueness of staircase corners (see the proof of Lemma \ref{lem::staircorners}), this set is preserved by transposition.
\end{proof}

Let us define a natural partial order on the set $\St_{\ov{n}}$ of staircase corners.
\begin{equation}
\label{eq::DL::order}
(i,j) \succeq (i',j')\ \stackrel{\mathsf{def}}{\Leftrightarrow} \ \begin{cases}
    i\geq i', \\
    j\leq j'.
\end{cases} 
\end{equation}
In words, to enlarge an element, one has to move in the down-left direction.

\begin{example}
Let us draw the Hasse diagram for the poset $\St_{\ov{n}}$ embedded into the Young diagram $\bbY_{\ov{n}}$ with $\ov{n}=(3,3,3,3,3,5,9,9,13,13,13,13,13,13,13,13,13,13,16,16)$
(see Example \ref{ex:large}):
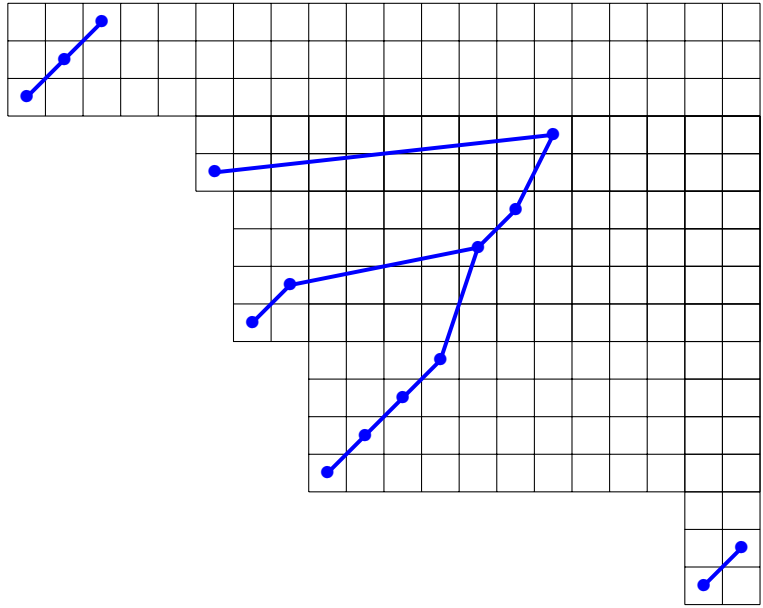
\[
\begin{tikzpicture}[scale=0.5]
	\draw[step=1cm] (0,-1) grid (20,2);
	\draw[step=1cm] (5,-3) grid (20,-1);
	\draw[step=1cm] (6,-7) grid (20,-1);
	\draw[step=1cm] (8,-11) grid (20,-1);
	\draw[step=1cm] (18,-14) grid (20,-1);

\node (v00) at (.5,-0.5) { {{\color{blue}$\bullet$}}};
\node (v10) at  (.5+1,0.5) {{ \color{blue}$\bullet$ }};
\node (v20) at (.5+2,1.5) {{ \color{blue}$\bullet$ }};
\draw[blue,line width=1.5pt] (.5,-0.5) -- (.5+1,0.5) -- (.5+2,1.5);

\node (v00) at (5.5,-2.5) { {{\color{blue}$\bullet$}}};
\node (v10) at  (5.5+9,-1.5) {{ \color{blue}$\bullet$ }};

\draw[blue,line width=1.5pt] (5.5,-2.5) -- (5.5+9,-1.5);

\node (v00) at (6.5,-6.5) { {{\color{blue}$\bullet$}}};
\node (v10) at  (6.5+1,-5.5) {{ \color{blue}$\bullet$ }};

\draw[blue,line width=1.5pt] (6.5,-6.5) -- (6.5+1,-5.5);

\node (v00) at (8.5,-10.5) { {{\color{blue}$\bullet$}}};
\node (v10) at  (8.5+1,-9.5) {{ \color{blue}$\bullet$ }};

\draw[blue,line width=1.5pt] (8.5,-10.5) -- (8.5+1,-9.5) -- (8.5+2,-8.5) -- (8.5+3,-7.5) -- (12.5,-4.5) -- (6.5+1,-5.5);

\node (v00) at (8.5+2,-8.5) { {{\color{blue}$\bullet$}}};
\node (v10) at  (8.5+3,-7.5) {{ \color{blue}$\bullet$ }};

\node (v00) at (12.5,-4.5) { {{\color{blue}$\bullet$}}};
\node (v10) at  (13.5,-3.5) {{ \color{blue}$\bullet$ }};

\draw[blue,line width=1.5pt] (12.5,-4.5) -- (13.5,-3.5) -- (5.5+9,-1.5);

\node (v10) at  (18.5,-13.5) {{ \color{blue}$\bullet$ }};
\node (v10) at  (19.5,-12.5) {{ \color{blue}$\bullet$ }};
\draw[blue,line width=1.5pt] (19.5,-12.5) -- (18.5,-13.5);
\end{tikzpicture}
\]
\end{example}

\medskip

Note that the condition that each column has at most one staircase corner implies that we have an embedding of the poset $\St_{\ov{n}}$ to a linearly ordered set:
\begin{equation*}
\begin{tikzcd}
\vrt: (\St_{\ov{n}},\prec) \ar[r,hook] &  ([1,m],>),
\end{tikzcd}
\end{equation*}
where for two integers $M_1,M_2$ we denote by $[M_1,M_2]$ the set of integer numbers $M$ such that  $M_1\le M\le M_2$.
Similarly, looking for the row index of a staircase corner we get another linearization of the poset $\St_{\ov{n}}$ 
\begin{equation*}
\begin{tikzcd}
\hor: (\St_{\ov{n}},\prec) \ar[r,hook] &  ([1,n_m],<).
\end{tikzcd}
\end{equation*}
Let us mention several combinatorial properties of the poset $\St_{\ov{n}}$ that easily follows from Definition~\ref{def::DL::indices}.
\begin{lem}
\label{lem::DL::posets}
\begin{enumerate}
\item
\label{item::DL::poset}
The Hasse diagram of the poset $(\St_{\ov{n}},\preceq)$ is a forest with smaller elements being closer to the root of a tree. 
In other words, for any $s\in \St_{\ov{n}}$ the set of elements that are smaller than $s$ contains the unique maximal element (if nonempty).
\item
\label{item::DL::linear}
For any $s\in \St_{\ov{n}}$ the subset $\St_{\ov{n}}\{{\succeq s}\}$ of elements that are greater than or equal to $s$ assemble an interval with respect to the linear order on $\vrt(\St_{\ov{n}})$:
$$
\forall s\in \St_{\ov{n}}\quad  \exists j_s\in [1,m] \  \colon \  \St_{\ov{n}}\{\succeq s \} = \vrt^{-1}([j_s,\vrt(s)]). 
$$
Similarly, for the second linearization:
$$
\forall s\in \St_{\ov{n}} \quad \exists i_s\in [1,n_m] \  \colon \  \St_{\ov{n}}\{\succeq s \} = \hor^{-1}([\hor(s),i_s]). 
$$
\item \label{item::DL::uncom}
For two uncomparable elements $s,t\in \St_{\ov{n}}$ one has 
$$\vrt(s)<\vrt(t) \in [1,m] \phantom{\frac{1}{2}} {\Leftrightarrow} \phantom{\frac{1}{2}} \hor(s){<}\hor(t) \in [1, n_m].
$$
\end{enumerate}
\end{lem}
\begin{proof}
Part \eqref{item::DL::poset} holds true by construction of the poset $(\St_{\ov{n}},\preceq)$. 

Let us prove the first statement in part \eqref{item::DL::linear} (the proof of the second statement is analogous).
We define $j_s=\min\vrt (\St_{\ov{n}}\{{\succeq s}\})$ and let
$r=\vrt^{-1}(j_s)$. Let $t\in \St_{\ov{n}}$ satisfy $j_s< \vrt(t)< \vrt(s)$. We claim that $t\succ s$. In fact, since 
$\vrt(t) < \vrt(s)$, if $t\not\succ s$, then 
$\hor(t)<\hor(s)$. Together with $\vrt(t) > \vrt(r)=j_s$, this implies $r\succ t$.
So $r\succ s$ and $r\succ t$, which gives that $s,t$ are comparable (by part \eqref{item::DL::poset}) and $t\succ s$.

Let us prove \eqref{item::DL::uncom}.  If $(i,j),(i',j') \in \St_{\overline n}$ are uncomparable, then by definition either $i<i',j<j'$ or $i>i',j>j'$. Then by definition of $\hor$ and $\vrt$ we get that the inequalities between $\vrt(i,j),\vrt(i',j')$ and $\hor(i,j),\hor(i',j')$ are the same.
\end{proof}

The following proposition describes the family of posets $\St_{\ov{n}}$.
\begin{prop}
    Suppose that $(\mathsf{S},\prec)$ is a poset whose Hasse diagram is a forest. Suppose moreover, we are given an order-preserving bijection $v:(\mathsf{S},\prec)\to ([1,m],>)$ such that all subsets $\mathsf{S}_{\preceq s}$ are $v$-preimages of certain intervals in $[1,m]$.
    Then there exists a partition $\ov{n}:=(n_1\le\ldots\le n_m)$ and an order-preserving bijection between $\mathsf{S}\stackrel{\simeq}{\rightarrow} \St_{\ov{n}}$.
\end{prop}
\begin{proof}
The algorithm that was used in the proof of Lemma~\ref{lem::staircorners} can be reversed.
Namely, let us recursively construct the following collection of partitions $\ov{n^k}=(n^k_k,\ldots,n^k_m)$ of increasing lengths. We set $\ov{n^m}:=(1)$ to be a partition of $1$. The decreasing in $k$ recurrence description is defined as follows:
$$
\begin{array}{c}
n^k_{k}:=\#\{s>k \colon v^{-1}(s)\succ k\}+1,\\
n^k_{i}:=n^{k+1}_{i}+1, \forall i=k_1,\ldots,m.
\end{array}
$$
The corresponding Young diagrams are embedded one in another, and at each step the difference $\bbY_{\ov{n^k}}\setminus\bbY_{\ov{n^{k+1}}}$ consists of one row and at the leftmost column. Thus, we add exactly one staircase corner on each step and it remains to verify that the partial order coincides with the given one.
\end{proof}

\subsection{DL-dense arrays}
\label{sec::DL_dense}
Let us introduce notation and combinatorial objects related to a Young diagram $\bbY_{\ov{n}}$. Our notation originates from similar constructions suggested by Danilov-Koshevoy in~\cite{DK1, DK2} known for rectangular diagrams and is read $D$ (for Down), $L$ (for left) dense arrays. 

\begin{definition}\label{def:arrays}
\begin{itemize}
\item 
A map $R:\bbY_{\ov{n}}\to \bZ_{\geq 0}$  attaching a non-negative integer number $R_{i,j}$ to each cell
$(i,j)\in \bbY_{\ov{n}}$ is called an array of shape $\bbY_{\ov{n}}$;
\item The total sum $|R|:=\sum_{(i,j)\in\bbY_{\ov{n}}} R_{ij}$ is called the degree of the array $R$;
\item The collection of sums of elements in each row is called the horizontal weight of an array:
    $$\hor(R):=\left(\sum_{j=1}^{m}R_{1j}, \ldots, \sum_{j=1}^{m} R_{n_mj} \right) \in \bZ^{n_m};$$
\item 
The \emph{vertical weight} of an array is defined as
    $$\vrt(R):=\left(\sum_{i=1}^{n_1} R_{i1}, \ldots, \sum_{i=1}^{n_m}R_{im} \right)\in \bZ^{m},$$    
i.e. as the collection of sums of elements in each column.
\end{itemize}      
\end{definition}

\begin{definition}
An array $R$ of Shape $\bbY_{\ov{n}}$ is called $DL$-dense if 
\begin{itemize}
\item $(i,j)\notin \St_{\ov{n}} \ \Rightarrow \ R_{ij}=0$;
\item if $(i,j)\prec (i',j')\in \St_{\ov{n}}$ then $R_{ij}\leq R_{i'j'}$.
\end{itemize}
\end{definition}

\begin{rem}
The set of $DL$-dense arrays is in one-to-one correspondence with the set of order preserving $\bZ_{\geq0}$-valued functions on the poset $(\St_{\ov{n}},\prec)$.
\end{rem}

\begin{dfn}
\label{def::DL::dense::set}
 For each partition ${\lambda}=(\lambda_1\geq\lambda_2\geq\ldots\geq \lambda_k)$, where $k=\#\St_{\ov{n}}$ let us denote by $\DL_{\ov{n}}({\lambda})$ the set of $DL$-dense arrays $R$ such that the multi-set $\{R_{s}\ \colon\ s\in\St_{\ov{n}}\}$ coincides
 with the multi-set $\{\la_i\}_{i=1}^k$.
\end{dfn}

The vertical and horizontal weights define a pair of embeddings:
$$ 
\begin{tikzcd}
    S_{n_m}{\lambda} \arrow[r,"\hor",hookleftarrow]  & \DL_{\ov{n}}({\lambda})
    \arrow[r,"\vrt",hookrightarrow] & S_{m}{\lambda},
\end{tikzcd}
$$
where $S_{n_m}{\lambda}$ (resp. $S_{m}{\lambda}$) is the $S_{n_m}$-orbit of the collection
$(\la_1,\dots,\la_k,\underbrace{0,\dots,0}_{n_m-k})$ (resp. $S_{m}$-orbit of $(\la_1,\dots,\la_k,\underbrace{0,\dots,0}_{m-k})$).

\begin{rem}
The set of all $DL$-dense arrays of shape $\bbY_{\ov{n}}$ of degree $N$ is the disjoint 
 union $\cup_{\lambda\vdash N}\DL_{\ov{n}}({\lambda})$.
\end{rem}

\subsection{Serpentines and DL-dense arrays}
\label{sec::serp::DL::compare}
The following Lemma is clear from the definition of DL-dense arrays.

\begin{lem}\label{lem:DLdenceSerpentine}
For $R\in DL_{\ov{n}}$ let $d_j=\vrt(R)_j$ (note that each column of $R$ contains at most one non-zero entry). Let $\mu^{(j)}_i=\sum_{a=1}^j R_{i,a}$. Then  $\mu^{(j)}\in\Ser^{n_j}_{d_j,\mu^{(j-1)}}$.
\end{lem}
\begin{proof}
If the $j$-th column does not contain a staircase corner, then $d_j=0$ and the claim is trivial. Assume that $(s,j)$ is a staircase corner. Then, for any $i>s$, $d_j \leq \mu_i^{(j-1)}$ by $DL$-dense condition. This guarantees the serpentine conditions.
\end{proof}

\begin{lem}\label{lem:SerDL}
Let $\ov{d}\in\bZ_{\ge 0}^m$ be an $\ov{n}$-admissible collection and assume that $d_j\ne 0$ for all $j$.
Let $\mu^{(j)}$, $1\le j\le m$  be defined by 
\[
\mu^{(0)}=(0), \ \mu^{(j)} \in \Ser_{d_j,\mu^{(j-1)}}^{n_j},\  \mu_+^{(j)} = (d_1,\dots,d_j)_+.
\]
Then $\St_{\ov{n}} = \{(i,j):\ \mu^{(j)}_i > 0, \mu^{(j-1)}_i= 0\}.$
\end{lem}
\begin{proof}
We note that according to Lemma \ref{lem::reordering} for any $j=1,\dots,m$ 
there exists at most one $i$ such that $\mu^{(j)}_i > 0, \mu^{(j-1)}_i= 0\}$.
Let us show that $\St_{\ov{n}}$ contains an element in the $j$-th column if and only if a number $i$ as above does exist, and in this case $(i,j)\in\St_{\ov{n}}$.

The proof goes by induction on the number of columns $m$. For $m=1$ the claim is obvious: $\mu_i^{(1)}\ne 0$ implies that  
$i=n_1$, $\mu^{(1)}_{n_1}=d_1$, which is compatible with $(n_1,1)\in\St_{\ov{n}}$.
Now, for arbitrary $m$, we note that $(n_1,1)\in\St_{\ov{n}}$ and all other elements of $\St_{\ov{n}}$ are obtained by removing a thin hook 
whose vertex is $(n_1,1)$ and then again considering the corners of the diagram thus obtained. Now the procedure of thin hook removal is compatible with the algorithm
from Lemma \ref{lem::reordering}. More precisely, since  $\mu^{(1)}_{n_1}=d_1>0$, all the entries $\mu^{(j)}_{n_1}$ are positive and hence do not affect
the sets $\{l:\ \mu^{(j)}_l>0\}\setminus \{l:\ \mu^{(j-1)}_l>0\}$.
\end{proof}

\begin{cor}\label{lem:zeroNonStaircase}
Assume that there is no staircase corner in the $i$-th row. Then for any 
$\overline n$-admissible  composition $\overline d$ we have 
$\hb(\overline d)_i=0$.
\end{cor}
\begin{proof}
As in the  above Lemma,  
we consider a collection $\mu^{(j)}$, $1\le j\le m$  such that   $\mu^{(0)}=(0)$, $\mu^{(j)} \in \Ser_{d_j,\mu^{(j-1)}}^{n_j}$,  $\mu_+^{(j)} = (d_1,\dots,d_j)_+$ (see Corollary \ref{cor:iteratedserp}). 
Assume that $\mu_i^{(j)}\neq 0,\mu_i^{(j-1)}= 0$ for some $j$. We need to show that there is a staircase corner in the $i$-th row.
If all $d_j$ are positive, then this is an immediate consequence of Lemma \ref{lem:SerDL}. If some of $d_j$ are zero, one removes from $\bbY_{\ov{n}}$ all columns corresponding to trivial entries of $\ov{d}$ and uses once again Lemma \ref{lem:SerDL} for the smaller Young diagram.
\end{proof}

\begin{lem}\label{lem:inequalityStaircase}
Let  $(i,j),(i',j')\in \St_{\ov{n}}$ with $(i,j)\succeq (i',j')$. Then 
$\hb(\ov{d})_i\geq \hb(\ov{d})_{i'}$ 
for any $\overline n$-admissible  composition $\overline d$.
\end{lem}
\begin{proof}
We first note that the assumption that $\ov{d}$ is $\overline n$-admissible is necessary to make the algorithm from Lemma \ref{lem::reordering} 
work. So in what follows we assume that this condition is satisfied.

We prove Lemma by induction on $m$, i.e. on the number of columns. The simplest non-trivial case is $m=2$. We have two option: $n_1<n_2$ and $n_1=n_2$. In the first case 
$\St_{\ov{n}}=\{(n_1,1),(n_2,2)\}$, so the two elements are not comparable and there
is nothing to prove. In the second case $\St_{\ov{n}}=\{(n_1,1),(n_1-1,2)\}$ and
$\hb(\ov{d})_{n_1}=\max(d_1,d_2)$, $\hb(\ov{d})_{n_1-1}=\min(d_1,d_2)$ with all
other entries being zero. 

Now let us consider the last column of $Y_{\ov{n}}$. Let $\ov{n}'=(n_1,\dots,n_{m-1})$.
First assume that $n_m>n_{m-1}$. Then 
$\St_{\ov{n}}\setminus \St_{\ov{n}'}=(n_m,m)$ and the element $(n_m,m)$ is uncomparable with any of the elements of $\St_{\ov{n}'}$. Hence, we are done by induction (since the first $m-1$ entries of $\hb(\ov{d})$ do not depend on the last column for $n_m>n_{m-1}$). 
Now assume that $n_m=n_{m-1}$. There are two option: either $\St_{\ov{n}}=\St_{\ov{n}'}$
or these two sets differ by one element located in the last column. 

Assume that $\St_{\ov{n}}=\St_{\ov{n}'}$. By definition, $\hb(\ov{d})$ is 
obtained by several steps of the procedure described in Lemma \ref{lem::reordering} 
(one starts with a number $d_1$ and puts it into the bottom of the first column, then proceeds
with $d_2$ and so on). So, for the case $\St_{\ov{n}}=\St_{\ov{n}'}$, using the induction 
assumption, it suffices to check that the last step (adding $d_m$) does not change the desired
inequalities between the contents of the elements of $\St_{\ov{n}'}$. Let 
$\ov{d}'=(d_1,\dots,d_{m-1})$ and let $\hb(\ov{d}')$ be constructed using $\ov{n}'$.
By induction, we know that $\hb(\ov{d}')_i\geq \hb(\ov{d}')_{i'}$ ($i>i'$ are as in the 
statement of our Lemma). Now inequality has a chance to change only in one of the following cases.
\begin{enumerate}[label=\alph*), itemsep=0pt, topsep=0pt]
\item $\hb(\ov{d}')_i=\hb(\ov{d})_i$, $\hb(\ov{d}')_{i'}\ne \hb(\ov{d})_{i'}$. \\
This means that at some stage of the procedure from Lemma \ref{lem::reordering} $\hb(\ov{d}')_{i'}$ got involved in the process. Since 
$i>i'$ and $\hb(\ov{d}')_i$ has not changed, procedure from Lemma \ref{lem::reordering} implies that $\hb(\ov{d})_i\geq \hb(\ov{d})_{i'}$ as desired.
 
 \item $\hb(\ov{d}')_i\ne \hb(\ov{d})_i$, $\hb(\ov{d}')_{i'}= \hb(\ov{d})_{i'}$. 
\\ Then $\hb(\ov{d}')_i\le \hb(\ov{d})_i$ and hence 
$\hb(\ov{d})_i\geq \hb(\ov{d})_{i'}$.

\item  $\hb(\ov{d}')_i\ne \hb(\ov{d})_i$, $\hb(\ov{d}')_{i'}\ne \hb(\ov{d})_{i'}$. \\
Recall that $i'<i$. Since the algorithm from Lemma \ref{lem::reordering} works from top to bottom (from rows with a smaller number to the rows with a larger number) the new values in  $\hb(\ov{d})$
increase from top to bottom. Hence  $\hb(\ov{d})_i\geq \hb(\ov{d})_{i'}$.
\end{enumerate}

In the following paragraph, it is important to distinguish between the operations $\hb$ corresponding to different shapes, so we use upper indices.
The remaining case to work out is $n_m=n_{m-1}$ and $\St_{\ov{n}}\ne\St_{\ov{n}'}$, i.e
the difference of the two posets consists of an element $(s,m)$ with $1\le s\le n_m$.
The only extra statement to be checked (compared with the previous case) is as follows:
let $(i,j)\in \St_{\ov{n}'}$, $i>s$ (since $j<m$, this implies  $(i,j)\succeq (s,m)$);
then $\hb_{\ov{n}'}(\ov{d})_i\geq \hb_{\ov{n}'}(\ov{d})_s$. The important thing to note is that 
$\hb_{\ov{n}'}(\ov{d}')_s=0$. In fact, suppose that $\hb_{{\ov{n}}}(\ov{d}')_s>0$ and let $r<m$ be the 
smallest number such that  $\hb_{{\ov{n}}}(d_1,\dots,d_r)_s>0$.
Then $(s,r)\in \St_{(n_1,\dots,n_r)}$, which contradicts the
assumption $(s,m)\in \St_{\ov{n}}\setminus\St_{\ov{n}'}$. Now since $\hb_{{\ov{n}'}}(\ov{d}')_s=0$
one gets $\hb_{{\ov{n}'}}(\ov{d}')_i\geq \hb_{{\ov{n}'}}(\ov{d}')_s$. The same arguments as in the previous
case (with $\St_{\ov{n}}=\St_{\ov{n}'}$) implies $\hb_{{\ov{n}}}(\ov{d})_i\geq \hb_{{\ov{n}}}(\ov{d})_s$.
\end{proof}

\begin{cor}\label{cor:transposedDLdense}
    For any $\overline n$-admissible  composition $\overline d$ there exists a DL-dense array $R$ such that 
    $\hor(R)=\hb(\ov{d})$.
\end{cor}

 The notation  $\hb(\overline d)$ refers to \emph{half-bubble-sort} and enjoys many nice properties.  
We plan to study the combinatorics of the ``bubble-sort" algorithm and its relationship with the Bruhat order elsewhere.
\begin{rem}

\label{rem::hbs} 
Note that if $\ov{d}$ is an $\ov{n}$-admissible collection, then $\hb_{\ov{n}}(\ov{d})$ is an admissible collection for the transposed diagram $\ov{n}^{\tau}$ (see~\eqref{eq:transposed} for the definition). Moreover, $\hb_{\ov{n}^{\tau}}(\hb_{\ov{n}}(\ov{d}))$ is a permutation of $\ov{d}$ which is a result of a full bubble-sort algorithm applied to $\ov{d}$ with respect to the partial order on $\vrt(\St_{\ov{n}})$.
Namely, we claim that:
\begin{equation}
\label{eq::hbs::hbs}
\begin{array}{rcl}
\left(\hb_{\ov{n}^{\tau}}(\hb_{\ov{n}}(\ov{d}))\right)_j\neq 0 & \Rightarrow & j\in \vrt(\St_{\ov{n}}); \\
s\preceq t \in \St_{\ov{n}} & \Rightarrow & \left(\hb_{\ov{n}^{\tau}}(\hb_{\ov{n}}(\ov{d}))\right)_{\vrt(s)} \leq \left(\hb_{\ov{n}^{\tau}}(\hb_{\ov{n}}(\ov{d}))\right)_{\vrt(t)}. 
\end{array}
\end{equation}
Moreover, $\hb_{\ov{n}^{\tau}}(\hb_{\ov{n}}(\ov{d}))$ is the minimal (with respect to the Bruhat order) permutation of $\ov{d}$ yielding properties~\eqref{eq::hbs::hbs}.
\end{rem}

\section{Demazure modules and standard filtrations}
\label{sec::Demazure::all}
\subsection{Highest weight categories and standard filtration}
\label{sec::HWC}
One of the key technical tools used in this paper is  
the theory of highest weight categories introduced first by Bershtein-Gelfand-Gelfand in~\cite{BGG} and later on generalized by Cline-Parshall-Scott in~\cite{CPS}.
This approach allows us to describe filtrations of various modules of the Borel
subalgebra with the subquotients being Demazure and van der Kallen modules.
We refer the reader to our previous papers~\cite{FKhMO,FKhM2} for the necessary 
background; here we give a brief summary.

Suppose that the abelian category $\calC$ satisfies the following properties:
\begin{itemize}
 \item $\calC$ enjoys the Krull-Schmidt property;
\item the multiplicity of each simple object in any object of $\calC$ is finite.
\end{itemize}
Assume that there is a lower finite partial order $\leq$ on the indexing set $\Lambda$ of the set of irreducibles in $\calC$ 
and let $\calC_{\leq\lambda}$ be the Serre subcategory spanned by the set of irreducibles $\{L(\mu)\colon \mu\leq \lambda\}.$

\begin{definition}
A projective cover of $L(\lambda)$ in $\calC_{\leq\lambda}$ is called the standard module $\Delta_{\lambda}$; 
an injective hull of $L(\lambda)$ in $\calC_{\leq\lambda}$ is called the costandard module $\nabla_{\lambda}$.
\end{definition}
We note that if standard and costandard modules exist, they are defined uniquely up to a (non-canonical) isomorphism.

\begin{definition}
The standard filtration $\cF^\lambda$ on $M \in \mathcal{C}$ is defined as follows: $\cF^\lambda$ is the smallest submodule of $M$ such that $M/\cF^{\lambda}(M)$ is (the maximal) quotient that belongs to $\calC_{\leq\lambda}$.
\end{definition}

\begin{definition}\label{def:ef}
An excellent filtration on an object $M \in \mathcal{C}$ is a decreasing filtration such that the graded pieces are isomorphic to standard modules.
\end{definition}

Let $\mathbb{P}_\lambda$ be the projective cover of $L(\lambda)$ in a category  $\calC$.

\begin{definition}
    A category $\calC$ with the aforementioned properties is called the highest weight category if, first, $\forall \lambda\in \Lambda$ there exist standard $\Delta_\lambda$ and costandard $\nabla_\la$ objects in $\calC_{{\leq \lambda}}$ and, second, one of the following properties is satisfied for all $\lambda\in\Lambda$:
    \begin{itemize}
        \item $\calC$ has enough projectives, and there exists an excellent filtration of the kernel of the map $\mathbb{P}_\lambda\rightarrow \Delta_\lambda$ such that all subquotients are isomorphic to $\Delta_{\mu}$ with $\mu>\lambda$;
        \item the fully faithful embedding of abelian categories $\imath_{\lambda}:\calC_{{\leq}\lambda} \hookrightarrow \calC$ extends to a fully faithful embedding of the corresponding derived categories:
        $\imath_{\lambda}: \calD(\calC_{{\leq}\lambda}) \hookrightarrow \calD(\calC);$
        \item $\forall \lambda, \mu\ \ \dim\Ext_{\calC}^n(\Delta_\lambda,\nabla_\mu)=\delta_{n,0} \, \delta_{\lambda,\mu}.$
    \end{itemize}
\end{definition}
The aforementioned properties are equivalent if the category $\calC$ has enough projectives. In our main example, this is not completely true, because
the modules we consider are infinite-dimensional; however, they are inverse limits of finite-dimensional modules, and one can easily upgrade the setting to this case (see e.g.~\cite{CPS, vdK, BS}).

 The following Lemma is an easy corollary of the general formalism of highest weight categories (see e.g.~\cite[\S3]{Kh}):
\begin{lem}\label{lm::excellent}
    A module $M$ of a highest weight category $\calC$ admits an excellent filtration $\calF$ iff $\Ext^1(M,\nabla_{\lambda})=0$ for all $\lambda\in\Lambda$.
    Moreover, the number $[M:\Delta_\lambda]$ of subquotients isomorphic to $\Delta_\lambda$ is equal to the dimension $\dim\Hom_{\calC}(M,\nabla_{\lambda})$.
\end{lem}
\begin{cor}\label{lem:standexc}
    Assume that a category $\calC$ is a highest weight category. Then if $M \in \calC$ admits an excellent filtration, then its standard filtration is excellent.
\end{cor}
\begin{proof}
The claim follows from the homological criterion for the existence of an excellent 
filtration from Lemma~\ref{lm::excellent}. We note that the criterion is the same 
for an arbitrary excellent filtration and for the standard one.
\end{proof}

\begin{lem}\label{lm::TensoProductDecomposition}
Suppose that a category ${\calC}$ admits a monoidal product $\otimes$.
Let $D$ be a module such that {$D \otimes$} is an exact functor and for a given subset $\Theta\subset \Lambda$ 
one has:
$$
\forall \lambda\in\Theta \ D\otimes\Delta_\lambda \text{ admits an excellent filtration.}
$$
Then for any $M \in \calC$ which admits an excellent filtration whose subquotients are
isomorphic to $\Delta_\lambda,\lambda \in \Theta$, the module $D \otimes M$ admits an excellent filtration and, moreover, we have the following equality for multiplicities:
    \[
    [D \otimes M:\Delta_{\lambda_1}]=\sum_{\la_2} [M:\Delta_{\lambda_2}] \cdot [D \otimes\Delta_{\lambda_2}:\Delta_{\lambda_1}].\]
\end{lem}

\begin{proof}
Assume that the module $M$ admits an excellent filtration of length more than one. Then it admits the short exact sequence
\[0 \rightarrow M_1 \rightarrow M \rightarrow M_2 \rightarrow 0,\]
where $M_1$, $M_2$ admit excellent filtrations by $\Delta_\lambda, \lambda \in \Theta$ of shorter length.
Therefore, by the exactness we have
\[0 \rightarrow D \otimes M_1 \rightarrow D \otimes M \rightarrow D \otimes M_2 \rightarrow 0.\]
Thus, we obtain our claim by induction on the length of the standard filtration.
\end{proof}

Note that the tensor product of a Lie algebra module over a field is exact.

\subsection{Cherednik order and standard modules}
\label{sec::Cher::standard}
Let $\fb_n,\fb_{n,-}\subset\mgl_n$ be the Lie subalgebras of upper-triangular and 
lower-triangular
matrices and let $\fh_n=\fb_n\cap\fb_{n,-}$ be the diagonal Cartan subalgebra.
The space $\fh_n$ is spanned by the diagonal matrix units, the dual basis is denoted by $\varepsilon_i$.
{We denote by $\Phi$ the root system of the Lie algebra $\mgl_n$.}
One has  the decomposition $\Phi=\Phi_+\sqcup \Phi_-$ with  $\Phi_+=\{\varepsilon_i-\varepsilon_j, i<j\}$ 
and the simple roots are given by $\alpha_i=\varepsilon_i-\varepsilon_{i+1}$.
We denote by $P$ the weight lattice for $\mgl_n$ 
and let $P_+\subset P$ consists of weights $\lambda = \sum_{i=1}^n \lambda_i \varepsilon_i$ such that 
$\la_i\in\bZ_{\ge 0}$ and $\la_i\ge \la_{i+1}$ for all $i$. 
The Weyl group $W$ is isomorphic to the symmetric group $S_n$. 

For each $\la\in P_+$ 
let $V_\la$ be the irreducible  highest weight $\fg$
module with highest weight $\la$. We denote by $v_\la\in V_\la$ the highest weight vector. One has $V_\la=\U(\fb_{n,-})v_\la$, where $\U(\fb_{n,-})$ is the universal 
enveloping algebra. The space $V_\la$ has a weight decomposition with respect to the $\fh_n$ action and $v_\la$ spans the weight $\la$ subspace of $V_\la$.

\begin{rem}
    The weight $\lambda = \sum_{i=1}^n \lambda_i \varepsilon_i$ can be considered as the weight of each $\mgl_m, m \geq n$. The corresponding composition in this situation is equal to $(\lambda_1,\dots,\lambda_n,0,\dots,0)$.
\end{rem}

\begin{definition}
\label{def::Cherednik::order}
The Cherednik order $\prec$ on the set of compositions can be explicitly written as follows. 
For a composition $\la$, we write $\la_+$ for the partition in the $W$ orbit of $\la$. Then for two compositions
$\la=(\la_1,\dots,\la_n)$ and $\mu=(\mu_1,\dots,\mu_n)$ we write $\la\succeq \mu$ if 
\begin{itemize}
\item $\la_1+\dots +\la_n=\mu_1+\dots +\mu_n$;
\item $(\la_+)_1 + \dots + (\la_+)_r\ge (\mu_+)_1 + \dots + (\mu_+)_r$ for $1\le r\le n$;
\item if $\la_+=\mu_+$ and $\la=\tau_\la (\la_+)$,
$\mu=\tau_\mu (\mu_+)$, then $\tau_\mu \preceq  \tau_\lambda$ in the Bruhat order.
\end{itemize}
\end{definition}

\begin{rem}
We note that the first two conditions in Definition \ref{def::Cherednik::order} mean that $\la_+\ge \mu_+$ 
in the dominance order. 
\end{rem}

Recall that the Bruhat partial order on the Weyl group $W$ is the minimal order generated by the following inequalities:
$$
\text{ if for }\alpha\in \Phi_+, \sigma\in W \text{ we have }l(s_\alpha\sigma)=l(\sigma)+1 \text{ then } \sigma \preceq s_{\alpha}\sigma.
$$
In particular, the identity element is the minimal element and the unique longest element $w_0$
is the maximal element in $W$
(we refer to~\cite{BB} for the details on the Bruhat order).

The group $W$ acts naturally on the weight lattice.
For any given weight $\lambda\in P$ there exists a unique element $\sigma_\lambda\in W$ of minimal length such that $\lambda=\sigma_{\lambda}(\lambda_+)$ with $\lambda_+\in P_+$.
Let $V_{\la_+}$ be the irreducible finite-dimensional $\gl$-module with highest weight $\lambda_+$.
The $\lambda$-weight subspace of $V_{\lambda_+}$ is one-dimensional,
we denote by $v_{\la}$ a generator of this space (i.e.  $v_{\la}\in V_{\la_+}$ is an extremal weight vector). The Demazure module $D_{\la}\subset V_{\la_+}$
is defined as $\U(\fb_n)v_{\la}$.
In particular, we have 
$$\sigma\preceq\tau \ \Rightarrow \ v_{\sigma\lambda_+}\in D_{\tau\lambda_+} \ \Rightarrow \ D_{\sigma\lambda_+} \subset D_{\tau\lambda_+} 
$$
{The following theorem is due to Joseph \cite{J}. An alternative proof is found in \cite{Po}, Proposition 2.1: }
\begin{thm}[\cite{J}]
\label{thm::Joseph}
For any $\lambda\in P$ the relations below are defining 
for the Demazure module $D_\lambda$ (here $e_\alpha\in\fb_n$ is the weight $\al$ Chevalley generator):
\begin{equation}\label{thmJosepf}
e_\alpha^{\max\{-\langle \alpha^\vee,\lambda \rangle,0\}+1}v_\lambda=0, \alpha \in \Phi_+.
\end{equation}
\end{thm}
We will also need the van der Kallen modules $K_\la$ which are $\fb_n$-modules defined by
\[
K_\la = D_\la\Bigl/\Bigr.\sum_{\mu\prec\la} D_\mu = 
D_\la\Bigl/\Bigr.\sum_{\substack{\mu\in W\lambda\\ D_{\mu}\subsetneq D_{\la}}} D_\mu. 
\]
These modules were defined by van der Kallen in~\cite{vdK}.
We note that $K_\la$ is cyclic $\fb_n$-module defined by the following defining relations:
\begin{equation}
e_\alpha^{\max\{-\langle \alpha^\vee,{\lambda} \rangle,0\}}v_\lambda=0, \alpha \in \Phi_+.
\end{equation}

{In this paper, we deal only with Lie algebras of type $A$, whose root systems are self-dual, so  
we use the identification $\lambda^\vee =\lambda$ to simplify the notation.
In what follows, we also use right (opposite) Demazure modules $D_\la^{op}$ and right (opposite) 
van der Kallen modules $K_\la^{op}$. 
Both $D_\la^{op}$ and $K_\la^{op}$ are right cyclic modules of $\fb_m$.
As vector spaces, one has
\[
D_\la^{op}\simeq D_{w_0\la},\ K_\la^{op}=K_{w_0\la}
\]
(the weight is twisted by the longest Weyl group element $w_0$)
and the right $\fb_{m}$ action is induced by the anti-automorphism of $\fb_{m}$ defined by $w_0$.
Explicitly, a matrix unit $E_{ij}\in\fb_m$ acts on $D_\la^{op}$ (resp. $K_\la^{op}$) in the same way as  $E_{m+1-j,m+1-i}$ acts on $D_{w_0\la}$ (resp. $K_{w_0\la}$).

The main object of this paper is the space of staircase matrices $\Mat_{\ov{n}}$ that have left $\fb_{n_m}$ and right $\fb_{m}$ actions (see Section~\ref{sec::bimodule}), thus we are interested in the categories $\mathcal{C}_n$ and $\mathcal{C}^{op}_n$ of finite dimensional left and right representations of the Borel subalgebra $\fb_n\subset\mgl_n$, such that $\fh_n$ acts on the objects semi-simply with 
nonnegative integer eigenvalues.
Note that irreducible objects in both $\calC_n$ and $\calC_n^{op}$ are one-dimensional representations $\{\Bbbk_{\lambda},\lambda\in \bZ_{\geq 0}^n\}.$

The following theorem is proved in \cite{vdK}, Theorem 1.6 (see also \cite{FKhMO}):
\begin{thm}[\cite{vdK}]\label{lem:standard}
 The categories $\mathcal{C}_n$ and $\mathcal{C}^{op}_n$ have the highest weight structures with respect to the Cherednik order. The standard objects in  $\mathcal{C}_n$ are the Demazure modules, and the standard objects in  $\mathcal{C}^{op}_n$ are the van der Kallen modules.
\end{thm}
\begin{rem}
The sets of simple objects in $\mathcal{C}_n$ and in $\mathcal{C}_n^{op}$ are indexed by the same weight lattice $P$, and we are fixing the same Cherednik order for both
categories. A simple module $L(\la)$ in $\mathcal{C}_n$ is made into a simple module 
in $\mathcal{C}_n^{op}$ by changing $\la$ to $-\la$ (the right action is the negated left action). 
According to Theorem~\ref{lem:standard}, the standard objects 
in $\mathcal{C}_n$ and in $\mathcal{C}_n^{op}$ are different. The reason is as follows: if $\la\prec \mu$ in the Cherednik order, then $-\la\prec -\mu$ in the dual Cherednik
order, which differs from the Cherednik order by reversing inequality in the Bruhat comparison (the last comparison in Definition~\ref{def::Cherednik::order}).
\end{rem}

For a $\fh_n$ module $V$ and $\mu\in P$ let $V(\mu)\subset V$ be the
weight $\mu$ subspace. We define 
\[
\ch V (x_1,\dots,x_n)= \sum_{\mu=(\mu_1,\dots,\mu_n)} x_1^{\mu_1}\dots x_n^{\mu_n} \dim V(\mu). 
\]
In particular, for a partition $\la_+$ the character of the corresponding irreducible highest weight $\mgl_n$ module $V_{\la_+}$ is equal to the Schur polynomial $s_{\la_+}(x_1,\dots,x_n)$.
The Demazure modules $D_\la$ and the van der Kallen module $K_\la$ are labeled by compositions (arbitrary weight $\lambda\in P$).
One has 
\[
\ch D_\la(x_1,\dots,x_n) = \kappa_\la(x_1,\dots,x_n),\qquad
\ch K_\la(x_1,\dots,x_n) = a_\la(x_1,\dots,x_n),
\]
where $a_\la(x)$ and $\kappa_\la(x)$ are the Demazure atoms and key polynomials (see \cite{Al,P} and references therein; in particular, the identification of the Demazure characters with key polynomials goes back to Demazure \cite{Dem2}).

\begin{example}
If $\la$ is anti-dominant (i.e. $\lambda_1\leq \lambda_2\leq \ldots \leq \lambda_n$), then the Demazure module $D_{\la}$ coincides with the whole irreducible module $V_{\lambda_+}$. Hence  $\kappa_{\la}(x)$ is equal to the Schur polynomial $s_{\la_+}(x)$.
If $\la=\lambda_+$ is dominant, then $D_{\la_+}$ is one-dimensional and $\kappa_{\la_+}(x)=x^{\la_+}=a_{\la_+}(x)$.
\end{example}

Note that we deal with left and right modules over the Borel subalgebra.
In particular, the Cartan subalgebra $\fh_m\subset\fb_{m}$ acts on the opposite Demazure and van der Kallen modules, and we obtain the opposite key polynomials and opposite Demazure atoms:
\[
\kappa^{\la}(y_1,\dots,y_m)=\ch D_\la^{op},\ 
a^{\la}(y_1,\dots,y_m)=\ch K_\la^{op}.
\]
By definition, one has 
\begin{equation}\label{eq:left-right}
\kappa^{\la}(y_1,\dots,y_m)=\kappa_{w_0\la}(y_m,\dots,y_1),\ a^{\la}(y_1,\dots,y_m)=a_{w_0\la}(y_m,\dots,y_1).
\end{equation}
Here $w_0$ is the longest element in the Weyl group. In the case of a symmetric group, it reverses the composition:
$$w_0(\lambda_1,\ldots,\lambda_n)=(\lambda_n,\ldots,\lambda_1).$$
Note that both $\kappa^{\la}(y_1,\dots,y_m)$ and $a^{\la}(y_1,\dots,y_m)$ contain the term $y^\la$, which is the leading term for the dual Cherednik ordering in these polynomials.

\subsection{Pieri rule for Demazure modules}
\label{sec::Pieri::Demazure}

Consider the vector representation $V_{\varepsilon_1}$ of the Lie algebra $\mgl_n$. It contains the set $v_i$, $i=1, \dots,n$ of extremal vectors of weight $\varepsilon_i$ (which form the basis of this module). These vectors generate the following Demazure submodules:
\[D_{\varepsilon_1}\subset D_{\varepsilon_2}\subset\dots\subset D_{\varepsilon_n}.\]
Note that $D_{\varepsilon_n}\simeq Res_{\fb_n}^{\mgl_n}V_{\varepsilon_1}$. 

For a composition $\lambda=(\lambda_1,\dots,\lambda_n)$ the Demazure $\fb_n$-module $D_\lambda$ is generated by the extremal weight vector $v_\lambda$. Let
\[{\rm m}_\lambda=\{i = 1,\dots, n,\ \not\exists j{>}i, \lambda_j=\lambda_i\}.\]

\begin{lem}
\label{lem::D:D_epsilon}
The set $\{v_\lambda \otimes v_{i} \colon i \in {\rm m}_{\lambda}\}$ generates the $\fb_n$-module $D_\lambda \otimes D_{\varepsilon_n}$.
\end{lem}
\begin{proof}
    Assume that $j \notin {\rm m}_\la$. Then there exists $i>j$ such that $\lambda_i=\lambda_j$ and $i \in {\rm m}_\la$.
Therefore,
    \[v_\lambda \otimes v_{\varepsilon_j}=E_{ji}(v_\lambda \otimes v_{\varepsilon_i}).\]
    Thus, the set $\{v_\lambda \otimes v_{i} \colon i \in {\rm m}_{\lambda}\}$ generates $v_\lambda\otimes D_{\varepsilon_n}$ and, consequently generates the whole space $D_\lambda\otimes D_{\varepsilon_n}$.
\end{proof}

The following inequalities follow directly from Definition~\ref{def::Cherednik::order} of the Cherednik partial order:
$$
\forall i,i' \in {\rm m}_\la \text{ we have } \lambda + \varepsilon_i\preceq \lambda + \varepsilon_{i'} \ \Leftrightarrow \ \lambda_{i'}>\lambda_i.$$ 
Let us consider the components of the standard filtration of the module $D_\lambda \otimes D_{\varepsilon_n}$ with respect to the order $\prec$. We define 
$\fb_n$-submodules 
$\mathcal{G}^i\subset D_\la\T D_{\varepsilon_n}$ for $i \in {\rm m}_\la$ by the formula
\[
\mathcal{G}^i=\sum_{\lambda_{i'}\ge \lambda_i}\U(\fb_n)(v_\lambda \otimes v_{{i'}}).
\]
\begin{prop}\label{filtrationOneBoxProduct}
\begin{equation}\label{eq:G/G}
\mathcal{G}^i\left/\sum_{{\{i' \in \rm{m}_\lambda:\lambda_{i'}>\lambda_i\}}} \mathcal{G}^{i'}\right.\simeq D_{\lambda +\varepsilon_i}.
\end{equation}
    In particular, $D_\lambda \otimes D_{\varepsilon_n}$ admits an excellent filtration (see Definition \ref{def:ef}).
\end{prop}
\begin{proof}
Thanks to Lemma~\ref{lem::D:D_epsilon} we know that \eqref{eq:G/G} is a cyclic $\fb_n$-module generated by cyclic vector $v_\lambda\otimes v_i$.
Let us check the relations from Joseph's Theorem~\ref{thm::Joseph}. Take a vector $v_\lambda \otimes v_{i}$ of weight $\lambda + \varepsilon_i$. For the root $\alpha =\varepsilon_j-\varepsilon_{j'}$, $i \notin \{j,j'\}$ we have
    $\max\{-\langle \alpha,\lambda \rangle,0\}=\max\{-{\langle \alpha,\lambda+ \varepsilon_i \rangle},0\}$, and hence
\[
e_\alpha^{\max\{-\langle \alpha,\lambda+\varepsilon_i \rangle,0\}+1}(v_\lambda\otimes v_{i})=0.
\]
The same holds if $j=i$. Consider the case $j'=i$. If $\max\{-\langle \alpha,\lambda\rangle,0\}=0$, then the claim clearly holds. 
Assume that $\max\{-\langle \alpha,\lambda \rangle,0\}>0$, that is, $\lambda_j>\lambda_i$. 
We have
 \[e_\alpha^{-\langle \alpha,\lambda+\varepsilon_i \rangle}(v_\lambda\otimes v_{i})=(e_\alpha^{-\langle \alpha,\lambda+\varepsilon_i \rangle}v_\lambda)\otimes v_{i}=e_\alpha^{-\langle \alpha,\lambda+\varepsilon_i \rangle+1}(v_\lambda\otimes v_{j})\in \mathcal{G}^{i'},\]  
  where $i'$ is maximal such that $\lambda_{i'}=\lambda_j$. 
{Indeed, $v_\lambda\otimes v_{j}=e_{\varepsilon_j-\varepsilon_
   {i'}}(v_\lambda\otimes v_{i'})$}.
Therefore,  $\mathcal{G}^i\left/\sum_{\lambda_{i'}>\lambda_i} \mathcal{G}^{i'}\right.$ satisfies the defining relations of $D_{\lambda +\varepsilon_i}$. However, due to~\eqref{AQ} we have the coincidence of characters:
   \[\kappa_\lambda(x_1,\dots,x_n) s_{\varepsilon_1}(x_1,\dots,x_n)=\sum_{i \in {\rm m}_\la}\kappa_{\lambda+\varepsilon_i}(x_1,\dots,x_n).\]
   Therefore, all the subquotients $\mathcal{G}^i\left/\sum_{\lambda_{i'}>\lambda_i} \mathcal{G}^{i'}\right.$ are in fact isomorphic to Demazure modules.
\end{proof}

\begin{prop}\label{StandardFiltrationTensor}
 Let $\la$ be a composition and let  $\ell$ be the maximal index of a nonzero entry of $\la$. The following tensor products admit excellent filtrations:
 \begin{itemize}
     \item $D_\lambda \otimes D_{\varepsilon_n}\otimes\dots  \otimes D_{\varepsilon_n}$ with $\ell\le n$;
     \item $D_\lambda \otimes S^d D_{\varepsilon_n}$ with $\ell\le n$ and all $d\ge 0$;
     \item $D_\lambda \otimes S^{d_1}D_{\varepsilon_{n_1}}  \otimes\dots \otimes S^{d_m}D_{\varepsilon_{n_m}}$ with $\ell \leq n_1 \leq \dots \leq n_m$ and any collection of nonnegative integers $d_1,\dots,d_m\ge 0$.
 \end{itemize}
\end{prop}
\begin{proof}
    The first claim follows by induction from Proposition \ref{filtrationOneBoxProduct} and Lemma \ref{lm::TensoProductDecomposition}.

The module $S^d D_{\varepsilon_n}$ is a direct summand in $D_{\varepsilon_n}^{\otimes d}$. Therefore, the module $D_\lambda \otimes S^d D_{\varepsilon_n}$ is a direct summand of
$D_\lambda \otimes D_{\varepsilon_n}^{\otimes d}$. Hence, it has an excellent filtration by the first claim.

The third claim follows from the second by induction.
\end{proof}
   
Consider now the natural inclusion $\fb_n\subset \fb_{n'}$ for $n'>n$. 
For any composition $\lambda$ of length not exceeding $n$ the restriction of the Demazure module $D_\lambda$ over $\fb_{n'}$ to $\fb_n$  is isomorphic to $D_\lambda$ over $\fb_n$. We use the notation $D_\lambda$ without mentioning the number $n$.
We call the description of the standard filtration on $D_\lambda \otimes S^d D_{\varepsilon_n}$ the Pieri rule for Demazure modules. 

Our next goal is to give some details on the combinatorics of the Pieri rule for Demazure modules. Recall that on the level of characters (key polynomials), the corresponding combinatorics is discussed in Section~\ref{subsection:Serpentines} in terms of serpentines (Definition~\ref{def: serp}). 
Let us assign to each $d$-serpentine $\mu$ associated with $\lambda$ the vector in the symmetric power of the Demazure module $D_{\varepsilon_n}$:
$$v^{\mu-\lambda}:=v_{1}^{\mu_1-\lambda_1}\cdot\ldots\cdot v_{n}^{\mu_n-\lambda_n}\in S^d(D_{\varepsilon_n}).$$

\begin{lem}
The monomials $v_\lambda \otimes v^{\mu-\lambda}$, $\mu\in\Ser^{n}_{d,\la}$,  generate the $\fb_n$-module $D_\lambda \otimes S^d D_{\varepsilon_n}$.
\end{lem}
\begin{proof}
Let us denote by $M$ the submodule generated by $v_\lambda\otimes v^{\mu-\lambda}$ with $\mu$ from the set of $d$-serpentines of $\lambda$.
Let us show that $M$ contains the space $v_\lambda \otimes S^d D_{\varepsilon_n}$.
The proof is by induction on the degree $d$ of symmetric power and on the number $0\leq n'\leq n$ of variables such that $v_{{s}}$, $s\leq n'$ do not appear in the monomials from $S^d(D_{\varepsilon_n})$ we are dealing with. 
The induction hypothesis is as follows: for all $d'\leq d$ and a given $n'\leq n$ we have:
$$\forall \ov{b}=(0,\ldots,0,b_{n'+1},\ldots,b_n) \colon |\ov{b}|=d \text{ one has } v_\lambda \otimes v^{\ov{b}} \in M \subset D_\lambda\otimes S^d(D_{\varepsilon_n}).$$
Suppose that there exists $n''>n'$ such that $\lambda_{n''}\leq\lambda_{n'}$. Then $E_{n',n''}v_\lambda=0$ and hence 
$$
v_{\lambda}\otimes v_{{n'}}^{b_{n'}}\cdot\ldots\cdot v_{n}^{b_n} = (E_{n'n''})^{b_{n'}}
\left(v_\lambda \otimes  
v_{{n'+1}}^{b_{n'+1}}\cdot\ldots\cdot v_{{n''}}^{b_{n''}+b_{n'}}
\cdot\ldots v_{n}^{b_n} \right).
$$
It remains to look at the case when $\forall n''>n'$ we have $\lambda_{n'}<\lambda_{n''}$.
   In this case $n' \in {\rm m}_\lambda$ by the definition of ${\rm m}_\lambda$.
Let $a=\min \{\lambda_{n''}, n''>\lambda\}-\lambda_{n'}$. Then by the induction assumption, all the elements of the form $v_\lambda \otimes \prod_{i\in n',\dots,n}v_{i}^{b_i}$, $b_{n'}\leq a$, belong to the module under consideration. Take $n''$ such that $\lambda_{n''}-\lambda_{n'}=a$. Then $E_{n',n''}^{a+1}v_{\lambda}=0$. Take $b_{n'}>a$. Then the element
$v_{\lambda}\otimes v_{{n'}}^{b_{n'}}\cdot\ldots\cdot v_{n}^{b_n}$ lies in the linear span of elements
\[(E_{n'n''})^{b_{n'}-i}
\left(v_\lambda \otimes  
v_{{n'}}^{i}\cdot\ldots\cdot v_{{n''}}^{b_{n''}+b_{n'}-i}
\cdot\ldots v_{n}^{b_n} \right),~ i=0, \dots,a\]
by the standard $\msl_2$ argument. This completes the proof.
\end{proof}

Note that the weight of the element $v_{\lambda} \otimes v^{\mu-\lambda}$ is equal to $\mu$ and we define the following filtration  on $D_\lambda \otimes S^d D_{\varepsilon_n}$ indexed by the set of serpentines:
\[
\mathcal G^{\mu}=\sum_{\mu' \succeq \mu, \mu' \in \Ser^n_{d,\la}}\U(\mathfrak{b}_n)(v_\lambda \otimes v^{\mu'-\lambda}),\ \mu\in \Ser^n_{d,\la}.
\]

\begin{thm}\label{thm:PieriRule}
(Pieri rule for Demazure modules).
    \[\mathcal{G}^{\mu}\left/ \sum _{\mu'\succ\mu,\mu'\in \Ser^n_{d,\la}}\mathcal{G}^{\mu'}\right.\simeq D_{\mu}.\]
\end{thm}
\begin{proof}
According to Proposition {\ref{prop:serp}} we have equality~\eqref{AQ} of characters:
\[\ch(D_\lambda \otimes S^d D_{\varepsilon_n})=\sum_{\mu \in \Ser^n_{d,\la}}\ch D_\mu.\]
By Proposition \ref{StandardFiltrationTensor} the module 
$D_\lambda \otimes S^d D_{\varepsilon_n}$ admits an excellent filtration. Therefore, the subquotients of this filtration are isomorphic to 
$D_\mu, \mu \in \Ser^n_{d,\la}$. Let
\[
(\mathcal{G}')^\mu=\sum_{\substack{x \in D_\lambda\otimes S^d D_{\varepsilon_n}\\ \wt(x)\succeq \mu}}\U(\fb_n)x
\]
be the component of the standard filtration on $D_\lambda\otimes S^dD_{\varepsilon_n}$.
By definition, we have $(\mathcal{G}')^\mu \supset \mathcal{G}^\mu$. Assume that for some $\mu \in \Ser^n_{d,\la}$ $(\mathcal{G}')^\mu \neq \mathcal{G}^\mu$ and $(\mathcal{G}')^{\mu'} = \mathcal{G}^{\mu'}$ for all $\mu' \succ \mu$. Then we have a natural injection 
\[\mathcal{G}^{\mu}\left/ \sum _{\mu'\succ\mu,\mu'\in \Ser^n_{d,\la}}\mathcal{G}^{\mu'}\right.\hookrightarrow D_{\mu}.\]

However, by the constructions both modules $\mathcal{G}^{\mu}\left/ \sum _{\mu'\succ\mu,\mu'\in \Ser^n_{d,\la}}\mathcal{G}^{\mu'}\right.$ and $D_{\mu}$ are cyclic generated by the element of weight $\mu$.
Therefore, this injection is in fact a bijection which contradicts the assumption. Therefore, $(\mathcal{G}')^{\mu'} = \mathcal{G}^{\mu'}$ for all $\mu'\in \Ser^n_{d,\la}$.
\end{proof}

We consider the modules $D_{\varepsilon_{n_i}}$ for $0<n_1\leq \dots\leq n_m$. Let 
$\{v_{1,l},\dots,v_{n_l,l}\}$ be the basis of 
$D_{\varepsilon_{n_l}}$. As before with each $d$-serpentine $\mu$ associated with $\lambda$, we assign a vector in the symmetric power of the Demazure module
$$v_l^{\mu-\lambda}:=v_{1,l}^{\mu_1-\lambda_1}\cdot\ldots\cdot v_{n_l,l}^{\mu_{n_l}-\lambda_{n_l}}\in S^{d}(D_{\varepsilon_{n_l}}).$$

\begin{cor}
\label{ProductFiltrationBasis}
    The elements
\begin{equation} \label{eq:elmenv}
v_1^{\mu^{(1)}}v_2^{\mu^{(2)}-\mu^{(1)}}\dots v_k^{\mu^{(m)}-\mu^{(m-1)}},
\end{equation}
where $\mu^{(i+1)}\in\Ser^{n_{i+1}}_{d_{i+1},\mu^{(i)}}$,
form a minimal generating set of the associated graded to the standard filtration of the module 
\begin{equation}\label{eq:tensS}
S^{d_1}D_{\varepsilon_{n_1}}  \otimes S^{d_2}D_{\varepsilon_{n_2}}  \otimes\dots \otimes S^{d_m}D_{\varepsilon_{n_m}}.
\end{equation}
Moreover, the classes of the elements $v_1^{\mu^{(1)}}v_2^{\mu^{(2)}-\mu^{(1)}}\dots v_k^{\mu^{(m)}-\mu^{(m-1)}}$ generate $D_{\mu^{(m)}}$.
\end{cor}
\begin{proof}
By Corollary \ref{lem:standexc}, the standard filtration on \eqref{eq:tensS} is excellent. 
Theorem \ref{thm:PieriRule} says that there exists an excellent filtration on 
 \eqref{eq:tensS} such that in the associated graded module  
each element \eqref{eq:elmenv} generates the corresponding Demazure module $D_\mu$. However, by definition,
the element \eqref{eq:elmenv} in the associated graded space to the standard filtration generates a quotient of $D_\mu$. Hence, comparing the characters, vector \eqref{eq:elmenv}  generates exactly $D_\mu$ and all these vectors 
together generate the whole (graded) tensor product \eqref{eq:tensS}.
\end{proof}

\section{Bimodule of staircase matrices and right standard filtration}
\label{sec::right::Cauchy}
\subsection{Bimodule of staircase matrices}
\label{sec::bimodule}
We continue to use notation introduced in Section~\ref{sec::combinatorics} and fix a non-decreasing collection of integer numbers $\overline{n}:=(n_1\leq \ldots \leq n_m)$ with $n_1>0$.
The Young diagram $\bbY_{\overline{n}}$ defines the space of staircase matrices $\Mat_{\overline{n}}$ consisting  of linear functions $A:\bbY_{\overline{n}}\to\Bbbk$. In other words, $\Mat_{\overline{n}}$ is the subspace of
rectangular $n_m\times m$ matrices whose entries vanish outside $\bbY_{\overline{n}}$. 

The space $\Mat_{\overline{n}}$ is acted from the left by the Borel subalgebra $\fb_{n_m}$ and from the right by the Borel subalgebra $\fb_m$ of upper triangular matrices (by left and right multiplication). 
These actions commute, so we obtain a bimodule structure. In what follows we are interested in the bimodule $S(\Mat_{\overline{n}})$ (the symmetric space of $\Mat_{\ov{n}}$).

Let $v_{i,j}\in \Mat_{\overline{n}}$, $(i,j)\in \bbY_{\overline{n}}$ be the matrix units forming a basis of $\Mat_{\overline{n}}$ (i.e. $v_{i,j}$ sends the cell $(i,j)$ to $1$ and all other elements to zero). The space
$S^\bullet(\Mat_{\overline{n}})$ is spanned by elements   of the form $\prod_{i,j} v_{i,j}^{r_{i,j}}$ for all possible collections
of nonnegative integers $r_{i,j}$.

\begin{rem}
Let $E_{i,j}\subset \fb$ be a matrix unit from the upper triangular Borel subalgebra. Then 
\[
E_{i,j} v_{a,b} = \delta_{j,a} v_{i,b},\qquad v_{a,b} E_{i,j} = \delta_{b,i} v_{a,j}.
\]
\end{rem}

Recall (see Definition~\ref{def:arrays}) that a collection $R$ of nonnegative integers $r_{i,j}$, $(i,j)\in \bbY_{\overline{n}}$ is called an array.
For an array $R$ we denote 
$$v^R=\prod_{i,j} v_{i,j}^{r_{i,j}}\in S^\bullet(\Mat_{\overline{n}}).$$
One has the following equalities for the action of the diagonal matrix units on the symmetric algebra:
    \[E_{ii}v^R=(\hor(R))_i v^R,\ v^R E_{ii}=(\vrt(R))_i v^R.\]
\begin{example}
 Let $\overline{n}=(1,2,3)$  and all $r_{i,j}=1$, $(i,j)\in \bbY_{\overline{n}}$. Then 
 \[
 v^R = v_{1,1}v_{1,2}v_{1,3}v_{2,1}v_{2,2}v_{3,1},\ 
 E_{ii}v^R= (4-i)v^R,\ v^RE_{i,i} = iv^R.
 \]
\end{example}
    
\begin{dfn}        
The left $\fh_n$-weight  $\lambda$ subspace 
$_\lambda S(\Mat_{\overline{n}})$ of $S(\Mat_{\overline{n}})$  is spanned by the elements $v^R$ with $\hor(R)=\lambda$.
Respectively, the right $\fh_n$-weight  ${\overline d}$ subspace $ S(\Mat_{\overline{n}})_{\overline d}$ of $S(\Mat_{\overline{n}})$  is spanned by the elements $v^R$ with $\vrt(R)={\overline d}$.
The notation  ${_\lambda S}(\Mat_{\overline{n}})_{\overline d}$ is used for the left-right weight $(\lambda,{\overline d})$ subspace of $S(\Mat_{\overline{n}})$.
\end{dfn}

We have the following decomposition of the left $\fb_{n_m}$-modules
\[S(\Mat_{\overline{n}})=\bigoplus_{\overline d} S(\Mat_{\overline{n}})_{\overline d}.\]
Indeed, the left and the right actions commute. Therefore, each right-weight component is closed under the left action. In the same way, we have the decomposition of the right $\fb_m$-modules
\[S(\Mat_{\overline{n}})={\bigoplus_\lambda} {_\lambda S(\Mat_{\overline{n}})}.\]

\subsection{Excellent filtrations for the left and right actions}
\label{sec::left-right-fibration}
The goal of this subsection is to study the basic properties of the left and right standard filtrations on the bimodule $S(\Mat_{\overline{n}})$. By definition, we have the following left and right standard filtrations with respect to the Cherednik order:
\begin{gather}\label{LeftFIltrationBimodule}
^\lambda\mathcal{F}:=\U(\fb_{n_m})\left({\bigoplus_{\mu {\not \preceq} \lambda}}{_\mu S(\Mat_{\overline{n}})}\right);\\
\label{RightFIltrationBimodule}
\mathcal{F}^{\overline d}:=\left({\bigoplus_{{\overline c} {\not\preceq} {\overline d}}}{S(\Mat_{\overline{n}})_{\overline c}}\right)\U(\fb_{m}).
\end{gather}
In the above formulas, 
the direct sums are taken over all $\mu$ and  ${\overline c}$, which are not less than or equal to $\la$ and ${\overline d}$, respectively. 
Note that the upper left and right indices label the filtration spaces, and the lower (left and right) indices indicate the weight decomposition.

We write explicitly the right weight decomposition of the spaces  $ ^\lambda\mathcal{F}$ and the left weight decomposition of the spaces $\mathcal{F}^{\overline d}$:
\[
(^\lambda\mathcal{F})_{\overline d}:=\U(\fb_{n_m})\left({\bigoplus_{\mu {\not\preceq} \lambda}}{_\mu S(\Mat_{\overline{n}})}_{\overline d}\right),\qquad
_\lambda(\mathcal{F}^{\overline d}):=\left({\bigoplus_{\overline c {\not\preceq} \overline d}}{_\lambda S(\Mat_{\overline{n}})_{\overline c}}\right)\U(\fb_m).
\]
We note that the right-hand sides are the left standard filtration of the space 
$S(\Mat_{\overline{n}})_{\overline d}$ and the right standard filtration of the space  ${_\lambda S}(\Mat_{\overline{n}})$.

By definition, we have the following isomorphism:
\begin{equation}
    S(\Mat_{\overline{n}})_{\overline d}\simeq S^{d_1}D_{\varepsilon_{n_1}}  \otimes S^{d_2}D_{\varepsilon_{n_2}}  \otimes\dots \otimes S^{d_m}D_{\varepsilon_{n_m}}.
\end{equation}

Therefore, Proposition \ref{StandardFiltrationTensor} tells us that the summand $S(\Mat_{\overline{n}})_{\overline d}$ admits the left excellent filtration.
Corollary  \ref{ProductFiltrationBasis} can be reformulated in the following way:
\begin{prop}
    The monomial elements
\[v_1^{\mu^{(1)}}v_2^{\mu^{(2)}-\mu^{(1)}}\dots v_k^{\mu^{(m)}-\mu^{(m-1)}}, \text{ such that } \forall j \ \mu^{(j+1)}\in\Ser^{n_{j+1}}_{d_{j+1},\mu^{(j)}}\]
    form a basis of the standard filtration of the left module $S(\Mat_{\overline{n}})_{\overline d}$.
\end{prop}

Now consider the filtrations $^\lambda\mathcal{F}$, $\mathcal{F}^{\overline d}$. The next lemma tells us that these filtrations are, in fact, bimodule filtrations.

\begin{lem}\label{lem:standardfiltrationbimodule}
The left $\fb_{n_m}$-module ${^\lambda}\mathcal{F}$ is closed under the right $\fb_m$-action, i. e. ${^\lambda}\mathcal{F} \U(\fb_m)={^\lambda}\mathcal{F}$.  
The right $\fb_m$-module $\mathcal{F}{^{\overline d}}$ is closed under the left $\fb_{n_m}$-action, 
i. e. $\U(\fb_{n_m})\mathcal{F}{^{\overline d}}=\mathcal{F}{^{\overline d}}$.
\end{lem}
\begin{proof}
Each subspace ${S(\Mat_{\overline{n}})_{\overline d}}$ is the right submodule. Therefore, the left module generated by the sum of these subspaces is closed under the right action. This completes the proof of the first part of the lemma. The proof of the second part is similar.
\end{proof}

\subsection{Right standard filtration}
Proposition \ref{StandardFiltrationTensor} tells us that the module $S(\Mat_{\overline n})$ as the left module admits a filtration whose subquotients are Demazure modules.
Our goal here is to study the structure of $S(\Mat_{\overline n})$ as a right module. 

Recall that the Borel subalgebra $\fb_m$ of upper triangular matrices acts on $S(\Mat_{\overline n})$ by right multiplication.
To a collection ${\overline n}$ we attach a transposed collection $\ov{n}^{\tau}$ of lengths of rows of the corresponding Young diagram (see \eqref{eq:transposed}).
We note that the Borel subalgebra $\fb_{m}$ acts on 
$\Mat_{{\overline n}^{\tau}}$ by the left multiplication.

Given a matrix 
$A\in \Mat_{\overline n}$ we define the (dual) matrix $A'\in \Mat_{{\overline n}^{\tau}}$ by the formula
$A'_{i,j}=A_{n_m-j+1,m-i+1}$. 
We introduce a similar notation for the elements $x\in\fb_{m}$. Namely, we denote by $x'$ the element  with the entries $x'_{i,j}$ given by
$x'_{i,j}=x_{m-j,m-i}$
Then for $x\in\fb_{m}$ and  $A\in \Mat_{\overline n}$ one has $(Ax)'=x' A'$.
Hence, Proposition \ref{StandardFiltrationTensor} implies the following Corollary.

\begin{cor}
 The right  $\fb_m$-module $S(\Mat_{\overline n})$ admits a filtration by $\fb_m$-modules such that the graded pieces of the associated
 graded space are isomorphic to the right Demazure modules. 
\end{cor}

\begin{prop}
The right standard filtration on  $S(\Mat_{\overline n})$ is excellent.
\end{prop}
\begin{proof}
    Each Demazure module has a filtration by van der Kallen modules. Therefore, by the previous corollary, we get that $S(\Mat_{\overline n})$ admits a filtration by right $\fb_{m}$-modules such that the graded pieces of the associated graded module are isomorphic to van der Kallen modules. However, by Theorem \ref{lem:standard}, the standard modules for the right action are the van der Kallen modules. Therefore, $S(\Mat_{\overline n})$ admits a right excellent filtration.
Finally, by Corollary \ref{lem:standexc}, the existence of an excellent filtration implies that the standard filtration is excellent. 
\end{proof}

By Lemma \ref{lem:standardfiltrationbimodule} the components of the right standard filtration $\mathcal{F}^{\overline d}$ are in fact $\fb_{n_m}\de\fb_{m}$-bimodules. Therefore, the subquotients $\mathcal{F}^{\overline d}\left/ \sum_{{\overline c}\succ {\overline d}}\mathcal{F}^{\overline c} \right.$ of this filtration also have the structure of 
$\fb_{n_m}\de\fb_{m}$-bimodules. However, the right standard filtration on  $S(\Mat_{\overline n})$ is excellent. Therefore, we have
\begin{equation}\label{eq:RightFiltrationSubquotient}
\mathcal{F}^{\overline d}\left/ \sum_{{\overline c} \succ {\overline d}}\mathcal{F}^{\overline c} \right. \simeq D \otimes K_{\overline d}^{op} 
\end{equation}
where the multiplicity space $D$ has the structure of 
$\fb_{n_m}$- module. Our next goal is to study this module.

\begin{lem}\label{lem:multiplicity}
The multiplicity space $D$ as a left $\fb_{n_m}$-module admits the following description: 
\[
D \simeq S(\Mat_{\overline n})_{\overline d}\left/ S(\Mat_{\overline n})_{\overline d}\bigcap \sum_{\overline c\succ \overline d}\mathcal{F}^{\overline c} \right. .
\]
\end{lem}
\begin{proof}
The ${\overline d}$ weight space of the van der Kallen module $K_{\overline d}^{op}$ is one-dimensional. Therefore, the module $D$ is isomorphic to the space of the elements of the right weight ${\overline d}$ in the subquotient on the left-hand side of  \eqref{eq:RightFiltrationSubquotient}.
\end{proof}

Consider an array $R$ of shape $ \bbY_{\ov{n}}$, $\vrt(R)=\overline{d}$.
We define compositions  $\mu^{(j)}$, $j=1,\dots,m$ by the formula 
$\mu^{(j)}:=\left(\sum_{l=1}^{j}R_{1l}, \ldots, \sum_{l=1}^{j} R_{n_ml} \right)$.

\begin{lem}\label{lem:hnev}
Assume that 
$\hor(R)_+\neq { \overline d}_+$ and
that $\mu^{(j)} \in \Ser_{d_j,\mu^{(j-1)}}^{n_j}$, $j=1,\dots,m$.
Then
\[v^R \in S(\Mat_{\overline n})_{\ov{d}} \bigcap \sum_{\overline c\succ \overline d}\mathcal{F}^{\ov c}.\] 
\end{lem}
\begin{proof}
We denote $\hor(R)=\lambda$. Recall the dual Young diagram $\bbY_{\overline{n}^\tau}$, where 
$\overline{n}^\tau=(n^\tau_1,\dots,n^\tau_{n_m})$.
We consider the following right $\fb_m$-module
 \[_\lambda S(\Mat_{\overline n})\simeq S^{\lambda_1}(D^{op}_{\varepsilon_{n_1^\tau}})\otimes \dots \otimes S^{\lambda_{n_m}}(D^{op}_{\varepsilon_{n^\tau_{n_m}}}).\]
For a composition $a=(a_1,\dots,a_{n^\tau_i})$ we denote by 
$(v'_i)^a$ the monomial $\prod_{j=1}^{n^\tau_i} v_{ij}^{a_j}$.
  By Corollary \ref{ProductFiltrationBasis}
the elements
\[(v_1')^{\overline c^{(1)}}(v_2')^{ \overline c^{(2)}-\overline c^{(1)}}\dots (v_{n_m}')^{\overline c^{(n_m)}-\overline c^{(n_m-1)}},\]
where {$\overline c^{(i+1)}\in \Ser_{\lambda_{i+1},\overline c^{(i)}}^{n_{i+1}}$}, form a basis of this module. However by Corollary \ref{cor:weightinequality} {applied twice,
$\overline c^{(n_m)}_+\geq \lambda_+\geq \overline{d}_+$ and by the assumption of
the statement of the lemma, $\lambda_+> \overline{d}_+$}. Therefore, all the elements of the horizontal weight whose dominant part is greater than $\overline{d}_+$ (in the dominance order) lie in the right module generated by all elements of the vertical weight greater than $\overline{d}$ (in the Cherednik order). This completes the proof.
\end{proof}

Recall Definition \ref{def:AdmissibleDefinition} of $\overline{n}$-admissible compositions.

\begin{cor}\label{cor:zero}
 If $\overline{d}$ is not $\ov{n}$-admissible, then  
$\mathcal{F}^{\overline{d}} = \sum_{\overline{c} \succ \overline{d}}\mathcal{F}^{\overline{c}}.$  
\end{cor}
\begin{proof}
The space 
$S(\Mat_{\overline n})_d$ is generated by the monomials $v^R=v_1^{\mu^{(1)}}\dots v_m^{\mu^{(m)}-\mu^{(m-1)}}$ for 
an iterated $\overline{d}$ serpentine $\mu^{(1)},\dots,\mu^{(m)}$ by Corollary \ref{ProductFiltrationBasis}.
Let $l$ be the minimal number such that $n_l < \#\{i=1,\dots,l:\ d_i > 0\}$; in particular, $(\mu^{(l)})_+\ne (d_1,\dots,d_l)_+$ (the right-hand side has more non-zero entries than the left-hand side). 
Then   Corollary \ref{cor:SubsequenceInequality} implies that $\mu^{(m)}_+\ne \overline{d}_+$. Hence, by Lemma \ref{lem:hnev}  $v^R\in \sum_{\overline{c} \succ \overline{d}}\mathcal{F}^{\overline{c}}$. 
Now Lemma \ref{lem:multiplicity} implies that $D=0$ and hence the desired claim follows from \eqref{eq:RightFiltrationSubquotient}. 
\end{proof}

\begin{cor}\label{cor:surjmult}
  If $\overline{d}$ is $\ov{n}$-admissible, then there exists a surjection $D_{\hb(\overline{d})}\rightarrow D$ (see \eqref{eq:RightFiltrationSubquotient}).
\end{cor}
\begin{proof}
{As in the previous proof, we take the monomials $v^R=v_1^{\mu^{(1)}}\dots v_m^{\mu^{(m)}-\mu^{(m-1)}}$ for 
an iterated $\overline{d}$ serpentine $\mu^{(1)},\dots,\mu^{(m)}$.} Then by Corollary \ref{cor:weightinequality} and {Corollary \ref{cor:iteratedserp}}, all of them except for one of weight $\hb(\overline{d})$ have horizontal weights $\mu$ such that $\mu_+>\overline d_+$. 
Therefore, by Lemma  \ref{lem:hnev} all these $v^R$ lie in 
    $S(\Mat_{\overline n})_{\overline d} \bigcap \sum_{\overline c\succ \overline d}\mathcal{F}^{\overline c}$. 
Therefore, by Lemma \ref{lem:multiplicity} $D$ is generated by one element of weight $\hb(\overline{d})$. 
\end{proof}

\begin{thm}\label{thm:RSF}
 If $\ov{d}$ is $\ov{n}$--admissible, then  
 \[\mathcal{F}^{\overline{d}}\left/ \sum_{\overline{c} \succ \overline{d}}\mathcal{F}^{\overline{c}} \right. \simeq D_{\hb(\ov{d})} \otimes K_{\overline d}^{op} .\]
 If $\ov{d}$ is not $\ov{n}$--admissible, then 
 $\mathcal{F}^{\overline{d}} = \sum_{\overline{c} \succ \overline{d}}\mathcal{F}^{\overline{c}}.$
\end{thm}
\begin{proof}
Corollary \ref{cor:zero} takes care of the second case, so let us assume that $n_l \geq \#\{i=1,\dots,l:\ d_i >0\}$
for any $l \leq m$. Thanks to Corollary \ref{cor:surjmult}, it suffices to show that the surjection  $D_{\hb(\ov{d})}\twoheadrightarrow D$ is an isomorphism.

We  show the following 
\begin{equation}\label{eq:Dmu}
D\simeq  S(\Mat_{\overline n})_{\overline d}\left/ S(\Mat_{\overline n})_{\overline d} \bigcap \sum_{\overline c_+>\overline d_+}\mathcal{F}^{\overline c} \right.
\twoheadrightarrow D_{\hb(\ov{d})}.
\end{equation}
Together with the surjection $D_{\hb(\ov{d})}\twoheadrightarrow D$, this
would imply the desired isomorphism.

Let us first prove the existence of the surjection in \eqref{eq:Dmu}.
Let $\overline c_+>\overline d_+$. Then
the left module $\mathcal{F}^{\overline c}$ is generated by elements of the vertical weights $\overline b$ such that $\overline b_+ \not < \overline c_+$ and, hence, 
$\overline b_+ \not < d_+$.
Therefore 
$S(\Mat_{\overline n})_{\overline d} \bigcap \sum_{\overline c_+>\overline d_+}\mathcal{F}^{\overline c}\subset \U(\fb_{n_m}) (\sum_{\overline b_+ \not \leq \overline d_+}S(\Mat_{\overline n})_{\ov{b}})$.
Hence, by Corollary\ref{cor:weightinequality} the space
$\sum_{\overline c_+>\overline d_+}\mathcal{F}^{\overline c}$ is generated as {\it left} module by the elements of weights $b$, such that $b_+ \not \leq d_+$. However, $ S(\Mat_{\overline n})_{\overline d}$ admits a standard filtration. Therefore, there exists a surjection
\[ S(\Mat_{\overline n})_{\overline d}\left/ S(\Mat_{\overline n})_{\overline d} \bigcap \sum_{\overline c_+>\overline d_+}\mathcal{F}^{\overline c} \right.\twoheadrightarrow D_{\hb(\ov{d})} .\]

Now let us prove the isomorphism in \eqref{eq:Dmu}.
By Lemma \ref{lem:multiplicity} we have 
\[
D \simeq S(\Mat_{\overline n})_{\overline d}\left/ S(\Mat_{\overline n})_{\overline d} \bigcap \sum_{\overline c\succ\overline d}\mathcal{F}^{\overline c} \right..
\]
Since $\overline c\succ\overline d$, either $\overline c_+>\overline d_+$ or $\overline c_+=\overline d_+$.
Let us consider the case  $\overline c_+= \overline d_+$. By equation \eqref{eq:RightFiltrationSubquotient}, we have
\[
\mathcal{F}^{\overline c}\left/ \sum_{{\overline c'} \succ {\overline c}}\mathcal{F}^{\overline c'} \right.  \simeq D' \otimes K_{\ov c}^{op}.
\]
Therefore, for any $\ov c \succ \ov d$, $\ov c_+= \ov d_+$, the module $\mathcal{F}^{\overline c}\left/ \sum_{{\overline c'} \succ {\overline c}}\mathcal{F}^{\overline c'} \right.  $  does not contain elements of the vertical weight $\overline d$, because the van der Kallen module $K_{\overline c}$ does not contain elements of weights $\sigma \overline c$, $1 \neq \sigma \in W$. Therefore
\begin{equation}\label{eq:D}
D \simeq S(\Mat_{\overline n})_{\overline d}\left/ S(\Mat_{\overline n})_{\overline d} \bigcap \sum_{\overline c_+>\overline d_+}\mathcal{F}^{\overline c} \right..
\end{equation}
\end{proof}

\begin{cor}
    The associated graded module for $S(\Mat_{\overline n})$ is isomorphic to 
    \[gr( S(\Mat_{\overline n})) \simeq \bigoplus_{\overline d \text{ is }\overline n \text{-admissible}}D_{\hb(\ov{d})}\otimes K_{\overline d}^{op}.\]
\end{cor}

The comparison of characters gives us the following.
\begin{cor}\label{cor:firstCauchy}
The following ``right" Cauchy identity holds:
\[
\prod_{(i,j)\in \bbY_{\overline{n}}}\frac{1}{1-x_iy_j}=
\sum_{\overline d \text{ is }\overline n \text{-admissible}}
\kappa_{\hb(\ov{d})}(x)\,a^{\overline d}(y).\]
\end{cor}

\section{Left standard filtration and Cauchy identities}
\label{sec::left::Cauchy}
The goal of this section is to study the left standard filtration of the module $S(\Mat_{\overline{n}})$.
\subsection{Generators of the left standard filtration}
\label{sec::left::generators}

We start with a classical and simple Lemma which is $2 \times 2$ case of the Howe duality \cite{Ho}. Let $\Mat_2$ be the space of two by two matrices and let us consider the standard $\mgl_2\de \mgl_2$ action on the symmetric algebra $S(\Mat_2)$. 
We denote by $u_{11},u_{12},u_{21},u_{22}$ the standard generators of $S(\Mat_2)$.
For a two by two array $A$ (a matrix with non-negative integer coefficients) we set 
$u^A=\prod_{i,j=1}^2 u_{ij}^{A_{ij}}$.
Finally, let $\fb_2\subset\mgl_2$ be the Borel subalgebra of upper triangular matrices.
\begin{lem}
    \label{lem:gl_2HoweComputations}
 Let  $(a_1,a_2)=(A_{11}+A_{21}, A_{12}+A_{22})_+$. Then $u^A$ belongs to the 
 $\U(\fb_2)$-$\U(\fb_2)$-bimodule generated from the vectors 
$u^{A^{(l)}}$, $0\le l\le a_2$, where  
$A^{(l)}=\begin{pmatrix} 0 & a_2-l\\ a_1+l & 0 \end{pmatrix}$.
\end{lem}
\begin{proof}
By Howe duality, $S^N(\Mat_2)=\bigoplus_{b \geq c, b+c =N} V_{b \varepsilon_1+ c \varepsilon_2}\otimes V_{b \varepsilon_1+ c \varepsilon_2}^{op}$.
Moreover,
\begin{equation}\label{eq:2x2}
\bigoplus_{l =0}^{a_2} V_{(a_1+l) \varepsilon_1+ (a_2-l) \varepsilon_2}\otimes V_{(a_1+l) \varepsilon_1+ (a_2-l)\varepsilon_2}^{op}=\sum_{l = 0}^{a_2}\U(\fb_2) u_{21}^{a_1+l}u_{12}^{a_2-l} \U(\fb_2).
\end{equation}
In fact, both the right-hand side and the left-hand side of \eqref{eq:2x2} are $\mgl_2-\mgl_2$-bimodules, so it suffices to show that 
\[
E_{12}^l u_{21}^{a_1+l}u_{12}^{a_2-l}E_{12}^l, \quad 0\le l\le a_2
\]
are linearly independent (since the linear independence of these vectors would imply that
the tensor products of the highest weight vectors
of summands of the left hand side of \eqref{eq:2x2} belong to the right hand side). 
Now, let us consider the terms $u^B$ showing up in
$E_{12}^l u_{21}^{a_1+l}u_{12}^{a_2-l}E_{12}^l$.
Then the maximal value of $B_{11}$ is equal to $l$ and thus moving from $l=a_2$ to $l=0$, one shows that
a vanishing linear combination should be trivial.

Therefore, any vector $u^B \in S^N(\Mat_2)$ such that its vertical weight satisfies $\vrt(B)_+ \geq (a_1,a_2)$ is contained in 
$\sum_{l = 0}^{a_2}\U(\fb_2) u_{21}^{a_1+l}u_{12}^{a_2-l} \U(\fb_2).$ In particular, since $\vrt(A)_+ = (a_1,a_2)$, we arrive at the statement of the Lemma. 
\end{proof}

Recall (see Definition \ref{def::DL::dense::set}) that for any partition $\lambda$ we defined a subset of $DL$-dense arrays $\DL_{\overline n}(\lambda)$. The following 
lemma claims that monomials corresponding to arrays
supported on $\St_{\ov{n}}$ can be expressed in terms of $DL$-dense arrays. {Recall that the set $\St_{\ov{n}}$ has the poset structure given by equality \eqref{eq::DL::order}}.
Let $\mathbb{Z}^{\St}(\lambda)$ be the set of arrays $A\in\Mat_{\ov{n}}$ such that, first, $A_{ij}$ vanishes unless $(i,j)\in\St_{\ov{n}}$, and, second, the multi-set  $\{A_{ij}, (i,j)\in\St_{\ov{n}}\}$ coincides with the multi-set $\{\la_i\}_i$ assigned to the partition $\lambda$.

{In the following lemma the letters $\lambda, \mu$ denote partitions rather than compositions.}

\begin{lem}\label{lem:vAvB}
Assume $A \in \mathbb{Z}^{\St}(\lambda)$.
Then 
  \[v^A \in \sum_{B \in \DL_{\overline n}(\mu), \mu \geq \lambda} \U(\fb_{n_m})v^B\U(\fb_{m}).\]
\end{lem}
\begin{proof}
Assume that lemma holds for all $A \in \mathbb{Z}^{\St}(\nu)$ with $\nu > \lambda$.
Let us first show the following statement. Assume that $A_{ij}<A_{i'j'}$ for $(i,j)\succ(i',j')$, $(i,j),(i',j')\in \St_{\ov{n}}$, and consider the arrays $R^l$, $l=0,\dots,A_{ij}$, such that $R_l$ differs from $A$ at 
exactly $(i,j),(i',j')$ and  $R^l_{i,j}=A_{i'j'}+l$, $R^l_{i'j'}=A_{ij}-l$. Then 
\begin{equation}v^A\in \sum_{l=0}^{A_{ij}}\U(\fb_{n_m})v^{R^l}\U(\fb_{m})\subset \U(\fb_{n_m})v^{R^0}\U(\fb_{m})+\sum_{B \in \DL_{\overline n}(\mu), \mu > \lambda}\U(\fb_{n_m})v^B\U(\fb_{m})\label{eq:simpleSwitch}
\end{equation}
by Lemma \ref{lem:gl_2HoweComputations} {for $A_{11}=A_{22}=0$} and the above assumption.

For $A,B\in \mathbb{Z}^{\St}(\lambda)$ we say that $B$ differs from $A$ by a simple switch if there exist two 
comparable elements $(i,j), (i',j')\in\St_{\ov{n}}$ such that 
$B_{ij}=A_{i'j'}$, $A_{ij}=B_{i'j'}$ and $A,B$ coincide otherwise. Then Equation \eqref{eq:simpleSwitch} tells us that if $A(1)$ differs from $A$ by simple switch in $(i,j)\succ(i',j')$ and $A_{ij}<A_{i'j'}$, then:

\[v^{A(1)} \in \sum_{B \in \DL_{\overline n}(\mu), \mu \geq \lambda} \U(\fb_{n_m})v^B\U(\fb_{m}) \implies v^{A(1)} \in \sum_{B \in \DL_{\overline n}(\mu), \mu \geq \lambda} \U(\fb_{n_m})v^B\U(\fb_{m}) .\]

Then it remains to prove that for any $A\in \mathbb{Z}^{\St}(\lambda)$  
there exists a sequence of simple  switches $A=A(0)\to A(1)\to\dots\to A(N)$ such that $\hor(A(\bullet+1))\prec\hor(A(\bullet))$ and $A(N)\in\DL_{\ov{n}}$. In fact, if $A\in \DL_{\ov{n}}$ then there is nothing to prove. 
If $A\notin \DL_{\ov{n}}$ then there exists a pair $(i,j), (i',j')\in\St_{\ov{n}}$,
$(i,j)\succ(i',j')$ such that $A_{ij}<A_{i'j'}$. We make a simple switch at these elements and obtain $A(1)$. If $A(1)\in \DL_{\ov{n}}$, we are done. If not, we proceed as above. It remains to show that 
the procedure will terminate at some $A(N)\in\DL_{\ov{n}}$. In fact, let us define the number of disorders of $A\in \mathbb{Z}^{\St}(\lambda)$  as the number of pairs $(i,j)\succ(i',j')$ such that $A_{ij}<A_{i'j'}$
(in particular, the number of disorders of a DL-dense array is zero). Then for any $A\in \mathbb{Z}^{\St}(\lambda)$ the number of disorders of $A(1)$ is smaller than that of $A$. A simple switch interchanges the values
of only two sites, and a disorder corresponding to this pair disappears.
{Now one easily sees that after the switch, the number of disorders in other pairs 
either stays the same or becomes smaller.}
\end{proof}

\begin{prop}\label{prop:BimoduleGenerators}
The monomials $v^A$, $A\in \DL_{\ov{n}}$  of total weight $N$, generate the $\fb_{n_m}\de\fb_{m}$-bimodule $S^N(\Mat_{\overline{n}})$.
\end{prop}
\begin{proof}
By Lemma \ref{lem:vAvB}, it is enough to prove our Proposition for a larger set of monomials, namely, for $A\in \bZ^{\St}$ (i.e., for the arrays $A$ supported on the poset $\St_{\ov{n}}$) {of total weight $N$}.

We prove 
by induction on the size of the Young diagram $\bbY_{\overline{n}}$.
Let us take $j$ such that $n_j>n_{j-1}$. Then by  Definition~\ref{def::DL::indices} we get $(n_j,j)\in\St_{\ov{n}}$.
Let $\overline{n'}$ be the collection obtained from $\overline{n}$ by erasing the $n_j$-th row and $j$-th column.
    The Borel subalgebras $\fb_{m-1}$ and $\fb_{n'_{m-1}}$ act on $S^{N}(\Mat_{\overline{n'}})$ from the left and from the right correspondingly. We have an embedding  $\fb_{m-1}\hookrightarrow\fb_{m}$ and $\fb_{n'_{m-1}}\hookrightarrow\fb_{n_m}$ that corresponds to omitting the $n_j$-th row and $j$-th column in the Young diagram $\bbY_{\overline{n}}$. 
    
Consider the following decomposition of the symmetric tensors:
    \begin{multline*}
    S^N(\Mat_{n}) \simeq \\
    \bigoplus_{N_1+N_2+N_3=N} 
    \left(
    \begin{array}{c}
    S^{N_1}(\Span\{ v_{n_j j}\}) \otimes S^{N_2}\Bigl(\Span\{ v_{s j} \colon s< n_j\} + \Span\{ v_{n_j t} \colon t>j\}\Bigr)\otimes
    \\
    \otimes S^{N_3}(\Span\{ v_{st} \colon s\neq n_j \ \& \ t\neq j\})
    \end{array}
    \right)\simeq
    \\
    \simeq \bigoplus_{N_1+N_2+N_3=N} S^{N_1}(\Span\{v_{n_j j}\}) \otimes S^{N_2}(\Span\{ v_{s j},v_{n_jt} \colon s< n_j,t>j\}) \otimes S^{N_3}(\Mat_{\overline{n'}}).
    \end{multline*}
By induction, we know that the summands with $N_2=0$ are generated by $\fb_{n'_{m-1}}\de\fb_{m-1}$ from the desired monomials. 

We proceed by induction in $N_1+N_2$. Denote the sum of the subspaces corresponding to $N_1+N_2 \leq p$ by $D_p$.
Note that 
\[E_{sn_j}S^{N_2}(\Span\{ v_{s j},v_{n_jt} \colon s< n_j,t>j\})\subset S^{N_2-1}(\Span\{ v_{s j},v_{n_jt} \colon s< n_j,t>j\})\otimes (\Mat_{\overline{n'}}).\]
Therefore, modulo $D_{N_1+N_2-1}$ the elements $E_{sn_j}$ act only on the factor $S^{N_1}(\Span\{v_{n_j j}\})$. By the same arguments, the same holds for the elements $E_{jt}$. Thus, for any $X \in  S^{N_3}(\Mat_{\overline{n'}})$, $\sum_{a_s}+\sum_{b_t}=N_2$
\[v_{n_jj}^{N_1}\prod_{s=1}^{n_j-1}v_{sj}^{a_s}\prod_{t=j+1}^{m}v_{n_jt}^{b_t}X=\prod_{s=1}^{n_j-1}E_{s n_j}^{a_s}(v_{n_jj}^{N_1+N_2}X)\prod_{t=j+1}^{m}E_{jt}^{b_t} ~{\rm mod} ~D_{N_1+N_2-1}.\]

This completes the proof by induction.
\end{proof}

We need the following stronger version of Proposition \ref{prop:BimoduleGenerators}.

\begin{prop}\label{lem:DLgenerateStandardFiltration}
  The monomials $v^A$, $A\in \DL_{\overline{n}}$ such that $\hor(A)$ is not less than a given composition $\lambda$ (with respect to the Cherednik order~\ref{def::Cherednik::order}), generate the $\fb_{n_m}\de\fb_{m}$-bimodule ${}^\lambda \mathcal{F}$.  
\end{prop}
\begin{proof}
 We first check that for a DL-dense array $A$ the defining relations of the Demazure module \eqref{thmJosepf} are satisfied, i.e., that the left action of certain powers of operators from $U(\fb_{n_m})$ kills $v^A$ in the associated graded to the appropriate filtration. Namely, for $i<i'$ we prove
\begin{equation}\label{eq:JosephSubquotients}
    E_{i{i'}}^{\max\{\hor(A)_{i'}-\hor(A)_i,0\}+1}v^A \in 
\sum_{\substack{B \in \DL_{\overline n}\\ \hor(B)_+ > \hor(A)_+}}\U(\fb_{n_m})v^B\U(\fb_{m}).
\end{equation}
The proof of~\eqref{eq:JosephSubquotients} is split into five cases depending on the local comparison of $\hor(A)_{i'}$ and $\hor(A)_{i}$.
\begin{itemize}[itemsep=0pt, topsep=0pt]
\item If {$\mathbf{(\hor(A)_{i'}=0)}$} then  
$\max\{\hor(A)_{i'}-\hor(A)_i,0\}+1 = 1$ and 
$E_{i{i'}}v^A=0$, since $v^A$ contains no factors of the form $v_{i',\bullet}$. 
\item If $\mathbf{(\hor(A)_i = 0)}$ then  
$\max\{\hor(A)_{i'}-\hor(A)_i,0\}+1 = \hor(A)_{i'} + 1$ and 
$E_{i{i'}}^{\hor(A)_{i'} + 1}v^A=0$, because $v^A$ contains exactly $\hor(A)_{i'}$ factors of the form $v_{i',\bullet}$.
\item If $\mathbf{(\hor(A)_i\geq \hor(A)_{i'}>0)}$, then  
$i$-th and $i'$-th rows contain staircase corners
whose columns will be denoted by $j$ and $j'$ respectively.
Since $i<i'$ and $A_{ij}>A_{i'j}$, the staircase corners $(i,j)$ and $(i',j')$ have to be uncomparable (since  $A$ is DL-dense) and hence $j<j'$. Let $C\in \DL_{\overline n}$ be a DL-dense array which coincides with $A$ in all cells except for $C_{ij}=A_{ij}+1$, $C_{i'j'}=A_{i'j'}-1$. 
Then $\hor(C)_+ > \hor(A)_+$ which gives   
\[
E_{i{i'}}v^A=\frac{A_{i'j'}}{A_{ij}+1}v^CE_{jj'}\in 
\sum_{\substack{B \in \DL_{\overline n}\\ \hor(B)_+ > \hor(A)_+}}\U(\fb_{n_m})v^B\U(\fb_{m}).\]
Note that $\hor(C)_+>\hor(A)_+$, therefore
we conclude by Lemma \ref{lem:vAvB} that \eqref{eq:JosephSubquotients} holds true if $\hor(A)_{i}\ge\hor(A)_{i'}$.
\item Suppose $\mathbf{(0<\hor(A)_i<\hor(A)_{i'})}$ and $(i,j)$ and $(i',j')$ are \textbf{uncomparable}. Define $C\in \DL_{\overline n}$ to be a DL-dense array which coincides with $A$ in all cells except for $C_{ij}=A_{i'j'}+1$, $C_{i'j'}=A_{ij}-1$. Then 
$\hor(C)_+ > \hor(A)_+$
and hence by by Lemma \ref{lem:vAvB} 
\[
v^C \in\sum_{\substack{B \in \DL_{\overline n}\\ \hor(B)_+ >  \hor(A)_+}}\U(\fb_{n_m})v^B\U(\fb_{m}).
\]
We conclude that 
\[
E_{i{i'}}^{A_{i'j'}-A_{ij}+1}v^A={\rm const}. v^CE_{j{j'}}^{A_{i'j'}-A_{ij}+1}\in\sum_{\substack{B \in \DL_{\overline n}\\ \hor(B)_+ > \hor(A)_+}}\U(\fb_{n_m})v^B\U(\fb_{m}).
\]
This proves \eqref{eq:JosephSubquotients} if $\hor(A)_{i'}-\hor(A)_i>0$ and 
$(i,j)$ and $(i',j')$ are uncomparable.
\item Suppose $\mathbf{(\hor(A)_{i'}>\hor(A)_{i}>0)\ \& \ \bigl((i,j)\prec (i',j')\bigr)}$ are \textbf{comparable staircase corners}. Then by Lemma \ref{lem:gl_2HoweComputations} we get
\[E_{i{i'}}^{A_{i'j'}-A_{ij}+1}v^A\in\sum_{l=0}^{A_{ij}}\U(\fb_{n_m})v^{C_l}\U(\fb_{m}),\]
where $C_l \in \DL_{\overline n}$ coincides with $A$ in all cells except for $C_{ij}=A_{ij}-l$, $C_{i'j'}=A_{i'j'}+l$. Then by Lemma \ref{lem:vAvB} we complete the proof of equation \eqref{eq:JosephSubquotients}.
\end{itemize}

For any $A \in \DL_{\overline n}$ the bimodule
\begin{equation}\label{eq:subquotientSc}
\sum_{\substack{B \in \DL_{\overline n}\\ \hor(B)\not \prec \hor(A)}}\U(\fb_{n_m})v^B\U(\fb_{m})\left/\sum_{\substack{B \in \DL_{\overline n}\\\hor(B) \not \preceq \hor(A)}}\U(\fb_{n_m})v^B\U(\fb_{m})\right.
\end{equation}
is cyclic and is generated by the class of $v^A.$ By formula \eqref{eq:JosephSubquotients} (note that $\hor(B)_+ > \hor(A)_+$ implies $\hor(B)\succ \hor(A)$ by definition of Cherednik order) and the commutativity of the left and the right actions, we get for any $x \in \U(\fb_m)$:
\begin{equation}\label{eq:JosephSubquotientsx}
    E_{i{i'}}^{\max\{\hor(A)_{i'}-\hor(A)_{i},0\}+1}v^Ax \in \sum_{\substack{B \in \DL_{\overline n}\\ \hor(B)\not \preceq \hor(A)}}\U(\fb_{n_m})v^B\U(\fb_{m}).
\end{equation}
Now we claim that all the elements of the bimodule \eqref{eq:subquotientSc} have left weights $\preceq \hor(A)$. In fact,
this bimodule is generated by the left action of $\U(\fb_{n_m})$ from the space
$v^A\U(\fb_m)$. By \eqref{eq:JosephSubquotientsx} we know that for any $x\in \U(\fb_{m})$ the left $\fb_{n_m}$-module generated from $v^Ax$ is a quotient of the Demazure module of weight $\hor(A)$. 

By Proposition \ref{prop:BimoduleGenerators} the bimodule $S(\Mat_{\overline n})$ is generated by the elements $v^{B}, B \in \DL_{\ov n}$. 
Let us prove that 
\begin{equation}\label{eq:lambdaS}  
\sum_{\lambda \not \prec \hor(A)}{}_{\lambda}S(\Mat_{\overline n})\subset \sum_{\hor(B)\not \prec \hor(A)}\U(\fb_{n_m})v^B\U(\fb_{m}).
\end{equation}
This is equivalent to the statement that the quotient \eqref{eq:subquotientSc} does not contain any weights $\lambda$ such that $\lambda \not \prec \hor(A)$. 
In other words, we need to show that all the weights of the quotient \eqref{eq:subquotientSc} are less than $\hor(A)$.
By formula \eqref{eq:JosephSubquotientsx} applied to a DL-dense $A'$ with $\hor(A')\prec \hor(A)$,
we know that the left $\fb_{n_m}$-module generated by a vector $v^{A'}x$, $x\in \U(\fn_m)$, is a quotient of the corresponding Demazure module.
Hence, all weight of the left $\fb_{n_m}$-module generated by a vector $v^{A'}x$, $x\in \U(\fn_m)$
are less than $\hor(A)$.

Now we obtain the desired claim  
\[{}^{\hor(A)}\mathcal{F}\subset \sum_{\hor(B)\not \prec \hor(A)}\U(\fb_{n_m})v^B\U(\fb_{m}).\]
(The opposite inclusion is clear).
In fact, by definition, ${}^{\hor(A)}\mathcal{F}$ is generated by the left hand side of \eqref{eq:lambdaS} as a $\fb_{n_m}$-module.
\end{proof}

\begin{cor}
Let $\la=\hor(A)$ for some $A\in \DL(\overline{n})$. Then the quotient    ${}^{\lambda} \mathcal{F}/ \sum_{\mu\succ \lambda} {}^{\mu} \mathcal{F}$ is isomorphic as a left $\fb_{n_m}$-module to a direct sum of  {copies of the} Demazure module $D_\la$.  
\end{cor}
\begin{proof}
By Proposition \ref{lem:DLgenerateStandardFiltration} we know that ${}^{\lambda} \mathcal{F}/ \sum_{\mu\succ \lambda} {}^{\mu} \mathcal{F}$ 
   is a cyclic bimodule generated by the class of $v^A$.
By \eqref{eq:JosephSubquotients} and the commutativity of the left and the right actions, for any $x \in \U(\fb_{m})$, the class of $v^Ax$ satisfies the defining relations of the Demazure module generated by the element of the weight $\lambda$. 
   Therefore, it is isomorphic to a quotient of the left module $D_{\la}\otimes K$, where $K$ is the $\lambda$ weight space. 
But the quotient ${}^{\lambda} \mathcal{F}/ \sum_{\mu\succ \lambda} {}^{\mu} \mathcal{F}$ has an excellent filtration. Therefore, it is in fact isomorphic to $D_{\la}\otimes K$.
\end{proof}

\begin{cor}\label{cor:nvanderKallenQuotientDemazure}
For $A\in \DL(\overline{n})$  one has  an isomorphism of  $\fb_{n_m}$-$\fb_{m}$-bimodules 
\begin{equation}\label{eq:K}
{}^{\hor(A)} \mathcal{F}\left / \sum_{\mu\succ \hor(A)} {}^{\mu} \mathcal{F}\simeq D_{\hor(A)}\otimes K\right.
\end{equation}
where $K$ is a quotient of $D_{\vrt(A)}^{op}$.
\end{cor}
\begin{proof}
    The left weight $\lambda$ space of ${}^{\lambda} \mathcal{F}/ \sum_{\mu\succ \lambda} {}^{\mu} \mathcal{F}$ has the structure of a right module. Hence,
by Proposition \ref{prop:BimoduleGenerators}, the quotient ${}^{\hor(A)} \mathcal{F}\left / \sum_{\mu\succ \hor(A)} {}^{\mu} \mathcal{F}\right.$ has the structure of the exterior tensor product of the left module $D_{\hor(A)}$ and the right module generated from $v^A$. 

Transposing  \eqref{eq:JosephSubquotients}, i. e., interchanging vertical and horizontal weights, we get for $i'<i$:
\begin{equation}
   v^A E_{i'i}^{\max\{\vrt(A)_{i'}-\vrt(A)_{i},0\}+1} \in \sum_{\substack{B \in \DL_{\overline n}\\ \vrt(B)_+ > \vrt(A)_+}}\U(\fb_{n_m})v^B\U(\fb_{m}).
\end{equation}

However, $\vrt(B)_+=\hor(B)_+$ for $B \in \DL_{\overline n}$. Therefore 
\begin{equation*}
   v^A E_{i'i}^{\max\{\vrt(A)_{i'}-\vrt(A)_{i},0\}+1} \in \sum_{\substack{B \in \DL_{\overline n}\\ \hor(B)\succ \hor(A)}}\U(\fb_{n_m})v^B\U(\fb_{m}).
\end{equation*}
Hence, there exists a surjection of the right Demazure module $D_{\vrt(A)}^{op}$ onto $K$ and this completes the proof.
\end{proof}
Thus, we have constructed a left excellent filtration of $S(\Mat_{\overline n})$ with subquotients being standard modules corresponding to the weights $\hor(A), A \in \DL_{\overline n}$. Therefore, we get the following corollary, which covers one case of Theorem
\ref{thm:B}.
\begin{cor}\label{cor:LSFtriv}
Assume that $\lambda \neq \hor(A)$ for any $A \in \DL_{\overline n}$. Then 
${}^{\lambda} \mathcal{F}= \sum_{\mu\succ \lambda} {}^{\mu} \mathcal{F}.$
\end{cor}

\subsection{Generalized van der Kallen modules}
\label{sec::vdK::gener}
\begin{dfn}
\label{def::n:vdK}
We call the (right) $\fb_{n_m}$-module  $K$ from \eqref{eq:K}   the right $\overline n$-van der Kallen module and  denote in by $K^{op}_{{\overline n},{\vrt(A)}}$.
\end{dfn}

So equation \eqref{eq:K} can be reformulated in the following way:
\begin{equation}\label{eq:Klambda}
{}^{\hor(A)} \mathcal{F}\left / \sum_{\mu\succ \hor(A)} {}^{\mu} \mathcal{F}\simeq D_{\hor(A)}\otimes K^{op}_{\overline n,\vrt(A)}\right.
\end{equation}

The goal of this subsection is to describe $K^{op}_{\overline n,\vrt(A)}$.

\begin{prop}\label{prop:nvdKfiltrationvdK}
The module $K^{op}_{\overline n,\vrt(A)}$ admits a filtration such that the consecutive subquotients are isomorphic to $K^{op}_{\overline d}$ with appropriate composition $\ov{d}$, yielding $\hb(\overline{d})=\hor(A)$.
Any van der Kallen module $K_{\overline d}$ with $\hb(\overline{d})=\hor(A)$ shows up once as a subquotient.    
\end{prop}
\begin{proof}
By Corollary \ref{cor:nvanderKallenQuotientDemazure} we know that $K_{\overline n,\vrt(A)}^{op}$ is a quotient of $D_{\vrt(A)}^{op}$ (in particular, cyclic). Hence, it admits a filtration by the quotients of the van der Kallen modules -- this filtration is induced by the corresponding filtration of the Demazure module 
(the subspaces of our filtration are labeled 
by the extremal vectors contained in $K_{\overline n,\vrt(A)}^{op}$).

By Corollary \ref{cor:firstCauchy} one has 
\[\sum_{\overline d \text{ is }\overline n \text{-admissible}}\kappa_{\hb(\ov{d})}(x)\, a^{\overline d}(y)=\prod_{(i,j)\in \bbY_{\overline{n}}}\frac{1}{1-x_iy_j}=\sum_{A \in \DL_{\overline n}}\kappa_{\hor(A)}(x)\,\ch(K_{\overline n,\vrt(A)}^{op})(y).\]  
Since the key polynomials form a basis of the polynomial ring,
we obtain
\[\ch(K_{\overline n,\vrt(A)}^{op})(y)=\sum_{\hb(\overline d)=\hor(A)}a^{\overline d}(y).\]

In particular $K_{\overline n,\vrt(A)}$ contains the nonzero extremal vectors $v_{\overline d}$ only for weights $\overline d$ such that 
$\hb(\overline d)={\hor}(A)$, $A \in \DL_{\overline n}$. Therefore, the subquotients of the above filtration are nonzero only for such weights 
$\overline d$ and are isomorphic to the quotients of the corresponding van der Kallen modules $K_{\overline d}$. But by the equality of characters, these quotients are isomorphic to the whole van der Kallen modules, and this completes the proof.
\end{proof}

\begin{cor}\label{cor:LCF}
One has the left Cauchy identity
\[
\prod_{(i,j)\in \bbY_{\overline{n}}}\frac{1}{1-x_iy_j}=\sum_{\substack{A \in \DL_{\overline n}\\ \hb(\ov{d})=\hor(A)}}\kappa_{\hor(A)}(x)\, a^{\ov{d}}(y).
\]  
\end{cor}

\begin{lem}\label{lem:transposedinequality}
Assume that $\ov{c}\prec \ov{d}$ is a pair of $\ov{n}$-admissible compositions,
$\ov{c},\ov{d}\in\bZ_{\ge 0}^m$.
 Then $\hb(\overline c) \preceq \hb(\overline d)$ with respect to Cherednik order on $\bZ^{n_m}$.
\end{lem}
\begin{proof}
If $\overline c_+ <\overline d_+$, then 
\[
\hb(\overline c)_+ = \overline c_+ < \overline d_+ = \hb(\overline d)_+ \ \Rightarrow \ \hb(\overline c) \prec \hb(\overline d).
\]
So from now on we assume that $\overline c_+ =\overline d_+$.
    Assume that $d_{j}<d_{j'}$ for $j<j'$. By definition we have $n_{j}\leq n_{j'}$. Therefore, we have the natural inclusion of the Demazure modules $D_{\varepsilon_{n_{j}}}\hookrightarrow D_{\varepsilon_{n_{j'}}}$. Therefore, we have an inclusion
\begin{multline*}
S^{d_{j'}}(D_{\varepsilon_{n_{j}}})\otimes  S^{d_{j}}(D_{\varepsilon_{n_{j'}}})\simeq S^{d_{j}}(D_{\varepsilon_{n_{j}}})\cdot S^{d_{j'}-d_{j}}(D_{\varepsilon_{n_{j}}})  \otimes  S^{d_{j}}(D_{\varepsilon_{n_{j'}}})\hookrightarrow 
\\ \hookrightarrow S^{d_{j}}(D_{\varepsilon_{n_{j}}})\otimes S^{d_{j'}-d_{j}}(D_{\varepsilon_{n_{j'}}}) \cdot  S^{d_{j}}(D_{\varepsilon_{n_{j'}}}) \simeq S^{d_{j}}(D_{\varepsilon_{n_{j}}})\otimes  S^{d_{j'}}(D_{\varepsilon_{n_{j'}}}).
\end{multline*}
Consequently, there exists an inclusion
\[\bigotimes_{l=1}^m S^{c_{l}}(D_{\varepsilon_{n_{l}}})\hookrightarrow \bigotimes_{l=1}^m S^{d_{l}}(D_{\varepsilon_{n_{l}}})\]
for $\overline c=s_{jj'}\overline d$ and hence for arbitrary $\overline c \prec \overline d$,
$\overline c_+ =\overline d_+$.
However, 
\[D_{\hb(\overline d)}\simeq \bigotimes_{l=1}^m S^{d_{l}}(D_{\varepsilon_{n_{l}}})\left/ \mathcal{G}^{>\overline d} \right.,\]
where $ \mathcal{G}^{>\ov{d}}$ is the submodule generated by all the elements of weights $\overline b$ such that $\overline b_+>\overline d_+$.
Thus by Corollaries \ref{cor:weightinequality}, \ref{cor:SubsequenceInequality} 
 and Theorem \ref{thm:PieriRule} we have a natural inclusion 
$D_{\hb(\overline c)}\hookrightarrow D_{\hb(\overline d)}.$
The inclusion of the Demazure modules is equivalent to the inequality $\hb(\ov{c})\preceq\hb(\ov{d})$ for the Cherednik order.
\end{proof}

\begin{lem}\label{lem:vrt(B)>c}
Let $\ov{c}\in\bZ_{\ge 0}^m$ be an $\ov{n}$-admissible composition and let $B\in\DL_{\ov{n}}$ be an array such that 
$\hb(\ov{c})=\hor(B)$. Then $\vrt(B)_+=\ov{c}_+$ and $\vrt(B)\succeq \ov{c}$ with respect to the Cherednik order.
\end{lem}
\begin{proof}
We first note that Corollary~\ref{cor:transposedDLdense} implies the existence of $B$. Then by Proposition~\ref{prop:nvdKfiltrationvdK} the van der Kallen module $K_{\overline{c}}^{op}$ is isomorphic to a subquotient of $K_{\overline n,\vrt(B)}^{op}$. However, by Corollary~\ref{cor:nvanderKallenQuotientDemazure} the module $K_{\overline n,\vrt(B)}^{op}$ is a quotient of $D^{op}_{\vrt(B)}$. Therefore, all the weights of the elements of $K_{\overline n,\vrt(B)}^{op}$ are less than or equal to $\vrt(B)$ in the Cherednik order.
\end{proof}

 \begin{prop}\label{prop:KD/D}
 For all DL-dense arrays $A$ whose multiset of nonzero values is equal to {the multiset of parts of a} partition $\lambda$, we have an isomorphism
 \begin{equation}
\label{eq::KD/D} 
K^{op}_{\overline n,\vrt(A)} \simeq D^{op}_{\vrt(A)}\left/\sum_{\substack{B \in \DL_{\overline n}({\lambda}{_+})\\ 
D^{op}_{\vrt(B)}\subsetneq D^{op}_{\vrt(A)}}} D^{op}_{\vrt(B)} \right.,
 \end{equation}
 i.e., the module $K$ is the quotient of the Demazure module by the sum of strictly
 embedded Demazure modules corresponding to the weights of the form $\vrt(B)$, 
 $B\in\DL_{\ov{n}}$. 
 \end{prop}
\begin{proof}
By Corollary \ref{cor:nvanderKallenQuotientDemazure} we have the surjection    
\[
\psi:D^{op}_{\vrt(A)}\twoheadrightarrow K^{op}_{\overline n,\vrt(A)}.
\]
Let us describe the kernel of $\psi$. 
Let $\overline{c}$ be a vertical weight of an extremal vector $v_{\ov{c}}\in D_{\vrt(A)}^{op}$ such that $v_{\ov{c}}\in\ker\psi$.
Let us prove the following two claims: 
\begin{itemize}
\item  if $\ov{c}=\vrt(B)$ with $D^{op}_{\vrt(B)}\subsetneq D^{op}_{\vrt(A)}$, then
$v_{\ov{c}}\in\ker\psi$,
\item if $v_{\ov{c}}\in\ker\psi$, then $\ov{c}=\vrt(B)$ with $D^{op}_{\vrt(B)}\subsetneq D^{op}_{\vrt(A)}$.
\end{itemize}

To prove the first claim we note that  if $\ov{c}=\vrt(B)$, then $\hb(\ov{c})=\hb(\vrt(B))=\hor(B)$. Since 
$D^{op}_{\vrt(B)}\ne D^{op}_{\vrt(A)}$, one has $\hor(B)\ne\hor(A)$ and hence
$\hb(\ov{c})\ne\hor(A)$. By Proposition \ref{prop:nvdKfiltrationvdK} this implies that $v_{\ov{c}}\in\ker\psi$.

Now assume that $v_{\ov{c}}\in\ker\psi$. 
By Corollary \ref{cor:transposedDLdense} there exists $B\in \DL_{\ov{n}}$ such that $\hb(\ov{c})=\hor(B)$. We claim that 
$\ov{c}\preceq \vrt(B)\prec \vrt(A)$. The second inequality holds true thanks to Lemma \ref{lem:transposedinequality}. In fact,  
$\ov{c}\prec \vrt(A)$, because by the definition of a standard module, the weights of all the vectors of $D^{op}_{\vrt(A)}$ are $\preceq \vrt(A)$. 
This inequality implies $\hb(\ov{c})=\hb(\vrt(B))=\hor(B)\preceq \hor(A)$. Now, applying again Lemma \ref{lem:transposedinequality} for the transposed Young diagram, we derive 
$\vrt(B) \preceq \vrt(A)$ (the equality is not possible, since 
$\hb(\ov{c})\ne \hor(A)$ thanks to the fact that $v_{\ov{c}}\in\ker\psi$).
It remains to note that thanks to Lemma \ref{lem:vrt(B)>c} that $\ov{c}\preceq \vrt(B)$. 
We have shown that the set of extremal vectors $v_{\ov{c}}$ in $\ker\psi$ coincides with the set of extremal vectors of the sum of Demazure
modules $D^{op}_{\vrt(B)}$, $B\in\DL_{\ov{n}}({\lambda})$, that are strictly
embedded into $D^{op}_{\vrt(A)}$. 
This implies that $\psi$ defines a surjection from the right hand side of~\eqref{eq::KD/D} to generalized van der Kallen module: 
\begin{equation}\label{eq:surjtoK}
\psi: D^{op}_{\vrt(A)}\left/\sum_{\substack{B \in \DL_{\overline n}({\lambda}) \\ 
D^{op}_{\vrt(B)}\subsetneq D^{op}_{\vrt(A)}}} D^{op}_{\vrt(B)} \right. \twoheadrightarrow 
K^{op}_{\overline n,\vrt(A)}.
\end{equation}
Let us show that the kernel of this surjection is trivial.
Indeed, consider the filtration $\calF$ of the left hand side of \eqref{eq:surjtoK} induced by Demazure modules $D_{\sigma\lambda}^{op}\hookrightarrow D_{\vrt(A)}^{op}$ with $\sigma\lambda\succeq\vrt(A)$. The associated graded components are the quotients of the van der Kallen modules $K_{\sigma\lambda}^{op}$ because they are cyclic $\fb_m$-modules generated by extremal vectors and do not contain other extremal vectors that belong to the smaller terms of filtration.
However, by Proposition \ref{prop:nvdKfiltrationvdK} the module 
$K^{op}_{\overline n,\vrt(A)}$ is filtered by these van der Kallen modules. Hence, the associated graded components coincide with van der Kallen modules and the kernel of $\psi$ is trivial.
\end{proof}

We see from Proposition~\ref{prop:KD/D} that isomorphism~\eqref{eq::KD/D} can be considered as a definition of a right $\ov{n}$-van der Kallen $\fb_m$-module.
We note that for each $\DL_{\ov{n}}$-dense array $A$ one can similarly define the left $\ov{n}$-van der Kallen $\fb_{n_m}$-module:
$$
K_{\ov{n},\hor(A)}= D_{\hor(A)}\left/\sum_{\substack{B\in\DL_{\ov{n}}(\hor(A)_+) \\ D_{\hor(B)}\subsetneq D_{\hor(A)}}} D_{\hor(B)}\right. .
$$
\begin{prop}
For each DL-dense array $A\in\DL_{\ov{n}}$, the module $K_{\ov{n},\hor(A)}$ is isomorphic to the space of multiplicities of the Demazure module $D_{\vrt(A)}^{op}$ with respect to the right standard filtration. 
Moreover, we have the following isomorphism of characters:
$$
\ch\left(K_{\ov{n},\hor(A)}\right)(x) = \sum_{\hb_{\ov{n^{\tau}}}(\ov{d}) = \vrt(A)} a^{\ov{d}}(x).
$$
\end{prop}
\begin{proof}
The transposition of a diagram $\ov{n}\mapsto\ov{n^{\tau}}$ interchanges the left modules with the right modules.
\end{proof}

\subsection{Alternating sign Cauchy identity}
\label{sec::Cauchy::sign}
Let us rewrite the left Cauchy identity (Corollary \ref{cor:LCF}) purely in terms of key
polynomials. Let us fix a multiset (or a partition) ${\lambda}$ and introduce a poset structure on the set of DL-dense arrays whose multiset of values is equal to ${\lambda}$.
For two arrays $A,B\in\DL_{\ov{n}}({\lambda})$ we say that $A\succeq B$ if $\vrt(A)\succeq\vrt(B)$ with respect to the Cherednik order 
(we note that  $\vrt(A)_+=\vrt(B)_+ ={\lambda}$). In other words, $A\succeq B$ is equivalent to 
$D^{op}_{\vrt(B)}\subset D^{op}_{\vrt(A)}$. Let $\mu^{\DL_{\ov{n}}(\lambda)}(A,B)$ be the M\"obius function of this poset (see e.g.~\cite{Lovasz}).
We put forward the following conjecture (see \cite{KhM} for the proof). 

\begin{conj}
\label{lem::mobius::character}
One  has for all $A\in\DL_{\ov{n}}(\lambda)$
\[
\ch\left(K^{op}_{\ov{n},\vrt(A)}\right) = \sum_{B\preceq A} \kappa^{\vrt(B)}(y)\,\mu^{\DL_{\ov{n}}({\lambda})}(A,B),
\]
implying
\begin{equation}
\label{eq::altern::Cauchy}
\prod_{(i,j)\in \bbY_{\overline{n}}}\frac{1}{1-x_iy_j}=
\sum_{l({\lambda})\leq \#\St_{\ov{n}} }
\sum_{\substack{A,B \in \DL_{\overline n}({\lambda})\\ A\succeq B}} \kappa_{\hor(A)}(x)\, \kappa^{\vrt(B)}(y)\, \mu^{{\DL_{\ov{n}}({\lambda})}}(A,B).
\end{equation}    
\end{conj}

It is easy to see that the left-hand side of~\eqref{eq::altern::Cauchy} has extra symmetries when the lengths of some columns (rows) are equal. 
More precisely, to each Young diagram $\bbY_{\ov{n}}$ we assign a collection of integer numbers $r_1:=\#\{i\colon n_i =n_1\}$, $r_2:=\#\{i\colon n_i = n_{r_1+1}\}$, \ldots. In other words, $r_i$'s count the numbers of columns of the same length.
Respectively, let $s_1,s_2,\ldots$ be the corresponding numbers for the transposed diagram $\ov{n^\tau}$.
Then the left-hand side of~\eqref{eq::altern::Cauchy} is symmetric with respect to the action of the product of symmetric groups $S_{s_1}\times S_{s_2}\times\ldots$ on $x$-variables and with respect to the action of $S_{r_1}\times S_{r_2}\times\ldots$ on the $y$-variables.
We note that each term in the right-hand side of~\eqref{eq::altern::Cauchy} also admits these symmetries.

Indeed, let $\fp_1\subset \gl_{n_m}$ be the parabolic subalgebra  with blocks
of sizes $s_1,s_2,\ldots$ and let $\fp_2\subset \gl_{m}$ be the parabolic subalgebra  with blocks of sizes $r_1,r_2,\ldots$.
Then the space of staircase matrices $\Mat_{\ov{n}}$ admits a natural action of the parabolic subalgebra $\fp_1$ from the left and $\fp_2$ from the right.
Moreover, all filtrations we are dealing with are invariant with respect to the action of these parabolic subalgebras. In particular, $\forall A\in\DL_{\ov{n}}$ the modules $D_{\hor(A)}$, $K_{\ov{n},\hor(A)}$ are $\fp_1$-modules and $D_{\vrt(A)}^{op}$, $K_{\ov{n},\vrt(A)}^{op}$ are right $\fp_2$-modules.
Hence, on the level of characters, we see that for any $A\in\DL_{\ov{n}}$ the key polynomial $\kappa_{\hor(A)}(x)$ is $S_{s_1}\times S_{s_2}\times\ldots$-symmetric, respectively, 
$\kappa^{\vrt(A)}(y)$ is $S_{r_1}\times S_{r_2}\times\ldots$-symmetric.

\appendix
\section{Examples}
\label{sec::examples}
\subsection{Rectangular matrices}\label{subsec:rectangular}
Let $n_1=\dots=n_m$, i.e. $\Mat_{\ov{n}}$ is the whole space of $n_1\times m$
matrices. Then $\St_{\ov{n}}=\{(n_1,1),(n_1-1,2),\dots,(1,n_1)\}$ for $m\ge n_1$
and $\St_{\ov{n}}=\{(n_1,1),(n_1-1,2),\dots,(n_1+1-m,m)\}$ for $m\le n_1$ (in particular,
the cardinality of $\St_{\ov{n}}$ is equal to $\min(n_1,m)$). Here
is the picture for $n_1=5$, $m=7$:
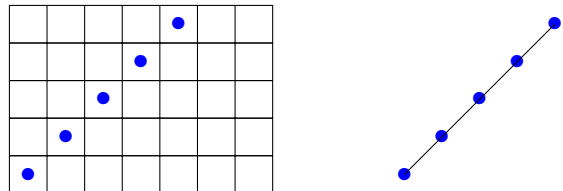
\[
\begin{tikzpicture}[scale=0.5]
\draw[step=1cm] (0,-1) grid (7,0);
\draw[step=1cm] (0,-2) grid (7,-1);
\draw[step=1cm] (0,-3) grid (7,-2);
\draw[step=1cm] (0,-4) grid (7,-3);
\draw[step=1cm] (0,-5) grid (7,-4);
\node (v0) at (0.5,-4.5) { {{\color{blue}$\bullet$}}};
\node (v1) at  (1.5,-3.5) {{ \color{blue}$\bullet$ }};
\node (v2) at (2.5,-2.5) {{ \color{blue}$\bullet$ }};
\node (v3) at (3.5,-1.5) {{ \color{blue}$\bullet$ }};
\node (v4) at (4.5,-0.5) {{ \color{blue}$\bullet$ }};
\node (v10) at (10.5,-4.5) { {{\color{blue}$\bullet$}}};
\node (v11) at  (11.5,-3.5) {{ \color{blue}$\bullet$ }};
\node (v12) at (12.5,-2.5) {{ \color{blue}$\bullet$ }};
\node (v13) at (13.5,-1.5) {{ \color{blue}$\bullet$ }};
\node (v14) at (14.5,-0.5) {{ \color{blue}$\bullet$ }};
\draw (10.5,-4.5) -- (11.5,-3.5);
\draw (11.5,-3.5) -- (12.5,-2.5);
\draw (12.5,-2.5) -- (13.5,-1.5);
\draw (13.5,-1.5) -- (14.5,-0.5);
\end{tikzpicture} 
\]
The poset $\St_{\ov{n}}$ is linearly ordered according to the first coordinate.
A composition $\ov{d}=(d_1,\dots,d_m)$ is $\ov{n}$-admissible if and only if 
there are at most $n_1$ non-zero entries in $\ov{d}$. 

For a composition $\ov{d}\in\bZ_{\ge 0}^m$ let $\ov{d}_-$ be a reordering of
$\ov{d}$ such that the smaller numbers come first; in particular, all zero entries
of $\ov{d}$ show up in the beginning. Let $s(\ov{d})$ be the number of non-zero entries 
of $\ov{d}$ (in particular, $s(\ov{d})\le n_1$ if $\ov{d}$ is $\ov{n}$-admissible). 
Then for an $\ov{n}$-admissible $\ov{d}$ the vector 
$\hb(\ov{d})\in\bZ_{\ge 0}^{n_1}$ is a sequence starting with $n_1-s(\ov{d})$ zeroes
followed by the last $s(\ov{d})$ entries of $\ov{d}_-$. In particular, the weight $\hb(\ov{d})$ is anti-dominant.

The Cauchy identity
coming for the right filtration reads as
\[
\prod_{i=1}^{n_1}\prod_{j=1}^{m}\frac{1}{1-x_iy_j}=
\sum_{\overline d: s(\ov{d})\le n_1}\kappa_{\hb(\overline d)}(x)\, a^{\overline d}(y).
\]
Since $\hb(\ov{d})$ is anti-dominant, the key polynomial $\kappa_{\hb(\overline d)}(x)$ is equal 
to the Schur polynomial whose (dominant) weight belongs to the $S_m$ orbit of $\ov{d}$. 
Now the sum of all Demazure atoms $a_{\overline d}(y)$ with the same value 
$\hb(\ov{d})$ also produces the same Schur polynomial. Hence, our formula is equivalent 
to the classical Cauchy identity.

Now, let us consider the Cauchy identity coming from the left standard filtration.
The set of DL-dense arrays consists of matrices $A$ whose non-zero entries 
are located at sites $(n_1,1), (n_1-1,2)$, etc. Let us denote these 
entries by $a_1,\dots,a_{\min(n_1,m)}$. Since $A\in\DL_{\ov{n}}$, one has 
$a_1\ge a_2\ge \dots$. By definition, $\hor(A)=(0^{n_1-\min(n_1,m)},a_{\min(n_1,m)},\dots,a_1).$
Now the Cauchy formula from Corollary~\ref{cor:firstCauchy} reads as
\[
\prod_{i=1}^{n_1}\prod_{j=1}^{m}\frac{1}{1-x_iy_j}=
\sum_{a_1\ge\dots\ge 
a_{\min(n_1,m)}}\sum_{\hb(\ov{d})=\vrt(A)}\kappa_{\vrt(A)}(x)\, a^{\overline d}(y).
\]
As in the right Cauchy formula, $\kappa_{\vrt(A)}(x)$ is a Schur polynomial and the sum of all Demazure atoms $a_{\overline d}(y)$
with $\hb(\ov{d})=\vrt(A)$ also produces the same Schur polynomial.

For any partition $\la$ the set $\DL_{\ov{n}}(\la)$ consists of a single element.
Hence, the alternating sign formula \eqref{eq::altern::Cauchy} is exactly the classical Cauchy identity.

\subsection{Upper-triangular matrices} \label{subsec:upper-triang}
In this subsection we consider the case $\ov{n}=(1, \dots,m-1,m)$. 
Then $\St_{\ov{n}}=\{(1,1),(2,,2),\dots,(m,m)\}$ (in particular,
$\St_{\ov{n}}$ has $m$ elements). Here
is the picture for $n_1=m=5$:
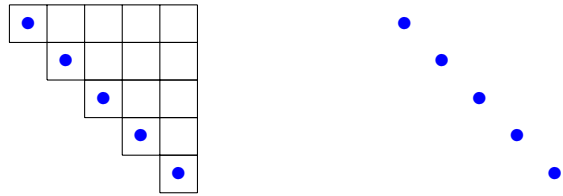
\[
\begin{tikzpicture}[scale=0.5]
\draw[step=1cm] (0,-1) grid (5,0);
\draw[step=1cm] (1,-2) grid (5,-1);
\draw[step=1cm] (2,-3) grid (5,-2);
\draw[step=1cm] (3,-4) grid (5,-3);
\draw[step=1cm] (4,-5) grid (5,-4);
\node (v0) at (0.5,-0.5) { {{\color{blue}$\bullet$}}};
\node (v1) at  (1.5,-1.5) {{ \color{blue}$\bullet$ }};
\node (v2) at (2.5,-2.5) {{ \color{blue}$\bullet$ }};
\node (v3) at (3.5,-3.5) {{ \color{blue}$\bullet$ }};
\node (v4) at (4.5,-4.5) {{ \color{blue}$\bullet$ }};
\node (v10) at (10.5,-0.5) { {{\color{blue}$\bullet$}}};
\node (v11) at  (11.5,-1.5) {{ \color{blue}$\bullet$ }};
\node (v12) at (12.5,-2.5) {{ \color{blue}$\bullet$ }};
\node (v13) at (13.5,-3.5) {{ \color{blue}$\bullet$ }};
\node (v14) at (14.5,-4.5) {{ \color{blue}$\bullet$ }};
\end{tikzpicture} 
\]
The poset $\St_{\ov{n}}$ consists of $m$ uncomparable elements.
Any composition $\ov{d}=(d_1,\dots,d_m)$ is $\ov{n}$-admissible
and $\hb(\ov{d})=\ov{d}$.
Hence, the right Cauchy formula {from Corollary~\ref{cor:firstCauchy}} reads as
\[
\prod_{1\le i\le j\le m}\frac{1}{1-x_iy_j}=
\sum_{\overline d}\kappa_{\overline d}(x)\, a^{\overline d}(y),
\]
which is the classical form of the non-symmetric Cauchy identity. 
{The left Cauchy formula from Corollary~\ref{cor:LCF}} coincides with the right one verbatim, since 
DL-dense arrays are all matrices $A$ supported on $\St_{\ov{n}}$ and $\vrt(A)=(A_{11},\dots,A_{m,m})$.

For a partition $\la$ let us define the numbers $r_1,r_2,\dots$ by
\[
\la_1=\dots\la_{r_1}>\la_{r_1+1}=\dots = \la_{r_1+r_2}>\dots.
\]
Let $\fp\subset \gl_m$ be the parabolic subalgebra with blocks of sizes $r_1,r_2,\dots$ and let $W_P\subset S_m$ be the corresponding subgroup of the Weyl group. Then the poset $\DL_{\ov{n}}(\la)$  is isomorphic to
the parabolic Bruhat graph and $\mu^{\DL_{\ov{n}}(\la)}$ in \eqref{eq::altern::Cauchy} is the corresponding M\"obius function (see \cite{BB}).

\subsection{The AGL formula}
Let us fix three numbers $n$, $p$ and $q$ with $n\ge q\ge p\ge 1$.
Following \cite{AGL} we consider the staircase shape Young diagram whose
column lengths $(n_1,\dots,n_p)$ are of the form $(n-p+1,n-p+2,\dots,q,q^{n-q})$
(in particular, we assume that $n-p+1\le q$). 
Then 
\begin{multline*}
\St_{\ov{n}}=\{(n-p+1,1),(n-p+2,2),\dots,(q,q-n+p),\\
(n-p,q-n+p+1),(n-p-1,q-n+p+2),\dots, (q+1-p,p)\},
\end{multline*}
(in particular, $\St_{\ov{n}}$ has $p$ elements).
Here is the picture for $p=7$,
$q=9$, $n=12$:

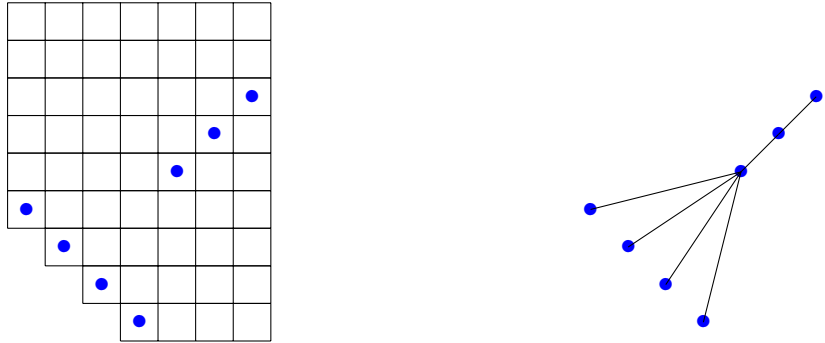
\[
\begin{tikzpicture}[scale=0.5]
  	\draw[step=1cm] (0,-1) grid (7,0);
	\draw[step=1cm] (0,-2) grid (7,-1);
	\draw[step=1cm] (0,-3) grid (7,-2);
	\draw[step=1cm] (0,-4) grid (7,-3);
    \draw[step=1cm] (0,-5) grid (7,-4);
	\draw[step=1cm] (0,-6) grid (7,-5);
    \draw[step=1cm] (1,-7) grid (7,-6);
	\draw[step=1cm] (2,-8) grid (7,-7);
    \draw[step=1cm] (3,-9) grid (7,-8); 
\node (v0) at (.5,-5.5) { {{\color{blue}$\bullet$}}};
\node (v1) at  (.5+1,-6.5) {{ \color{blue}$\bullet$ }};
\node (v2) at (.5+2,-7.5) {{ \color{blue}$\bullet$ }};
\node (v3) at (.5+3,-8.5) {{ \color{blue}$\bullet$ }};
\node (v4) at (.5+4,-4.5) {{ \color{blue}$\bullet$ }};
\node (v5) at (.5+5,-3.5) {{ \color{blue}$\bullet$ }};
\node (v6) at (.5+6,-2.5) {{ \color{blue}$\bullet$ }};

\node (v10) at (15.5,-5.5) { {{\color{blue}$\bullet$}}};
\node (v11) at  (15.5+1,-6.5) {{ \color{blue}$\bullet$ }};
\node (v12) at (15.5+2,-7.5) {{ \color{blue}$\bullet$ }};
\node (v13) at (15.5+3,-8.5) {{ \color{blue}$\bullet$ }};
\node (v14) at (15.5+4,-4.5) {{ \color{blue}$\bullet$ }};
\node (v15) at (15.5+5,-3.5) {{ \color{blue}$\bullet$ }};
\node (v16) at (15.5+6,-2.5) {{ \color{blue}$\bullet$ }};

\draw (15.5+6,-2.5) -- (15.5+5,-3.5);
\draw (15.5+5,-3.5) -- (15.5+4,-4.5);
\draw (15.5,-5.5) -- (15.5+4,-4.5);
\draw (15.5+1,-6.5) -- (15.5+4,-4.5);
\draw (15.5+2,-7.5) -- (15.5+4,-4.5);
\draw (15.5+3,-8.5) -- (15.5+4,-4.5);
\end{tikzpicture} 
\]

The Hasse diagram of the poset $\St_{\ov{n}}$ is a tree with $q-n+p$ leaves (these are the largest elements in the poset).
A composition $\ov{d}=(d_1,\dots,d_p)$ is $\ov{n}$-admissible if the number of non-zero entries of $\ov{d}$
does not exceed $q$; since $p\le q$, any composition $\ov{d}$ is $\ov{n}$-admissible.

Now let us describe $\hb(\ov{d})\in\bZ_{\ge 0}^q$ for $\ov{d}\in\bZ_{\ge 0}^p$.
The procedure consists of $p$ steps and starts with the zero vector. 
The first $p-(n-q)$ steps are very simple: at the $i$-th step we set $\hb(\ov{d})_{i+n-p}=d_i$.
So after the $p-(n-q)$ steps we have $\hb(\ov{d})_{i+n-p}=d_i$ and all other entries of $\hb(\ov{d})$ are zero. We are left to do more $n-q$ steps.
Let $i=p-(n-q)+j$ ($j=1,\dots,n-q$). At $i$-th step we determine $(\hb(\ov{d})_{n-p-j+1},\dots,\hb(\ov{d})_q)$ (so after
the last step corresponding to $j=n-q$, $i=p$ we  get the whole vector
$\hb(\ov{d})$). Explicitly, the $i$-th step 
is the algorithm from Lemma \ref{lem::reordering}
applied to $\lambda=(0,\hb(\ov{d})_{n-p-j+2},\dots,\hb(\ov{d})_q)$ and $d=d_i$.

The right Cauchy reads as
\[
\prod_{(i,j)\in Y_{\ov{n}}} \frac{1}{1-x_iy_j} = 
\sum_{\overline d\in\bZ_{\ge 0}^p}\kappa_{\hb(\overline d)}(x)a^{\overline d}(y).
\]
Recall that the following formula is proved in \cite{AGL}:
\begin{equation}\label{eq:A1}
\prod_{(i,j)\in Y_{\ov{n}}} \frac{1}{1-x_iy_j} = \sum_{\overline d\in\bZ_{\ge 0}^p} \kappa_{\ov{d}'}(x_1,\dots,x_q)\, a^{\ov{d}}(y_1,\dots,y_p),    
\end{equation}
where the composition 
$\ov{d}'=(0^{q-p},\al_1,\dots,\al_p)$ is defined by the following procedure (see \cite{AGL}, Theorem 3.20).
For $i$ running from $p$ to $1$ one performs three steps:
\begin{itemize}
\item for $j=i+1,\dots,p$ successively ignore the rightmost entry of $(d_1,\dots,d_p)$, which is equal to $\al_j$;
\item let $k_i=\min(i,n-q+1)$;
\item  $\al_i$ is equal to the maximum among the remaining (not ignored) rightmost $k_i$ elements in $(d_1,\dots,d_p)$.
\end{itemize}

We claim that the two formulas above coincide, i.e.  $\ov{d}'=\hb(\ov{d})$.   
One shows by decreasing induction that the numbers $\al_i$ showing up in $\ov{d}'$ are produced by the same algorithm as in Lemma \ref{lem::reordering}. For example, $\al_p=\max\{d_{p-(n-q)},\dots,d_p\}$ . Since $p+q>n$ one has $k_p=n-q+1$ and hence 
$\al_p=\max\{d_{p-(n-q)},\dots,d_p\}$. One easily sees that the algorithm of Lemma \ref{lem::reordering}  produces the same answer  (only the last 
$n-q+1$ elements from $\ov{d}$ determine $\hb(\ov{d})_p$ and $\hb(\ov{d})_p$ is the largest from them).

Let us comment on the left Cauchy formula from Corollary~\ref{cor:LCF}. A DL-dense array $A$ has (at most) $p$ non-zero entries $a_1,\dots,a_p$, where $a_i$
is the entry located in the $i$-th column. One has 
$a_p\le a_{p-1}\le \dots \le a_{q-n+p+1}$ and $a_{q-n+p+1}\le a_i$
for any $i=1,\dots,q-n+p$ (we note that by definition, $\vrt(A)=(0^{q-p},a_p,\dots,a_1)$). These inequalities determine the range of summation in the Cauchy formula from Corollary~\ref{cor:LCF}.

Finally, consider the poset $\DL_{\ov{n}}(\la)$. For an array $A\in \DL_{\ov{n}}(\la)$, the entries of $A$ in rows from $n-p$ to $q$ are fixed by $\la$ (since these entries are no smaller than the entries in the higher rows). 
Hence, the corresponding poset structure comes from the parabolic Bruhat graph (as in the non-symmetric example above).


\begin{thebibliography}{XXXXXX}
\bibitem[Al]{Al}
P.~Alexandersson, \emph{The symmetric functions catalog}, 
https://www.symmetricfunctions.com.

\bibitem[As1]{As1}
S.~Assaf, \emph{Weak dual equivalence for polynomials}, Ann. Comb. 26, 571 -- 591 (2022).

\bibitem[As2]{As2}
S.~Assaf,
\emph{An insertion algorithm for multiplying Demazure characters by Schur polynomials}, \eprint{2109.05651}



\bibitem[AQ1]{AQ1}
S. Assaf, D. Quijada,
{\it A Pieri rule for key polynomials},
Séminaire Lotharingien de Combinatoire 80B (2018),
Article \# 78, 12 pp. 


\bibitem[AQ2]{AQ2}
S. Assaf, D. Quijada,
{\it A Pieri rule for Demazure characters of the general linear group}, arXiv:1908.08502.


\bibitem[AE]{AE}
O. Azenhas and A. Emami,
{\it An analogue of the Robinson-Schensted-Knuth correspondence and non-symmetric Cauchy kernels for truncated staircases},
European J. Combin. 46 (2015), 16--44. 


\bibitem[AGL]{AGL}
O. Azenhas, T. Gobet, and C. Lecouvey,
\emph{Non symmetric Cauchy kernel, crystals and last passage percolation}, 
Tunisian Journal of Mathematics, Vol. 6 (2024), No. 2, 249--297.


\bibitem[BGG]{BGG}
I.N.Bernsteĭn,  I.M.Gelfand, S.I.Gelfand, \emph{Structure of representations that are generated by vectors of highest weight},  Funckcional. Anal. i Prilozhen. 5 (1971), no. 1, 1--9.


\bibitem[BB]{BB}
A. Bj\"{o}rner and F. Brenti, Combinatorics of Coxeter groups,
Graduate Texts in Mathematics Vol.~231, 
Springer, New York, 2005.



\bibitem[BC]{BC}
A.~Borodin, I.~Corwin, {\it Macdonald processes}, Probability Theory and Related Fields, 158(1--2):225--400, 2014.


\bibitem[BP]{BP}
A.~Borodin, L.~Petrov, {\it Integrable probability: From representation theory to Macdonald
processes}, Probability Surveys, 11:1--58, 2014.

\bibitem[BS]{BS}
J.~Brundan, C.~Stroppel, {\it Semi-infinite highest weight categories},  	arXiv:1808.08022.


\bibitem[BW]{BW}
D.~Betea, M.~Wheeler, {\it Refined Cauchy and Littlewood identities, plane partitions and symmetry classes of alternating sign matrices}, Journal of Combinatorial Theory, Series A, 137:126--165, January 2016.


\bibitem[CK]{CK}
S.-I. Choi and J.-H. Kwon,
{\it Lakshmibai-Seshadri paths and non-symmetric Cauchy identity},
Algebr. Represent. Theory 21 (2018), no. 6, 1381-1394. 

\bibitem[CPS]{CPS}
E. Cline, B. Parshall and L. Scott, {\it Finite dimensional algebras and highest weight categories},
J. Reine Angew. Math. 391 (1988), 85–99.


\bibitem[Dem1]{Dem1}
M.\,Demazure, \emph{D\'esingularisation des vari\'et\'es de Schubert 
g\'en\'eralis\'ees}, Ann. Sci. \'Ecole
Norm. Sup. (4), 7:53--88, 1974. 

\bibitem[Dem2]{Dem2}
M.\,Demazure, \emph{Une nouvelle formule des caracteres}, Bulletin des Sciences Math\'ematiques, 98(3):163--172, 1974.



\bibitem[DK1]{DK1}
V. Danilov, G. Koshevoy,  
{\emph Bi-crystals and crystal $(GL(V), GL(W))$-duality},  Kyoto University. Research Institute for Mathematical Sciences [RIMS] (2004).

\bibitem[DK2]{DK2}
V.\,I.\,Danilov, G.\,A.\,Koshevoi. 
{\emph Arrays and the combinatorics of Young tableaux}, 
Russian Mathematical Surveys 60.2 (2005): 269.


\bibitem[FKhM1]{FKhM1}
{E. Feigin, A. Khoroshkin, and Ie.~Makedonskyi},
{\it Duality theorems for current groups},
Israel Journal of Mathematics, vol. 248, 441--479 (2022).

\bibitem[FKhM2]{FKhM2}
{E. Feigin, A. Khoroshkin, and Ie.~Makedonskyi},
{\it Parahoric Lie algebras and parasymmetric Macdonald polynomials},
arXiv:2311.12673, to appear in Israel Journal of Mathematics.

\bibitem[FKhMO]{FKhMO}
E. Feigin, A. Khoroshkin, I. Makedonskyi, D. Orr, {\it Peter-Weyl theorem for Iwahori groups and highest weight categories},
arXiv:2307.02124. 


\bibitem[FMO]{FMO}
E. Feigin, I. Makedonskyi, D. Orr, {\it Nonsymmetric q-Cauchy identity and
representations of the Iwahori algebra},
https://arxiv.org/abs/2303.00241, to apper in Kyoto Journal of Mathematics.



\bibitem[FL]{FL}
A.M.~Fu, A.~Lascoux,
\emph{Non-symmetric Cauchy kernels for the classical groups},
Journal of Combinatorial Theory, Series A 116 (2009), 903--917.

\bibitem[Fu]{Fu}
W.~Fulton, \emph{Young Tableaux: With Applications to Representation Theory and Geometry}, 
London Math. Society Student Texts, Cambridge University Press, 1997.

\bibitem[HLMW]{HLMW}
J.~Haglund, K.~Luoto, S.~Mason, S.~van~Willigenburg,
\emph{Refinements of the Littlewood-Richardson rule},
Trans. Amer. Math. Soc., 363(3):1665--1686, 2011.

\bibitem[Ho]{Ho}
R.~Howe,
\emph{Perspectives on invariant theory: Schur duality, multiplicity-free actions and beyond.} The Schur lectures (1992)(Tel Aviv). 1995:1-82.


\bibitem[J]{J} A.~Joseph, {\it On the Demazure character formula}, Annales Scientifique de l'E.N.S., 1985, 389--419.

\bibitem[vdK]{vdK}
W. van der Kallen,
{\em Longest weight vectors and excellent filtrations.}
Math. Z. 201 (1989), no. 1, 19-31. 


\bibitem[Kh]{Kh}
A.~Khoroshkin, {\em Highest weight categories and Macdonald polynomials}, arXiv:1312.7053.

\bibitem[KhM]{KhM}
A.~Khoroshkin, Ie.~Makedonskyi, {\em Bubble sort and Howe duality for staircase matrices}, arXiv:2502.21184. 

\bibitem[Las]{Las} A.~Lascoux, {\it Double crystal graphs}, In Studies in Memory of Issai Schur. Birkh\"auser Boston, 2003.

\bibitem[L]{Lovasz}
L.\,Lov\'asz,  {\emph Combinatorial problems and exercises},  Vol. 361. American Mathematical Soc. 2007.


\bibitem[LS]{LS}
A.~Lascoux, M.-P.~Sch\"utzenberger,
\emph{Keys \& standard bases}, in Invariant theory
and tableaux (Minneapolis, MN, 1988), vol. 19 of IMA Vol. Math. Appl. Springer, New York,
1990, pp. 125--144.


\bibitem[Mas]{Mas}
S.K.~Mason, \emph{An explicit construction of type A Demazure atoms},  Journal of Algebraic Combinatorics, 29(3):295--313, 2009.


\bibitem[Ok]{Ok} A. Okounkov, {\it Infinite wedge and random partitions}, Selecta Math. (N.S.), 7(1):57--81, 2001.

\bibitem[OR]{OR} A.~Okounkov, N.~Reshetikhin, {\it Correlation function of Schur process with application to local geometry
of a random 3-dimensional Young diagram}, J. Amer. Math. Soc., 16(3):581--603 (electronic), 2003.

\bibitem[P]{P}
A.~Pun, {\it On decomposition of the product of Demazure atoms and Demazure characters}, University of Pennsylvania, 2016. 

\bibitem[Po]{Po}
P.~Polo,{\it Varieties de Schubert et excellentes filtrations}, p.281-311 in: Orbites unipotentes et representations III, Asterisque 173-174, Soc. Math. France, 1989

\bibitem[St]{St}
R.~ Stanley, {\it Enumerative Combinatorics: Volume 2}, Cambridge University Press, First edition, 2001.

\end{thebibliography}
\end{document}